\documentclass{mod}
\usepackage{url}
\usepackage{setspace}
\usepackage{scrextend}
\usepackage{cite}
\usepackage{xcolor}

\usepackage{graphicx}

\usepackage{hyperref}

\usepackage{caption}
\usepackage{subcaption}
\newcounter{Chapcounter}
\renewcommand{\theChapcounter}{Part \Roman{Chapcounter}} 

\newcommand{\chapter}[1]
{ \vspace{2em} {\centering
		\addtocounter{Chapcounter}{1} \huge {\textbf{\underline{\textit{\color{black} \theChapcounter: ~#1}}}} }
	\addcontentsline{toc}{section}
	{ \hspace{-3em}\underline{\textit{ \color{blue} 
				~\theChapcounter.~~ #1}}\vspace{0.2em}}
}

\hypersetup{colorlinks=true}
\hypersetup{linkcolor=black}

\usepackage[shortlabels]{enumitem}
\usepackage{amsmath}
\usepackage{amsthm}
\usepackage{amssymb}
\usepackage{gensymb}

\usepackage{titletoc}
\usepackage{mathtools}
\usepackage{mathrsfs}

\usepackage[all]{xy}
\usepackage[toc,page]{appendix}
\usepackage{titlesec}
\usepackage{upgreek}

\usepackage{ stmaryrd }

\numberwithin{equation}{section}
\renewcommand{\theequation}{\thesection.\arabic{equation}}

\makeatletter
\def\thm@space@setup{\thm@preskip=5pt
	\thm@postskip=5pt}
\makeatother
\newtheoremstyle{mystyle}      
{} 
{} 
{\itshape} 
{} 
{\bfseries} 
{.} 
{ } 
{} 

\theoremstyle{mystyle}
\newcounter{mainthm}

\newtheorem{thm}{Theorem}[section]

\newtheorem{lem}[thm]{Lemma}
\newtheorem{prop}[thm]{Proposition}

\newtheorem{defn-thm}[thm]{Definition-Theorem}

\newtheorem{defn-lem}[thm]{Definition-Lemma}

\makeatletter
\def\thm@space@setup{\thm@preskip=5pt
	\thm@postskip=5pt}
\makeatother
\newtheoremstyle{mydefinition}      
{} 
{} 
{} 
{} 
{\itshape\bfseries} 
{.} 
{ } 
{} 

\theoremstyle{mydefinition}

\newtheorem{ex}[thm]{Example}

\newtheorem{convention}[thm]{Convention}
\newtheorem{construction}[thm]{Construction}

\newtheoremstyle{rmk}
{5pt}
{5pt}
{}
{}
{\itshape}
{.}
{.5em}
{}

\theoremstyle{rmk}
\newtheorem{axiom*}[thm]{Axiom}

\newtheoremstyle{def}
{5pt}
{5pt}
{}
{}
{\bfseries}
{.}
{.5em}
{}
\theoremstyle{def}

\newtheorem{defn}[thm]{Definition}
\newtheorem{defn-prop}[thm]{Definition-Proposition}
\newtheorem{rmk}[thm]{Remark}

\newtheoremstyle{note}
{8pt}
{5pt}
{\itshape}
{10pt}
{\bfseries}
{}
{.5em}
{}

\theoremstyle{note}

\makeatletter
\renewenvironment{proof}[1][\proofname]{\par
	\vspace{-\topsep}
	\pushQED{\qed}%
	\normalfont
	\topsep0pt \partopsep0pt 
	\trivlist
	\item[\hskip\labelsep
	\itshape
	#1\@addpunct{.}]\ignorespaces
}{%
	\popQED\endtrivlist\@endpefalse
	\addvspace{6pt plus 6pt} 
}
\makeatother

\newcommand{\pr}{\mathrm{pr}}               
\newcommand{\pair}[2]{\langle #1,\,#2\rangle}

\newcommand{\id}{\mathrm{id}}

\newcommand{\dd}{\mathsf d}
\newcommand{\ev}{\mathrm{ev}}

\newcommand{\Cint}{C^{\mathrm{int}}}
\newcommand{\Cext}{C^{\mathrm{ext}}}
\newcommand{\End}{\mathrm{End}}

\newcommand{\m}{\mathfrak m}
\newcommand{\n}{\mathfrak n}

\DeclareMathOperator{\Hom}{Hom}

\renewcommand{\thefootnote}{\fnsymbol{footnote}}

\usepackage{perpage}
\MakePerPage{footnote}

\makeatletter
\renewcommand\@makefntext[1]{ %
	\parindent 1em%
	\noindent
	\llap{\thefootnote\enspace}%
	#1
}
\makeatother

\titlecontents{subsubsection}
[0em]
{\footnotesize \vspace{1pt}}%
{\qquad \qquad \thecontentslabel.\enspace}
{}
{\titlerule*[0.5pc]{.}\contentspage}%

\titlecontents{section}
[0.72em]
{\scriptsize\bfseries\vspace{3pt}}%
{\thecontentslabel.\enspace}
{}
{\titlerule*[0.38pc]{.}\contentspage}%

\titlecontents{subsection}
[0em]
{\scriptsize \vspace{0pt}}%
{\qquad \quad \thecontentslabel.\enspace}
{}
{\titlerule*[0.5pc]{.}\contentspage}%

\titleformat{\subsection}{
	\itshape \bfseries \normalsize}{[\thesubsection] \ }{0em}{}[\vspace{0.3em}]
\titlespacing{\subsection}
{0pt}
{2.25ex plus 1ex minus .2ex}
{0pt}

\titleformat{\paragraph}[runin]{
	\bfseries \normalsize}{[\theparagraph] \ }{0em}{}[]
\titlespacing{\paragraph}
{5pt}
{4pt}
{5pt}

\begin{document}
	\setlength{\parindent}{15pt}	\setlength{\parskip}{0em}
	
	\title[{Higher operad structure for Fukaya categories}]{ {Higher operad structure for Fukaya categories}
	}
	\author[Hang Yuan]{Hang Yuan}
	\address{Beijing Institute of Mathematical Sciencens and Applications (BIMSA), Beijing, 101408, China; E-mail: yuanhang@bimsa.cn}
	
	\begin{abstract} {\sc Abstract:}  
		Operads often arise from geometry. The standard $A_\infty$ operad can be derived from the cellular chains on the Stasheff associahedra, and an $A_\infty$ algebra is an algebra over this operad.
		The notion of an $\mathbf{fc}$-multicategory, also called a virtual double category, is a two-dimensional generalization of operads and multicategories. Here $\mathbf{fc}$ stands for the free category monad.
		
We establish a natural $\mathbf{fc}$-multicategory structure on the collection of moduli spaces of \textit{pseudo-holomorphic polygons} with boundary on sequences of Lagrangian submanifolds in a symplectic manifold.
		These moduli spaces are known to underlie the construction of Fukaya categories.
Based on this, we develop the theory of differential graded (dg) variants of $\mathbf{fc}$-multicategories and show that a broad range of $A_\infty$-type structures, such as $A_\infty$ algebras, $A_\infty$ (bi)modules, and $A_\infty$ categories (possibly curved), admit a uniform operadic formulation as algebras over dg $\mathbf{fc}$-multicategories.
	\end{abstract}

	\maketitle
	%
	%
	
	\hypersetup{
		colorlinks=true,
		linktoc=all,
		linkcolor=blue,
		citecolor=blue
	}
	
	\tableofcontents
	
	\hypersetup{
		colorlinks=true,
		linktoc=all,
		linkcolor=brown,
		citecolor=blue
	}

	%
	%

	\section{Introduction}
	A (non-symmetric) operad is a multicategory with one object.
The idea of multicategories goes back to Lambek’s work in the 1960s \cite{lambek2006deductive}, and was further developed by Boardman-Vogt \cite{boardman1973homotopy} and May \cite{may2006geometry} and many others. Since then, operads and multicategories have served as a useful language for encoding algebraic structures with many inputs and one output.
A more general perspective, due to Burroni \cite{burroni1971t} and developed further by many others, views multicategories as arising from monads. For a cartesian monad $T$ on a category, one can define $T$-multicategories, with classical multicategories and operads appearing as special cases. In this paper, we focus on the case $T=\mathbf{fc}$ (\S~\ref{s_fc_directedgraphs}) and study the notion of a \textbf{\textit{$\mathbf{fc}$-multicategory}} as introduced by Leinster in \cite{leinster1999fc,leinster1998general,leinster2004higher}. It is also called a \textit{\textbf{virtual double category}} as introduced by Cruttwell and Shulman in \cite{cruttwell2009unified}; see also \cite{koudenburg2019augmented,nasu2025logical}.

The discovery of $A_\infty$ algebras by Stasheff in the 1960s \cite{stasheff1963homotopy} marked a pivotal moment in homotopical algebra and motivated the use of operads in topology and mathematical physics. Indeed, the Stasheff associahedra $K = \{K_n\}$ form a cellular non-symmetric topological operad, that is, a sequence of topological spaces equipped with circle-$i$ partial composition operations $\circ_i: K_r \times K_s \to K_{r+s-1}$ given by cellular maps satisfying certain associativity conditions.
	Moreover,
	the dg operad of cellular chains of the associahedra is isomorphic to
	the standard $A_\infty$ dg operad $\mathcal A_\infty= \{\mathcal A_\infty(n)\}_{n\ge 2}$.
	It is freely generated by symbols $\mathbf m_n$, and its differential $\delta$ is decided on generators by
	\[
	\delta(\mathbf m_n) = \pm \sum_{r+s+t=n} \mathbf m_{r+1+t}\circ_{r+1} \mathbf m_s
	\]
	and extended by the operadic Leibniz rule \cite{markl1998homotopy}.
	It is known that \textit{an $A_\infty$ algebra} on a cochain complex $X$ can be viewed as an \textit{algebra over the dg operad $\mathcal A_\infty$}, or equivalently, a dg operad morphism 
	\begin{equation}
		\label{operad_A_inf_intro_eq}
		\alpha:\mathcal A_\infty \to \End (X)
	\end{equation}
	where $\End (X)$ is the standard endomorphism dg operad.
	
	Since Fukaya’s work \cite{Fukaya_1993}, $A_\infty$ structures have become standard in symplectic geometry. However, beyond $A_\infty$ algebras, variants such as $A_\infty$ bimodules \cite[\S 3.7]{FOOOBookOne} and $A_\infty$ categories \cite{Fukaya_1993,FuUnobstructed,SeidelBook} also appear frequently. 
	But, for these variants the operadic viewpoint in \eqref{operad_A_inf_intro_eq} is explored less often. Typically, one introduces a collection of multilinear operations (often with different input types) and then writes the required identities componentwise.
	For instance, an \textit{$A_\infty$ category} is defined as the data consisting of a set of objects, graded vector spaces $\hom(v_0,v_1)$ for pairs $(v_0,v_1)$ of objects, and higher compositions
	\[
	\mu_d: \hom(v_{d-1}, v_d)\otimes\cdots\otimes \hom(v_0, v_1)\longrightarrow \hom(v_0,v_d) \]
	satisfying certain $A_\infty$ associativity relations.
	While effective in practice, this may look like a collection of ad hoc formulas, more or less obscuring both the conceptual uniformity suggested by \eqref{operad_A_inf_intro_eq} and the geometric origin of the operations. In symplectic geometry, these structure maps are defined by (virtual) counts of pseudo-holomorphic polygons, whereas the purely algebraic package above can discard useful geometric data. For example, it does not retain the topological class of the boundary loop of these polygons, which is often useful in applications, such as in formulating the (boundary) divisor axiom for Fukaya's $A_\infty$ algebras \cite{FuCyclic,Yuan_I_FamilyFloer,Yuan_unobs}.
	
Now, there are two natural questions:
	
	\begin{enumerate}[(I)]
	
		\item \textit{Given the operad structure of Stasheff’s associahedra, can we likewise extract ``operad-like’’ structures from the moduli spaces of pseudo-holomorphic disks and polygons?}
		
		\item \textit{Given that an $A_\infty$ algebra can be viewed as an algebra over the standard dg $A_\infty$ operad as in (\ref{operad_A_inf_intro_eq}), can we formulate the various other $A_\infty$-type structures operadically in the same spirit?}
	\end{enumerate}

A satisfactory understanding of Question (I) may shed light on a systematic route to Question (II).

\subsection*{Geometric motivations}
Let's first address Question (I). 
In symplectic geometry, moduli spaces of pseudo-holomorphic curves with Lagrangian boundary conditions (single or multiple, embedded or immersed) should give rise to a rich family of $A_\infty$-type structures, suggesting that operads alone may not be flexible enough.
	
As a starting point, we consider the moduli space of pseudo-holomorphic disks with boundary on a single embedded Lagrangian submanifold (Figure \ref{figure_moduli}). The next statement records the key topological features of this moduli space and serves as our entry point for explaining how $\mathbf{fc}$-multicategories emerge from more general moduli spaces.
The following statement appears to be implicit in the standard literature \cite{FOOOBookOne,FOOO_Kuranishi}, while it is not formulated in the language of operads. Fukaya has pursued an operadic perspective on the differential-geometric theory of Kuranishi structures on moduli spaces \cite{Fukaya_operad}.

	\begin{prop} 
		\label{moduli_multicategory_prop_intro}
		Let $L$ be an embedded closed Lagrangian submanifold in a closed symplectic manifold $(X,\omega)$.
		Then, the moduli space of pseudo-holomorphic disks bounded by $L$ forms a non-symmetric topological multicategory (colored operad), with set of objects (colors) equal to $L$.
	\end{prop}

	\begin{figure}[h]
		\centering
		\includegraphics[width=10cm]{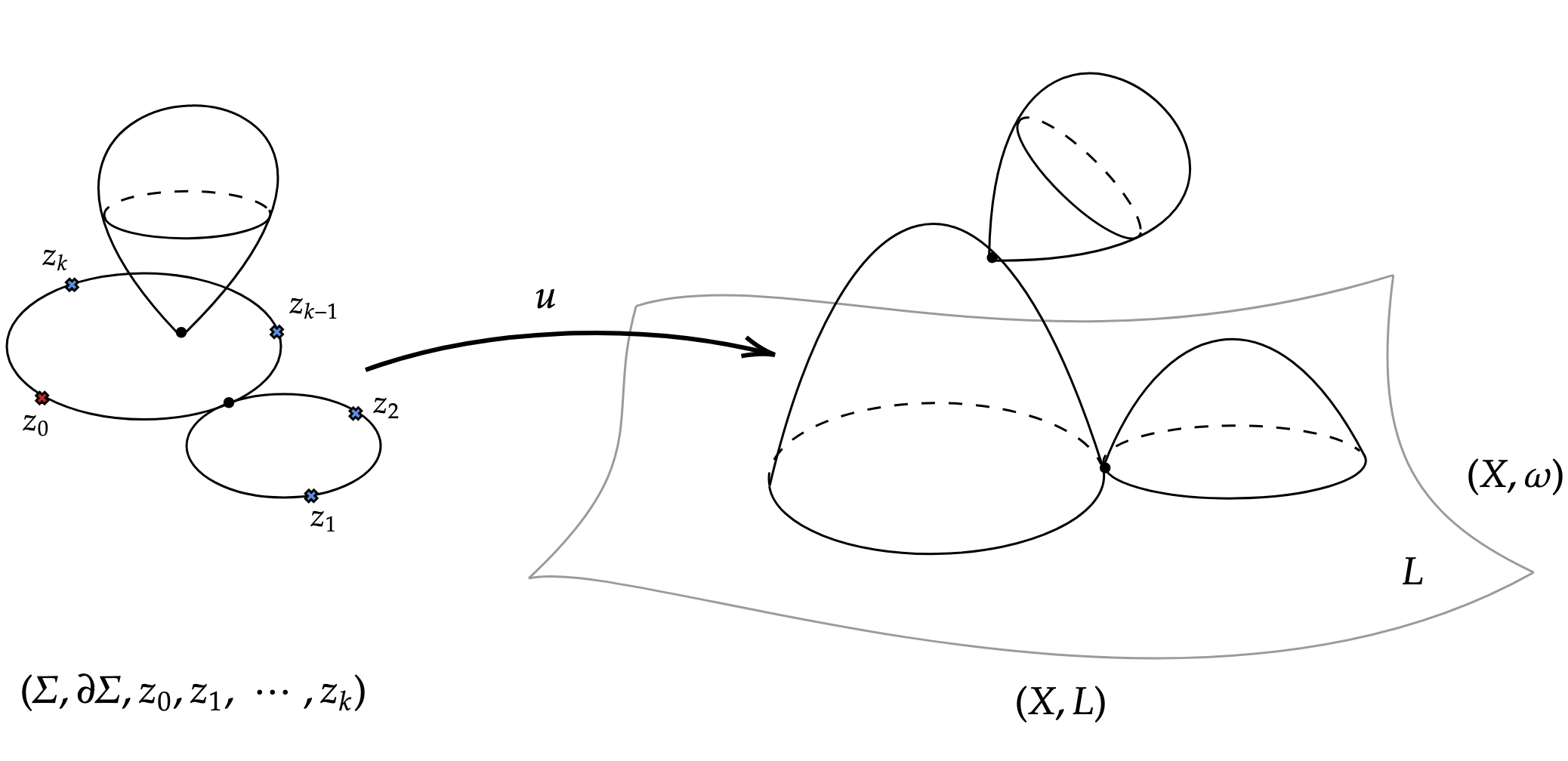}
		\caption{Moduli space of pseudo-holomorphic (stable) disks bounded by a single embedded Lagrangian}
		\label{figure_moduli}
	\end{figure}

It is not necessary for the reader to have extensive background in symplectic geometry, and one can follow the basic topological intuition conveyed by Figure \ref{figure_moduli}. One of the main aims of this paper is to use the geometry as motivation for the algebraic results developed later.

For the convenience of readers, we briefly recall the symplectic background.
	Fix $k\ge 0$ and $\beta\in H_2(X,L)$. 
	We consider a tuple
	$(\Sigma,\partial\Sigma;\, z_0,\dots,z_k;\, u)$ as in Figure~\ref{figure_moduli},	where $(\Sigma,\partial\Sigma)$ is an oriented nodal bordered Riemann surface, $z_0,\dots,z_k\in \partial\Sigma$ are $k+1$ boundary marked points ordered cyclically for the induced boundary orientation, 
	$u:(\Sigma,\partial\Sigma)\to (X,L)$ is a continuous map whose restriction to the smooth locus of $\Sigma$ solves certain Cauchy-Riemann equation, and these data is also required to satisfy the so-called stable condition. Such a map is usually called a pseudo-holomorphic curve or stable map in the symplectic literature; see \cite{FOOOBookOne,FOOODiskOne} for more details.

	Let $\mathcal M(k)$ denote the set of isomorphism classes of these tuples
	$
	(\Sigma,\partial\Sigma;\, z_0,\dots,z_k;\, u)$ which we refer to as the \textbf{\textit{moduli space}} of pseudo-holomorphic stable maps. Fukaya-Oh-Ohta-Ono use these moduli spaces, together with their analytic theory of Kuranishi structures, to construct an $A_\infty$ algebra structure on the de Rham complex $\Omega^*(L)$; see \cite{FOOOBookOne,FOOO_Kuranishi,FOOODiskOne,FOOODiskTwo}.

	\textit{At the level of topology}, the structure of the moduli space $\mathcal M(k)$ is well understood. First, we know that $\mathcal M(k)$ is a compact Hausdorff space \cite[Theorem 2.1.29]{FOOOBookOne}. Besides, there are natural continuous evaluation maps
	\begin{equation}
		\label{ev_i_embedded_intro}
		\ev_i: \mathcal M(k) \to L \ , \quad i=0,1,\dots, k
	\end{equation}
	defined by sending the isomorphism class of tuple $(\Sigma,\partial\Sigma; z_0,\dots, z_k; u)$ to the point $u(z_i)\in L$.
The following diagram illustrates the structure of the multicategory (colored operad) in Proposition \ref{moduli_multicategory_prop_intro}: 	
\[
\xymatrix{
	& \mathcal M(k) \ar[dr]^{\ev_0} \ar[dl]_{(\ev_1,\dots, \ev_k)}& \\
L^{\times k} & & L
}
\]
Here \(\ev_0\) serves as the target map, and \((\ev_1,\dots,\ev_k)\) serves as the source map for the multicategory structure.


Following the above discussion, our basic observation is that if one allows several Lagrangians (rather than a single one), or allows the Lagrangian to be immersed, one is naturally led to certain ``higher'' operadic objects, which may provide a more natural formalism for encoding the resulting $A_\infty$-type algebraic structures than the $A_\infty$ operad alone.
	Specifically, we can achieve the following:
	
	\begin{thm}[Theorem \ref{fc_moduli_unlabeled_thm}]
		\label{fc_moduli_introduction_thm}
	Let $(X,\omega)$ be a closed symplectic manifold. Let $\iota: L \to X$ be a Lagrangian immersion, i.e. $\iota^*\omega=0$.
		Then, the collection $\mathscr M_\iota$ of moduli spaces of pseudo-holomorphic polygons bounded by $\iota(L)$ naturally forms a topological $\mathbf{fc}$-multicategory.
	\end{thm}

An \textit{$\mathbf{fc}$-multicategory} \cite{leinster1999fc,leinster2004higher} can be regarded as an operad-like structure in which operations are indexed not by rooted trees, but by two-dimensional pasting diagrams. It consists of 0-cells, 1-cells, and 2-cells subject to suitable matching conditions; see Section~\ref{s_fc_multicategory} for more details. In Figure~\ref{figure:fc_from_2_cell_to_pseudo_holo}, the 0-cells are the $v_i$, the (horizontal) 1-cells are the edges $e_i$, and the 2-cell is $\mathbf u$. The key observation is that the composition of 2-cells appears to be compatible with the gluing of pseudo-holomorphic polygons in symplectic geometry (cf. Figure~\ref{figure:fc_from_2_cell_to_pseudo_holo}).

Let us also explain the relevant notions for moduli spaces. 
	As in the embedded case above but with some essential modification for the immersed case, we consider tuples
	$(\Sigma,\partial\Sigma; z_0,\dots,z_k; u,\gamma)$
	where $u:(\Sigma,\partial\Sigma)\to (X,\iota(L))$ is as before and $\gamma: \partial\Sigma\setminus\{z_0,\dots, z_k\} \to L$ is the extra data for a continuous lift of $u$ with $\iota \circ\gamma=u$.
Unlike \eqref{ev_i_embedded_intro}, in this immersed setting, the evaluation maps
	\[
	\ev_i:\mathcal M(k)\to L\times_X L \ ,  \quad i=0,1,\dots,k.
	\]
	take values in the fiber product $L\times_X L =L\times_\iota L= \{(p,q)\in L\times L \mid \iota(p)=\iota(q)\}$ and is defined by sending the isomorphism class of a tuple $(\Sigma,\partial\Sigma; z_0,\dots,z_k; u,\gamma)$ to the point $(\gamma(z_i-),\gamma(z_i+))$
	where $\gamma(z_i\pm)$ are the ``one-sided limits'' taken along $\partial\Sigma$ with respect to the induced boundary orientation; see \cite[Definition 3.17]{FuUnobstructed} for more details.
	In particular, the oriented boundary arc of $\partial\Sigma$ from $z_i$ to $z_{i+1}$ is mapped by $\gamma$ to a path $\gamma_i\subset L$, which naturally suggests keeping track of the (path-)connected components $L_v$ of $L$ with $L=\bigsqcup_v L_v$. The fiber product $L\times_X L$ then inherits a corresponding decomposition, and hence each evaluation map
	$\ev_i$
	induces a further decomposition of the moduli space $\mathcal M(k)$. This finer decomposition suggests that the aforementioned multicategory structure may admit a corresponding refinement.

\begin{rmk}
A typical example of $\iota$ is as follows. Let $\{L_v\}_{v\in V}$ be a collection of connected embedded Lagrangian submanifolds which intersect pairwise transversely. Set
	$
	L:=\bigsqcup_{v\in V} L_v$, and the natural inclusion defines a Lagrangian immersion $\iota:L\to X$.
Then, in Theorem~\ref{fc_moduli_unlabeled_thm} the $0$-cells are given by the index set $V$, the $1$-cells are intersection points in $L_v\cap L_{v'}$, and the $2$-cells are pseudo-holomorphic polygons.
\end{rmk}

\begin{rmk}
A concrete situation in symplectic geometry where it may be useful to introduce $\mathbf{fc}$-multicategory structures on moduli spaces is the following. We wish to distinguish (i) Lagrangian Floer theory for a pair $(L_0,L_1)$ from (ii) the Fukaya category with object set ${L_0,L_1}$. If we denote the corresponding collections of moduli spaces by $\mathscr M_1$ and $\mathscr M_2$, then $\mathscr M_1\subset \mathscr M_2$ and actually $\mathscr M_1$ forms a full $\mathbf{fc}$-submulticategory of $\mathscr M_2$; see Section \ref{s_Gromov_factor_closed} for relevant discussion.
\end{rmk}

	\begin{figure}[h]
		\centering
		\begin{tikzpicture}[x=0.75pt,y=0.75pt,yscale=-1,xscale=1]
			
			\draw [color={rgb, 255:red, 155; green, 155; blue, 155 }  ,draw opacity=1 ]   (527.72,226.76) -- (161.9,225.84) ;
			\draw [shift={(159.9,225.84)}, rotate = 0.14] [color={rgb, 255:red, 155; green, 155; blue, 155 }  ,draw opacity=1 ][line width=0.75]    (10.93,-3.29) .. controls (6.95,-1.4) and (3.31,-0.3) .. (0,0) .. controls (3.31,0.3) and (6.95,1.4) .. (10.93,3.29)   ;
			\draw [color={rgb, 255:red, 155; green, 155; blue, 155 }  ,draw opacity=1 ]   (458.62,114.1) -- (287.4,74.63) ;
			\draw [shift={(285.45,74.19)}, rotate = 12.98] [color={rgb, 255:red, 155; green, 155; blue, 155 }  ,draw opacity=1 ][line width=0.75]    (10.93,-3.29) .. controls (6.95,-1.4) and (3.31,-0.3) .. (0,0) .. controls (3.31,0.3) and (6.95,1.4) .. (10.93,3.29)   ;
			\draw [color={rgb, 255:red, 155; green, 155; blue, 155 }  ,draw opacity=1 ]   (254.95,77.99) -- (203.63,101.85) ;
			\draw [shift={(201.82,102.7)}, rotate = 335.06] [color={rgb, 255:red, 155; green, 155; blue, 155 }  ,draw opacity=1 ][line width=0.75]    (10.93,-3.29) .. controls (6.95,-1.4) and (3.31,-0.3) .. (0,0) .. controls (3.31,0.3) and (6.95,1.4) .. (10.93,3.29)   ;
			\draw [color={rgb, 255:red, 155; green, 155; blue, 155 }  ,draw opacity=1 ]   (194.93,107.45) -- (146.59,210.19) ;
			\draw [shift={(145.74,212)}, rotate = 295.2] [color={rgb, 255:red, 155; green, 155; blue, 155 }  ,draw opacity=1 ][line width=0.75]    (10.93,-3.29) .. controls (6.95,-1.4) and (3.31,-0.3) .. (0,0) .. controls (3.31,0.3) and (6.95,1.4) .. (10.93,3.29)   ;
			\draw [color={rgb, 255:red, 155; green, 155; blue, 155 }  ,draw opacity=1 ]   (538.32,211.05) -- (488.24,133.84) ;
			\draw [shift={(487.15,132.16)}, rotate = 57.03] [color={rgb, 255:red, 155; green, 155; blue, 155 }  ,draw opacity=1 ][line width=0.75]    (10.93,-3.29) .. controls (6.95,-1.4) and (3.31,-0.3) .. (0,0) .. controls (3.31,0.3) and (6.95,1.4) .. (10.93,3.29)   ;
			\draw [color={rgb, 255:red, 74; green, 144; blue, 226 }  ,draw opacity=1 ][line width=1.5]    (209.26,46.77) .. controls (270.64,143.92) and (368.55,95.73) .. (407.41,41.33) ;
			\draw [color={rgb, 255:red, 74; green, 144; blue, 226 }  ,draw opacity=1 ][line width=1.5]    (407.41,41.33) .. controls (415.95,138.48) and (489,185.89) .. (608.67,179.67) ;
			\draw [color={rgb, 255:red, 74; green, 144; blue, 226 }  ,draw opacity=1 ][line width=1.5]    (339.03,322.67) .. controls (333.59,261.27) and (482.78,199.1) .. (608.67,179.67) ;
			\draw [color={rgb, 255:red, 74; green, 144; blue, 226 }  ,draw opacity=1 ][line width=1.5]    (125.33,182.78) .. controls (251.22,198.32) and (316.49,246.5) .. (339.03,322.67) ;
			\draw [color={rgb, 255:red, 74; green, 144; blue, 226 }  ,draw opacity=1 ][line width=1.5]    (125.33,182.78) .. controls (231.01,191.33) and (241.12,120.6) .. (209.26,46.77) ;
			
			\draw (135.04,213.75) node [anchor=north west][inner sep=0.75pt]  [color={rgb, 255:red, 155; green, 155; blue, 155 }  ,opacity=1 ]  {$v_{n}$};
			\draw (536.43,215.21) node [anchor=north west][inner sep=0.75pt]  [color={rgb, 255:red, 155; green, 155; blue, 155 }  ,opacity=1 ]  {$v_{0}$};
			\draw (262.99,63.27) node [anchor=north west][inner sep=0.75pt]  [color={rgb, 255:red, 155; green, 155; blue, 155 }  ,opacity=1 ]  {$v_{i}$};
			\draw (360,100) node [anchor=north west][inner sep=0.75pt]  [color={rgb, 255:red, 155; green, 155; blue, 155 }  ,opacity=1 ]  {$e_{i}$};
			\draw (465.39,107.16) node [anchor=north west][inner sep=0.75pt]  [color={rgb, 255:red, 155; green, 155; blue, 155 }  ,opacity=1 ]  {$v_{i-1}$};
			\draw (331.83,231.53) node [anchor=north west][inner sep=0.75pt]  [color={rgb, 255:red, 155; green, 155; blue, 155 }  ,opacity=1 ]  {$e_{0}$};
			\draw (331.21,153.74) node [anchor=north west][inner sep=0.75pt]  [color={rgb, 255:red, 74; green, 144; blue, 226 }  ,opacity=1 ]  {$\mathbf{u}$};
			\draw (408.9,245.98) node [anchor=north west][inner sep=0.75pt]  [color={rgb, 255:red, 74; green, 144; blue, 226 }  ,opacity=1 ]  {$\gamma _{0}$};
			\draw (260.56,243.65) node [anchor=north west][inner sep=0.75pt]  [color={rgb, 255:red, 74; green, 144; blue, 226 }  ,opacity=1 ]  {$\gamma _{n}$};
			\draw (339.57,57.91) node [anchor=north west][inner sep=0.75pt]  [color={rgb, 255:red, 74; green, 144; blue, 226 }  ,opacity=1 ]  {$\gamma _{i}$};
			\draw (428.46,69.57) node [anchor=north west][inner sep=0.75pt]  [color={rgb, 255:red, 74; green, 144; blue, 226 }  ,opacity=1 ]  {$\gamma _{i-1}$};
		\end{tikzpicture}
		\caption{In Theorem \ref{fc_moduli_introduction_thm}, a pseudo-holomorphic polygon with boundary paths $\gamma_i$ in the Lagrangian component $L_{v_i}$ (blue) is viewed as a \(2\)-cell diagram in the corresponding \(\mathbf{fc}\)-multicategory (gray), and an intersection point in $L_{v_i}\cap L_{v_{i+1}}$ is viewed as a horizontal 1-cell.}
		\label{figure:fc_from_2_cell_to_pseudo_holo}
	\end{figure}
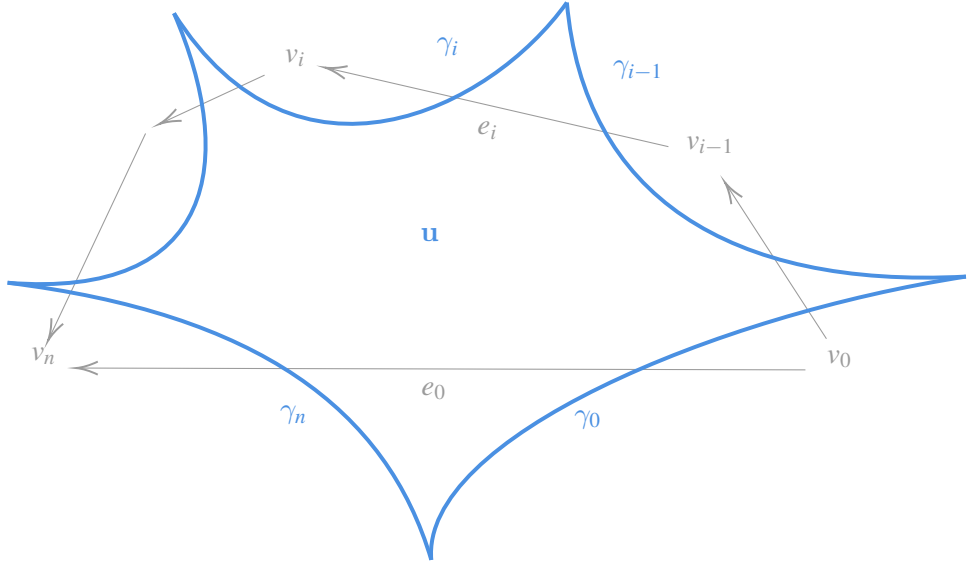

\subsection*{Algebraic implications}

Now, let's handle Question (II) at the start of introduction.
Following the ideas of Leinster in \cite{leinster2004higher}, we will develop the theory of \textit{differential graded (dg) $\mathbf{fc}$-multicategories} and \textit{algebras over dg $\mathbf{fc}$-multicategories} (see \S \ref{s_A_inf_as_algebra_for_dg_fc}).
Based on these notions, our main observations are the following:

\begin{thm}
\label{dg_fc_thm_intro}
There exist dg $\mathbf{fc}$-multicategories whose algebras recover the notions of $A_\infty$ algebras, $A_\infty$ categories, $A_\infty$ left/right modules, and $A_\infty$ bimodules.
\end{thm}

The basic idea behind this theorem is as follows. One may regard the standard dg operad $\mathcal A_\infty$ as arising from Stasheff’s associahedra by ``shrinking each cell to a point.'' Each codimension-one face of $K_n$ is of the form $K_{r+1+t}\times K_s$, corresponding to grafting. Thus the cellular boundary of the top cell gives rise to the differential $\delta$ on $\mathcal A_\infty$, endowing it with the structure of a dg operad. The $A_\infty$ relations are precisely the operadic expression of the identity $\partial^2=0$ for the cellular chains of the associahedra.

In view of \eqref{operad_A_inf_intro_eq}, an $A_\infty$ algebra is equivalently a morphism
$\alpha:\mathcal A_\infty \to \End(X)$
of dg operads.
Motivated by this, one may ask whether an analogous procedure of ``shrinking each cell to a point'' can be applied to the moduli spaces in $\mathscr M_\iota$ appearing in Theorem~\ref{fc_moduli_introduction_thm}.

Since the 0-cells and (horizontal) 1-cells of an $\mathbf{fc}$-multicategory form a directed graph, one is naturally led to try to construct an $\mathbf{fc}$-multicategory in which, for each prescribed boundary profile of 1-cells, the corresponding space of 2-cells is a singleton. In this way, one expects to obtain generalized versions of Stasheff’s dg operad $\mathcal A_\infty$. These are essentially the dg $\mathbf{fc}$-multicategories mentioned in Theorem~\ref{dg_fc_thm_intro}, whose algebras are $A_\infty$ categories.
Moreover, by restricting to suitable subcollections of $\mathscr M_\iota$, one can also construct dg $\mathbf{fc}$-submulticategories. These include, for example, dg $\mathbf{fc}$-multicategories whose algebras are $A_\infty$ bimodules. More exotic $A_\infty$-type structures also fit naturally into this framework, such as $A_\infty$ tri-modules and, more generally, $A_\infty$ $k$-modules for arbitrary $k>0$. We refer to Section~\ref{s_Ainf_type} for further examples and discussion.


The paper is organized as follows. In Section~\ref{s_operads_multicategories}, we review operads and multicategories, and describe the multicategory structure on moduli spaces of pseudo-holomorphic disks bounded by a single Lagrangian in Proposition~\ref{moduli_multicategory_prop_intro}. In Section \ref{s_fc_multicategory}, we review $\mathbf{fc}$-multicategories and unpacks their structure into explicit data. In Section \ref{s_moduli_fc_multicategories}, we discuss Theorem \ref{fc_moduli_introduction_thm} and briefly indicate how $\mathbf{fc}$-multicategories enter symplectic geometry. In Section \ref{s_A_inf_as_algebra_for_dg_fc}, we review algebras over $\mathbf{fc}$-multicategories, proposes dg $\mathbf{fc}$-multicategories, and introduces algebras over dg $\mathbf{fc}$-multicategories. Finally, in Section~\ref{s_Ainf_type} we address Theorem \ref{dg_fc_thm_intro} and explain how to recover various classical $A_\infty$-type structures as algebras over dg $\mathbf{fc}$-multicategories.

\subsection*{Acknowledgment}
The author would like to thank T. Leinster for helpful and thoughtful email correspondence.

	\section{Operads and multicategories}
	\label{s_operads_multicategories}
	
	In this section, we recall the general abstract frameworks for operads and multicategories, which provide a concise and clean formulation once one is familiar with the abstract language. Our main reference is Leinster's book \cite{leinster2004higher}; see also \cite{leinster1998general,leinster1999fc,hermida2000representable,loday2012algebraic,markl2002operads}. For a first reading, one may consult Definition \ref{S_labeled_multicategory_defn} directly, where we adopt the more classical and explicit description.

	\subsection{Monads}
	\label{s_monad}

	A {\textit{terminal object}} in a category $\mathcal{E}$ is an object $\mathbf 1$ such that for every $X\in\mathcal{E}$ there exists a unique morphism $X\to\mathbf 1$. For example, in the category $\mathbf{Set}$ of sets, a one-point set $\{\ast\}$ is a terminal object.
	Given objects $X,Y\in\mathcal{E}$, a {\emph{binary product}} is an object
	$X\times Y$ equipped with projections
	$\pr_1:X\times Y\to X$ and $\pr_2:X\times Y\to Y$
	such that for every $Z$ and morphisms $f:Z\to X$, $g:Z\to Y$ there exists a unique
	$\pair{f}{g}:Z\to X\times Y$ with $\pr_1\circ\pair{f}{g}=f$ and
	$\pr_2\circ\pair{f}{g}=g$.

	Given morphisms $f:X\to S$ and $g:Y\to S$ in $\mathcal{E}$, a {\emph{pullback}} (or \textit{fiber product})
	is an object $X\times_S Y$ with morphisms $\pi_X:X\times_S Y\to X$ and
	$\pi_Y:X\times_S Y\to Y$ making the square commute $f\circ\pi_X=g\circ\pi_Y$ and satisfying the universal property:
	for any object $Z$ with maps $u:Z\to X$, $v:Z\to Y$ such that $f\circ u=g\circ v$,
	there exists a unique $h:Z\to X\times_S Y$ with $\pi_X\circ h=u$ and $\pi_Y\circ h=v$.
	\[
	\xymatrix{
		Z \ar[drr]^{v} \ar[ddr]_{u} \ar@{-->}[dr]|-{\exists ! h} & & \\
		& X\times_S Y \ar[r]^{\pi_Y} \ar[d]_{\pi_X} & Y \ar[d]^{g} \\
		& X \ar[r]_{f} & S
	}
	\]
	We say a category $\mathcal{E}$ \emph{has finite limits} if it has a terminal object,
	all binary products, and all pullbacks (equivalently: all limits of finite diagrams).
	The condition that $\mathcal E$ has all finite limits is equivalent to that $\mathcal E$ has a terminal object and all pullbacks.
	In particular, binary products can be constructed as pullbacks over the terminal object.
	
	A {\emph{cartesian category}} is a category $\mathcal{E}$ that has a terminal object $\mathbf 1$ and all pullbacks.
	A functor $F:\mathcal E\to\mathcal E$ is \emph{cartesian} if it preserves pullbacks, namely,
	for every pullback square in $\mathcal E$, its image under $F$ is again a pullback square.
	Let $F,G:\mathcal E\to\mathcal E$ be functors. A natural transformation
	$\alpha:F\Rightarrow G$ is \emph{cartesian} if for every morphism $f:X\to Y$ in $\mathcal E$,
	the naturality square
	\[
	\xymatrix{
		F(X) \ar[r]^{F(f)} \ar[d]_{\alpha_X} & F(Y) \ar[d]^{\alpha_Y} \\
		G(X) \ar[r]_{G(f)} & G(Y)
	}
	\]
	is a pullback square in $\mathcal E$.
	See \cite[Definition 4.1.1]{leinster2004higher}.

	\begin{ex}[Cartesian categories]
		\label{ex_cartesian_categories}
		The category $\mathbf{Set}$ of sets is cartesian: the terminal object is $1=\{*\}$ and pullbacks exist.
		The category $\mathbf{Top}$ of topological spaces is cartesian: the terminal object is a one-point space, and pullbacks are as usual.
		For any cartesian $\mathcal E$ and any object $B\in\mathcal E$, the slice category $\mathcal E{/}B$ (whose objects are maps $X\to B$ for $X\in\mathcal E$ and morphisms $f:(X\!\xrightarrow{p}\!B)\to(Y\!\xrightarrow{q}\!B)$ are maps $f:X\to Y$ in $\mathcal E$ with $q\circ f=p$.) is cartesian. The terminal object is $\mathrm{id}_B:B\to B$, and their pullbacks (or \emph{fiber product over $B$}) is
		$X\times_B Y$ fitting into the usual pullback square in $\mathcal E$.
		In particular, $\mathbf{Top}{/}B$ is cartesian with fiber products $X\times_B Y$.
	\end{ex}

	A \emph{monad} on $\mathcal E$ is a triple $(T,\eta,\mu)$ consisting of an endofunctor
	$T:\mathcal E\to\mathcal E$, a unit $\eta:\mathrm{Id}_{\mathcal E}\Rightarrow T$, and
	a multiplication $\mu:T\!\circ T\Rightarrow T$ such that the usual associativity and unit
	axioms hold:
	\[
	\mu\circ T\mu \;=\; \mu\circ \mu T, \qquad
	\mu\circ T\eta \;=\; \mathrm{id}_T \;=\; \mu\circ \eta T .
	\]
	Specifically, for every object $X$, we have
	\[
	\begin{gathered}
		\xymatrix@C=48pt{
			T^3X \ar[r]^{T\mu_X} \ar[d]_{\mu_{T X}} & T^2X \ar[d]^{\mu_X} \\
			T^2X \ar[r]_{\mu_X} & T X
		}
	\end{gathered}
	\]
	\[
	\begin{gathered}
		\xymatrix@C=48pt{
			T X \ar[r]^{T\eta_X} \ar[dr]_{\mathrm{id}_{T X}} & T^2 X \ar[d]^{\mu_X} \\
			& T X
		}
		\qquad
		\xymatrix@C=48pt{
			T X \ar[r]^{\eta_{T X}} \ar[dr]_{\mathrm{id}_{T X}} & T^2 X \ar[d]^{\mu_X} \\
			& T X
		}
	\end{gathered}
	\]
	A monad $(T,\eta,\mu)$ on $\mathcal E$ is \emph{cartesian}
	if (i) the functor $T$ is cartesian, and
	(ii) the natural transformations $\eta$ and $\mu$ are cartesian.
	Equivalently, $T$ preserves pullbacks and, for every morphism $f:X\to Y$, both
	naturality squares
	\[
	\xymatrix{
		X \ar[r]^{f} \ar[d]_{\eta_X} & Y \ar[d]^{\eta_Y} \\
		T X \ar[r]_{T f} & T Y
	}
	\qquad
	\xymatrix{
		T T X \ar[r]^{T T f} \ar[d]_{\mu_X} & T T Y \ar[d]^{\mu_Y} \\
		T X \ar[r]_{T f} & T Y
	}
	\]
	are pullbacks.

	For any cartesian $\mathcal E$, the identity monad $(\mathrm{Id},\eta=\mathrm{id},\mu=\mathrm{id})$ is cartesian.
	Besides, the \textit{free monoid monad} $(T,\eta,\mu)$ on the cartesian category $\mathcal E= \mathbf{Top}$ is defined as follows.
	Define an endofunctor $T:\mathcal E\to\mathcal E$ by
	\begin{equation}
		\label{list_monad_eq}
		T X \;:=\; \coprod_{n\ge 0} X^{\times n}
	\end{equation}
	with the convention $X^{\times 0}$ is the terminal object, the one-point space. On a morphism $f:X\to Y$, define
	\[
	T f \;=\; \coprod_{n\ge 0} f^{\times n}:\ \coprod_{n\ge 0} X^{\times n}\longrightarrow \coprod_{n\ge 0} Y^{\times n}.
	\]
	The unit $\eta:\mathrm{Id}\Rightarrow T$ and multiplication $\mu:T^2\Rightarrow T$ are:
	\begin{align*}
		\eta_X &: X \longrightarrow T X,\qquad x\longmapsto [x]\in X^{\times 1}\subseteq TX,\\
		\mu_X &: T(TX) \longrightarrow TX,
	\end{align*}
	where \(\mu_X\) flattens a list of lists by concatenation.
	Concretely, 
	\[
	T(TX)
	\;=\;
	\coprod_{k\ge 0} (TX)^{\times k}
	\;\cong\;
	\coprod_{k\ge 0}\ \coprod_{(n_1,\dots,n_k)\in\mathbb N^k}
	\ \prod_{i=1}^k X^{\times n_i}.
	\]
	An element of $T(TX)$ is thus a finite list of finite lists:
	\[
	\bigl([x_{1,1},\dots,x_{1,n_1}],\;[x_{2,1},\dots,x_{2,n_2}],\;\dots,\;[x_{k,1},\dots,x_{k,n_k}]\bigr).
	\]
	Here $k$ can be $0$, in which case we have the empty list; some $n_i$ may also be $0$. The multiplication $\mu_X:T(TX)\to TX$ sends a list of lists to their concatenation:
	\[
	\mu_X\Bigl(\,[x_{1,1},\dots,x_{1,n_1}],\dots,[x_{k,1},\dots,x_{k,n_k}]\,\Bigr)
	\ :=\
	[x_{1,1},\dots,x_{1,n_1},\;x_{2,1},\dots,x_{2,n_2},\;\dots,\;x_{k,1},\dots,x_{k,n_k}]
	\ \in\ X^{\times (n_1+\cdots+n_k)}\subseteq TX.
	\]
	Namely, on the summand
	\(
	\prod_{i=1}^k X^{\times n_i}
	\)
	we define
	\[
	\mu_X\big|_{\prod_{i=1}^k X^{\times n_i}}:
	\ \prod_{i=1}^k X^{\times n_i}\ \longrightarrow\ X^{\times (n_1+\cdots+n_k)}
	\]
	to be the canonical isomorphism
	that reindexes the tuple
	\(
	((x_{1,1},\dots,x_{1,n_1}),\dots,(x_{k,1},\dots,x_{k,n_k}))
	\)
	as a single $(n_1+\cdots+n_k)$-tuple.
	If $k=0$, then $\mu_X$ maps it to the empty list
	\([\,]\in X^{\times 0}\subseteq TX\).
	Finally, it is routine to check that $(T,\eta,\mu)$ is actually a cartesian monad.
	

	\smallskip 
	
	\smallskip
	
	Let $(T,\eta,\mu)$ be a monad on a category $\mathcal C$.
	An \textit{algebra over the monad $T$} (or simply a \textit{$T$-algebra}) is a pair $(A, \alpha)$ with an object $A\in\mathcal C$ and a structure morphism
	$
	\alpha:\ T A \longrightarrow A
	$
	such that the \emph{unit} and \emph{associativity} axioms hold: (see \cite[B.4.2]{loday2012algebraic} and \cite[p7]{leinster2004higher})
	\begin{equation}
		\label{algebra_for_monad_alpha_condition_eq}
		\xymatrix@C=32pt{
			A \ar[r]^{\eta_A} \ar[dr]_{\mathrm{id}_A} & TA \ar[d]^\alpha \\
			& A & 
		}
		\qquad
		\xymatrix@C=44pt{
			T^2A \ar[r]^{\mu_A} \ar[d]_{T\alpha} & TA \ar[d]^\alpha \\
			TA \ar[r]_\alpha & A
		}
	\end{equation}
	A \emph{morphism of $T$-algebras} $f:(A,a)\to(B,b)$ is a map $f:A\to B$ in $\mathcal C$ with $f\circ a= b\circ Tf$.

	When $T$ is the free monoid monad (\ref{list_monad_eq}), the morphism $\alpha: TA\to A$ is decomposed to $\alpha_n :A^{\times n} \to A$ for $n\ge 0$.
	Since $A^{\times 0}$ is the terminal object by our convention, $\alpha_0$ specifies an element of $A$, denoted by $e$. Define $x_1\ast x_2= \alpha_2( x_1, x_2)$, and $(A, \ast, e)$ is a monoid (a semigroup with identity).

	\subsection{Bicategory of spans}
	Recall that a \emph{bicategory} $\mathcal B$ consists of the following data (see \cite[Definition 1.5.1]{leinster2004higher}):
	\begin{itemize}
		\itemsep 2pt
		\item A class whose elements are called \emph{objects} or \textit{0-cells}.
		\item For each pair of objects $A,B$, a category $\mathcal B(A,B)$
		whose objects $f:A\to B$ are called \emph{1-cells}
		and whose morphisms are called \emph{2-cells} $\alpha:f\Rightarrow g$.
		\item For each triple of objects $A,B,C$, a functor
		\[
		\circ:\ \mathcal B(B,C)\times \mathcal B(A,B)\;\longrightarrow\;\mathcal B(A,C),
		\]
		called \emph{composition}.
		\item For each object $A$, a distinguished 1-cell
		$\mathrm{id}_A:A\to A$ called the \emph{identity}.
		\item Natural isomorphisms
		\[
		a_{f,g,h}:(h\circ g)\circ f \;\cong\; h\circ(g\circ f),\qquad
		\lambda_f:\mathrm{id}_B\circ f \;\cong\; f,\qquad
		\rho_f:f\circ \mathrm{id}_A \;\cong\; f
		\]
		for composable 1-cells $f,g,h$.
	\end{itemize}
	These data satisfy the usual coherence conditions:
	the pentagon identity for $a$ and the triangle identities for $\lambda,\rho$. In reality, a bicategory with only one object (0-cell) is exactly a monoidal category.

	The following construction can be found in \cite[Definition 4.2.1]{leinster2004higher} or \cite[Definition 4.2]{hermida2000representable}.
	
	\begin{construction}
		\label{construction_span}
		Let $\mathcal E$ be a cartesian category and $(T,\eta,\mu)$ a cartesian monad on $\mathcal E$.
		We introduce the bicategory $\mathcal E_{(T)}$ of $T$-spans as follows.
		\begin{itemize}
			\itemsep 2pt
			\item 0-cells are those of $\mathcal E$.
			
			\item 1-cells $E\to E'$ are diagrams in $\mathcal E$
			\[
			\xymatrix{
				& M \ar[dr]^{c} \ar[dl]_{d}& \\
				TE & & E'
			}
			\]
			where $M$ is an object in $\mathcal E$ with a \emph{domain} map $d:M\to T E$
			and a \emph{codomain} map $c:M\to E'$.
			The corresponding identity 1-cell is $T E \xleftarrow{\eta_E} E \xrightarrow{\id_E} E$.
			\item 2-cells
			\[
			\xymatrix{
				& M \ar[dr]^{c} \ar[dl]_{d}& \\
				TE & & E'
			}
			\;\Rightarrow\;
			\xymatrix{
				& M \ar[dr]^{c'} \ar[dl]_{d'}& \\
				TE & & E'
			}
			\]
			are morphisms $\alpha:M\to M'$ in $\mathcal E$ such that $d=d'\alpha$ and $c=c'\alpha$.
			
			\item The composition $\circ$ of 1-cells
			\[
			\xymatrix{
				& M \ar[dl]_{d} \ar[dr]^{c} & \\ T E & & E'
			}
			\qquad
			\xymatrix{
				& M' \ar[dl]_{d'} \ar[dr]^{c'} & \\ T E' & & E''
			}
			\]
			is given by the diagram
			\[
			\xymatrix{ & & & M'\circ M \ar[dl]_{\pi_{TM}} \ar[dr]^{\pi_{M'}} \\
				& & TM \ar[dl]_{T(d)}\ar[dr]^{T(c)} & & M' \ar[dl]_{d'} \ar[dr]^{c'} \\
				& T^2E \ar[dl]_{\mu_E} & & TE' & & E'' \\ 
				TE}
			\]
			where 
			\begin{equation}
				\label{circ_product_Leinster_eq}
				M'\circ M=TM\times_{TE'} M'
			\end{equation}
			The compositions of 2-cells are defined in a similar way and omitted. The associator and unitors are the canonical isomorphisms induced by the universal property of pullbacks together with the monad axioms for $(T,\eta,\mu)$. The cartesianness ensures the needed pullbacks are
			preserved by $T$. 
		\end{itemize}
	\end{construction}

	Let $\mathcal E$ be any cartesian category and take the identity monad $T=\mathrm{Id}_{\mathcal E}$.
	Then, a 1-cell $E\to E'$ is just a diagram
	$
	E \;\xleftarrow{\ d\ }\; M \;\xrightarrow{\ c\ }\; E'
	$.
	The composition of 1-cells is the usual pullback composition, and the identity
	at $E$ is $E \xleftarrow{\ \mathrm{id}_E\ } E \xrightarrow{\ \mathrm{id}_E\ } E$.

	When $T$ is the free monoid monad in (\ref{list_monad_eq}), a 1-cell $E\to E'$ is a space $M$ equipped with ``multi-input to single-output'' maps to $E$ and $E'$ respectively:
	an element $m\in M$ determines a finite list $d(m)=[e_1,\dots,e_n]\in T E$
	of inputs and an output $c(m)\in E'$.
	Given $m\in M$ and $m' \in M'$ as above, the composite is defined exactly when
	$d'(n)=[\,c(m)\,]$, i.e.\ the $T E'$-input of $n$ is the singleton list with
	entry the output of $m$. The composition $M'\circ M$ collects such composable pairs,
	and the left leg of the composite reports the original $E$-inputs $[e_1,\dots,e_n]$,
	while the right leg reports the final output in $E''$.

	Let $\mathcal B$ be a bicategory.  
	A \emph{monad} on $X\in \mathcal B$ consists of a 1-cell $t : X \to X$ and 2-cells (called multiplication and unit)
	$m : t \circ t \;\Rightarrow\; t$, $u : 1_X \;\Rightarrow\; t$
	such that the following coherence conditions hold:
	\begin{enumerate}
		\item  the two composites
		\[
		(t \circ t \circ t) \;\xRightarrow{\,m \circ \mathrm{id}_t\,}\; (t \circ t) 
		\;\xRightarrow{\,m\,}\; t
		\qquad\text{and}\qquad
		(t \circ t \circ t) \;\xRightarrow{\,\mathrm{id}_t \circ m\,}\; (t \circ t) 
		\;\xRightarrow{\,m\,}\; t
		\]
		are equal as 2-cells.
		
		\item the two composites
		\[
		(t \circ 1_X) \;\xRightarrow{\,\mathrm{id}_t \circ u\,}\; (t \circ t) 
		\;\xRightarrow{\,m\,}\; t
		\qquad\text{and}\qquad
		(1_X \circ t) \;\xRightarrow{\,u \circ \mathrm{id}_t\,}\; (t \circ t) 
		\;\xRightarrow{\,m\,}\; t
		\]
		are both equal to the identity 2-cell $\mathrm{id}_t : t \Rightarrow t$.
	\end{enumerate}

	\subsection{Generalized multicategories}
	\label{s_T_multicategory}
	
	Let $\mathcal E$ be a cartesian category and $T=(T,\eta,\mu)$ a cartesian monad on $\mathcal E$.
	%
	\begin{defn}
		\label{multicategory_defn}
		A \emph{$T$-multicategory} (or a \textit{$T$-multicategory on $L$} to stress the object) is defined as a monad in the bicategory $\mathcal E_{(T)}$ of $T$-spans. Specifically, it is a tuple $C=(L,\mathcal M;d,c; \iota,\gamma)$ consisting of 
		\begin{itemize}
			\itemsep 2pt
			\item a 1-cell in $\mathcal E_{(T)}$, that is, a diagram
			\[
			\xymatrix{
				& \mathcal M \ar[dr]^{c} \ar[dl]_{d}& \\
				TL & & L
			}
			\]
			We think of elements of $\mathcal M$ as \emph{arrows}, with a \emph{domain map} $d$ (a ``$T$-list'' of inputs in $L$) and a \emph{codomain map} $c$ (a single output in $L$).
			\item $\gamma:\ \mathcal M\circ \mathcal M = \mathcal M \times_{T L} T \mathcal M \longrightarrow \mathcal M	$
			is a map (called multiplication) such that 
			\begin{equation}
				\label{T-multicategory_condition_2_eq}
				d\circ \gamma \;=\; \mu_{L}\circ T(d)\circ \pi_{T\mathcal M}
				,\qquad \text{and} \qquad c\circ \gamma \;=\; c\circ \pi_{\mathcal M}
			\end{equation}
			\item $\iota:L\to \mathcal M$ is a map (called unit) such that 
			\begin{equation}
				\label{T-multicategory_condition_1_eq}
				d\circ \iota=\eta_{L}, \qquad \text{and} \qquad c\circ \iota=\mathrm{id}_{L}
			\end{equation}
		\end{itemize}
		where $\pi_{T\mathcal M}: \mathcal M\circ\mathcal M\to T\mathcal M$ and $\pi_{\mathcal M}: \mathcal M\circ \mathcal M \to \mathcal M$ are the natural projection maps from the pullback and $\eta_L,\mu_L$ come from the monad $T=(T,\eta,\mu)$.
		These data are required to satisfy the natural coherence conditions; cf. \cite[Definition 4.2.2 \& 6.2.2]{leinster2004higher}.
		A \textit{$T$-operad} is defined to be a $T$-multicategory on the terminal object $\mathbf 1$ of $\mathcal E$ \cite[Definition 4.2.3]{leinster2004higher}.
		We say a $T$-multicategory \textit{topological} if the underlying category $\mathcal E$ is the category of topological spaces.
	\end{defn}

	\begin{ex}[Recovery of category]
		\label{ex_recover_category}
		Note that if $T$ is the identity monad, then a $T$-multicategory is simply a category \cite[Example 4.2.6]{leinster2004higher}. For example, the multiplication is $\gamma: \mathcal M\times_L \mathcal M \to \mathcal M$ where the fiber product refers to the composable pairs of morphisms; the unit map $\iota$ sends an object in $L$ to the identity morphism in $\mathcal M$.
	\end{ex}
	
	\begin{prop}
		For the free monoid monad $T$ in (\ref{list_monad_eq}), the notion of a $T$-operad coincides with that of a non-symmetric topological operad, while the notion of a $T$-multicategory on $L$ coincides with that of a non-symmetric topological multicategory on $L$.
	\end{prop}
	
	\begin{proof}
		See Proposition \ref{prop_coloured_operad} below for a slightly more general situation.
	\end{proof}

	Let $T$ be a cartesian monad as before. Suppose $C=(\mathcal L,\mathcal M; d,c; \iota,\gamma)$ and $C'=(\mathcal L',\mathcal M'; d',c'; \iota',\gamma')$ are $T$-multicategories.
	A \textit{map $f: C\to  C'$ of $T$-multicategories} is a map $f=(f_0,f_1)$ of the underlying graphs with the following commutative diagrams
	\[
	\xymatrix{
		\mathcal L \ar[rr]^{\iota}\ar[d]^{f_0} & & \mathcal M \ar[d]^{f_1} & & \mathcal M\circ \mathcal M \ar[d]^{f_1\ast f_1} \ar[rr]^\mu  & & \mathcal M \ar[d]^{f_1}  \\
		\mathcal L' \ar[rr]^{\iota'} & & \mathcal M' & & \mathcal M'\circ \mathcal M'  \ar[rr]^{\mu'} & & \mathcal M'
	}
	\]
	where $f_1\ast f_1$ is the natural map induced by two copies of $f_1$ through $T$.

	\subsection{Labeling in operads and multicategories}
	\label{s_additional_grading}
	
	In this subsection, we return to the classical notions of operads and multicategories.
	In an operad, one thinks of $n$-ary operations as forming an object
	$\mathcal O(n)$ in a monoidal category $\mathcal C$. Before imposing any composition laws, the raw data is
	simply a family
	\[
	\{ \mathcal O(n)\}_{n\ge 0}
	\]
	of objects of $\mathcal C$ indexed by arity $n\in\mathbb N$.  
	Such families provide the basic environment in which the notion of operads 
	can be defined (cf.~\cite[Definition~1.98]{markl2002operads}).
	In many situations it is useful to enrich this picture by an additional grading.
	For instance, operations may carry a ``weight,'' ``degree,'' or ``energy''
	that is additive under substitution. Formally, one fixes a monoid
	$(S, + , \theta )$ and considers families
	\[
	\{\mathcal O(n,\beta)\}_{n\ge 0,\ \beta \in S}
	\]
	indexed both by arity $n$ and by a grading $\beta\in S$, and we may think of $\mathcal  O(n)=\coprod_\beta \mathcal  O(n,\beta)$.
	These are like $S$-labeled families of objects in $\mathcal C$.
	We aim to develop this idea for both operads and multicategories.
	One major purpose of introducing this extra grading is to include the structure of \textit{curved} $A_\infty$ algebras \cite{FOOOBookOne,Yuan_I_FamilyFloer} which are crucial in various applications \cite{Yuan_unobs,Yuan_e.g._FamilyFloer,Yuan_c_1,FOOOToricOne,FOOOToricTwo,FOOO_bookblue}.
	Informally, we replace the multilinear maps $\m_k$ with a family ${\m_{k,\beta}}$ carrying an extra monoid grading and obeying the corresponding ``$\beta$-graded'' $A_\infty$ relations; see Section \ref{s_A_inf_algebra} for more details.

	Fix a commutative monoid $(S,+,\theta)$ where $\theta$ is the identity element. 
	In symplectic geometry, we often choose $S \subseteq H_2(X,L)$ or $\pi_2(X,L)$, consisting of classes with non-negative symplectic areas, where $L$ is a Lagrangian submanifold in a symplectic manifold $X$; see Section \ref{s_moduli_disk}.
	
	\begin{convention}
		While the term ``$S$-graded'' may be more natural, the term ``graded'' is often used in other contexts; to avoid ambiguity, let's adopt the term ``labeled’’ below.
	\end{convention}
	
	\begin{defn}
		\label{S_labeled_multicategory_defn}
		A multicategory $(L,\mathcal M; d,c,\iota,\gamma)$ on $L$ is a $T$-multicategory for the free monoid monad $T$. It is called \emph{$S$-labeled} if we are given a \emph{labeling map}
		\[
		|\cdot|:\ \mathcal M \longrightarrow  S
		\]
		with $|\iota(\ell)|=\theta$ for all $\ell\in L$, and satisfying the \emph{label additivity}
		\[
		\bigl|\,\gamma\bigl(x;\,[y_1,\dots,y_k]\bigr)\,\bigr|\ =\ |x|+\sum_{i=1}^k|y_i| \qquad \text{or equivalently} \qquad |x\circ_i y| =|x|+|y|
		\]
		Concretely, an \emph{$S$-labeled multicategory on $L$} is defined to be the following data:
		\begin{enumerate}
			\itemsep 2pt
			\item For each $n\ge 0$ and $\beta\in S$, a space $\mathcal M(n,\beta)$ of $n$-ary operations of label $\beta$.
			
			\item Structure maps (input and output colours)
			\[
			\ev_0:\mathcal M(n,\beta)\to L,
			\qquad
			\ev_i:\mathcal M(n,\beta)\to L \quad (1\le i\le n),
			\]
			so that each operation has $n$ ordered input colours and one output colour.
			
			\item A unit map
			\[
			\iota:L \to \mathcal M(1,\theta),
			\]
			assigning to each colour $\ell\in L$ a distinguished unary operation of label $\theta$.
			
			\item Composition maps:
			for all $k\ge 0$, $n_1,\dots,n_k\ge 0$, and $\beta_0,\beta_1,\dots,\beta_k\in S$,
			\[
			\gamma:\ 
			\mathcal M(k,\beta_0)\ \times_{L^{\times k}}\ \prod_{j=1}^k \mathcal M(n_j,\beta_j)
			\;\longrightarrow\; 
			\mathcal M\!\Bigl(\sum_{j=1}^k n_j,\ \beta_0+\sum_{j=1}^k \beta_j\Bigr).
			\]
			Here the fiber product condition requires $\ev_j(x)=\ev_0(y_j)$ for $x\in\mathcal O(k,\beta_0)$ and $y_j\in\mathcal O(n_j,\beta_j)$.
			Equivalently, we have the partial composition maps
			\[
			\circ_i : \mathcal M(k,\beta_0) \times_{L, (\ev_i, \ev_0)} \mathcal M(n,\beta) \to \mathcal M(k+n-1,\beta_0+\beta)
			\]
			
			\item These data satisfy the obvious associativity and unit axioms.
		\end{enumerate}
		Moreover, we say $\mathcal M$ is \emph{topological} if $L$ and $\mathcal M(n,\beta)$'s are topological spaces and $\ev_0,\ev_i,\iota,\gamma$ are continuous maps.
	\end{defn}

	\begin{prop}
		\label{prop_coloured_operad}
		Let $T$ be the free monoid monad $TX=\coprod_{n\ge 0}X^{\times n}$ on $\mathbf{Top}$ (\ref{list_monad_eq}), and let $(S,+,\theta)$ be a commutative monoid.
		Then:
		\begin{enumerate}
			\itemsep 2pt
			\item The notion of an $S$-labeled $T$-operad coincides with that of an $S$-labeled non-symmetric topological operad.
			\item The notion of an $S$-labeled $T$-multicategory on $L$ coincides with that of an $S$-labeled topological multicategory on $L$
		\end{enumerate}
	\end{prop}
	
	\begin{proof} 
		For the definition of (colored) operads, we refer to standard sources such as \cite{markl2002operads}. The proof proceeds in essentially the same way as in the unlabeled case (e.g. \cite[Example 4.2.7]{leinster2004higher}), and we include the argument for the labeled case here for completeness.
		Without loss of generality, we only address the second statement as the first is the special case $L$ being the terminal object.
		Suppose $(L,\mathcal M;d,c, \iota,\gamma)$ is an $S$-labeled $T$-multicategory in the sense of Definition \ref{S_labeled_multicategory_defn} with the labeling map $|\cdot|$. For $TL \cong \coprod_{n\ge 0} L^{\times n}$, we define the arity-$n$ component $\mathcal M(n)$ of $\mathcal M$ by performing the pullback
		\[
		\xymatrix{
			\mathcal M(n) \ar[r]^-{d_n} \ar[d]_{j_n}  &
			L^{\times n} \ar[d]^{i_n} \\
			\mathcal M \ar[r]_-{d} & TL .
		}
		\]
		where $d_n:\mathcal M(n)\to L^{\times n}$ records the $n$ input objects, and $j_n:\mathcal M(n)\to \mathcal M$ is the inclusion.
		Decompose further by labels:
		\[
		\mathcal M(n)\;=\;\coprod_{\beta\in S}\mathcal M(n,\beta), 
		\quad\text{where}\quad \mathcal M(n,\beta)=\{x\in\mathcal M(n):|x|=\beta\}= |\hspace{-0.25em} \cdot \hspace{-0.25em}|^{-1}(\beta)\cap \mathcal M(n).
		\]
		Thus, 
		\[
		\mathcal M \;=\; \coprod_{n\ge 0}\;\coprod_{\beta\in S}\mathcal M(n,\beta).
		\]
		Applying the universal property of the pullback to the natural inclusion $\mathcal M(n,\beta)\to\mathcal M(n)$, we also have the diagram
		\[
		\xymatrix{
			\mathcal M(n,\beta) \ar[r]^-{d_{n,\beta}} \ar[d]_{j_{n,\beta}}  &
			L^{\times n} \ar[d]^{i_{n}} \\
			\mathcal M \ar[r]_-{d} & TL .
		}
		\]
		for each $\beta\in S$, where $d_{n,\beta}$ and $j_{n,\beta}$ are naturally induced from $d_n$ and $j_n$ respectively. Write
		\begin{equation}
			\label{ev_abstract_eq}
			\ev_0^{(n,\beta)} = c\circ j_{n,\beta}: \mathcal M(n,\beta)\to L, \qquad
			\ev_i^{(n,\beta)} = \pr_i\circ d_{n,\beta}: \mathcal M(n,\beta)\to L\ \ (1\le i\le n)
		\end{equation}
		and call them the evaluation maps.
		In other words, each element $\mathbf u\in\mathcal M(n,\beta)$ has one \emph{output}
		$\ev_0(\mathbf u)$ and $n$ ordered \emph{inputs} $(\ev_1(\mathbf u),\dots,\ev_n(\mathbf u))$.
		For the free monoid monad,
		$\eta_L:L\to TL$ is precisely the inclusion of \emph{singleton lists}, i.e.
		$
		\eta_L \;=\; i_1 \circ \mathrm{id}_L \;:\; L \xrightarrow{\cong} L^{\times 1} \xrightarrow{\ i_1 \ } TL .
		$.
		By the condition $d\circ\iota=\eta_L$ (\ref{T-multicategory_condition_1_eq}),
		\[
		\xymatrix{
			L \ar[r]^-{\ \mathrm{id}_L\ } \ar[d]_{\ \ \iota} &
			L^{\times 1} \ar[d]^{\ \ i_1 } \\
			\mathcal M \ar[r]_-{d} & TL
		}
		\]
		commutes. By the universal property of the pullback, there is a unique map
		$\overline\iota:L\to \mathcal M(1)$ with
		$j_1\circ \overline\iota=\iota$ and $d_1\circ \overline\iota=\mathrm{id}_L$.
		One can show that $\overline\iota$ factors as
		$L\xrightarrow{}\mathcal M(1,\theta)\xrightarrow{}\mathcal M(1)$ because of the requirement $|\iota|=\theta$.
		Abusing the notation, we denote the map $L\to\mathcal M(1,\theta)$ by $\overline \iota$ as well.
		
		The pullback
		\[
		\xymatrix{
			\mathcal M\circ \mathcal M \ar[r]^-{\ \ \pi_{T\mathcal M}} \ar[d]_{\pi_{\mathcal M}}  &
			T\mathcal M \ar[d]^{\,T c} \\
			\mathcal M \ar[r]_-{d} & TL
		}
		\]
		for the free monoid monad $T$ can be identified with the space of pairs
		$
		\bigl(x,\ [y_1,\dots,y_k]\bigr)$ where $x\in\mathcal M$ and $[y_1,\dots,y_k]\in T\mathcal M\equiv \coprod_{k\ge 0} \mathcal M^{\times k}$
		subject to the matching condition
		\[
		d(x)\;=\;T(c)\bigl([y_1,\dots,y_k]\bigr)\;=\;[\,c(y_1),\dots,c(y_k)\,]\ \in\ TL.
		\]
		Under this identification, we have
		$
		\pi_{\mathcal M}\bigl(x,[y_1,\dots,y_k]\bigr)=x$ and $
		\pi_{T\mathcal M}\bigl(x,[y_1,\dots,y_k]\bigr)=[y_1,\dots,y_k]$.
		Observe that the following diagram
		\[
		\xymatrix@C=2.4em@R=2.0em{
			\mathcal M(k,\beta_0)\ \times_{L^{\times k}}\ \displaystyle\prod_{j=1}^k \mathcal M(n_j,\beta_j)
			\ar[r] \ar[d]
			&
			\displaystyle\prod_{j=1}^k \mathcal M(n_j,\beta_j)
			\ar[d]^{(\ev_0,\dots,\ev_0)}
			\ar[r]
			&
			T\mathcal M
			\ar[dd]^{\ T c}
			\\
			\mathcal M(k,\beta_0)
			\ar[r]_-{(\ev_1,\dots,\ev_k)}
			\ar[d]
			&
			L^{\times k}
			\ar[rd]_-{\ i_k \ }
			&
			\\
			\mathcal M
			\ar[rr]_-{d}
			&
			& 
			TL
		}
		\]
		is a summand of
		\[
		\xymatrix{ \mathcal M \times_{TL} T\mathcal M \ar[rr] \ar[d] & & T\mathcal M \ar[d]^{Tc} \\
			\mathcal M \ar[rr]^d & & TL}
		\]
		Thus, the composition $\gamma$ induces
		\begin{equation}
			\label{gamma_n1nk_eq}
			\gamma^{\beta_0;\beta_1,\dots,\beta_k}_{n_1,\dots, n_k}: \ \mathcal M(k,\beta_0)\ \times_{L^{\times k}}\ \displaystyle\prod_{j=1}^k \mathcal M(n_j,\beta_j ) \to  \mathcal M \times_{TL} T\mathcal M \to \mathcal M
		\end{equation}
		We further claim that it has image contained in $\mathcal M(n_1+\cdots+n_k,\beta_0+\beta_1+\cdots+\beta_k)$.
		To see this, by definition, an input point is a tuple
		$(x;[y_1,\dots,y_k])$ with $x\in\mathcal M(k,\beta_0)$, $y_j\in\mathcal M(n_j,\beta_j)$ such that the
		matching conditions $\ev_j(x)=\ev_0(y_j)$ hold.
		By the source law at (\ref{T-multicategory_condition_2_eq}),
		\[
		d\!\left(\gamma^{\beta_0;\beta_1,\dots,\beta_k}_{n_1,\dots,n_k}(x; [y_1,\dots,y_k])\right) = \mu_L ([ d(y_1),\dots, d(y_k)])
		\]
		Since each $d(y_j)$ lives in $L^{\times n_j}$ and the monad multiplication $\mu_L$ concatenates the lists $d(y_1),\dots, d(y_k)$, it follows that the left hand side lives in $L^{\times( n_1+\cdots +n_k)}$.
		By the pullback definition $
		\mathcal M(n)=\mathcal M\times_{TL} L^{\times n}$
		along the inclusion $L^{\times n}\hookrightarrow TL$, this exactly says
		the composite lies in the subobject $\mathcal M(n_1+\cdots+n_k)$. It further lies in $\mathcal M(n_1+\cdots+n_k, \beta_0+\beta_1+\cdots+\beta_k)$ due to the label additivity condition:
		\[
		|\gamma(x; [y_1,\dots, y_k]) = |x|+\sum_j  |y_j| =\beta_0+\cdots+\beta_j
		\]
		
		Equivalently, given $x\in\mathcal M(k,\beta_0)$ and $y\in\mathcal M(n,\beta)$ with $\ev_i(x)=\ev_0(y)$ for a fixed $i$, we can define the partial compositions
		\begin{equation}
			\label{partial_composition_abstract_eq}
			\circ_i : \mathcal M(k,\beta_0) \times_{L, (\ev_i, \ev_0)} \mathcal M(n,\beta) \to \mathcal M(k+n-1,\beta_0+\beta)
		\end{equation}
		via
		\[
		x\circ_i y \;:=\;
		\gamma_{\,1,\dots,1,\;n,\;1,\dots,1}^{\beta_0;\,\theta,\dots,\theta,\;\beta,\;\theta,\dots,\theta}
		\bigl(x;\ y_1,\dots,y_{i-1},y,y_{i+1},\dots,y_k\bigr)
		\;\in\; \mathcal M(k+n-1,\ \beta_0+\beta).
		\]
		where $y_j:=\overline\iota(\ev_j(x))\in \mathcal M(1,\theta)$ for $j\ne i$.
		Associativity and unit axioms for $\gamma$ ensure that the partial compositions $\circ_i$ satisfy the standard operad axioms (see \cite[Section 5.9.4]{loday2012algebraic}): for any $\lambda\in\mathcal M(\ell,\beta_1)$, $\mu\in \mathcal M(m,\beta_2)$, and $\nu\in \mathcal M(n,\beta_3)$, we have
		\begin{equation}
			\label{operad_axiom_standard}
			\begin{aligned}
				(\lambda \circ_i \mu)\ \circ_{\,i-1+j}\nu \; &=\; 
				\lambda \circ_i (\mu \circ_j \nu), 
				\qquad 1 \leq i \leq l,\; 1 \leq j \leq m, \\
				(\lambda \circ_i \mu)\ \circ_{\,k-1+m}\nu \; &=\; 
				(\lambda \circ_k \nu) \circ_i \mu, 
				\qquad 1 \leq i < k \leq l.
			\end{aligned}
		\end{equation}
	\end{proof}

	Note that the partial compositions $\circ_i$ in 
	(\ref{partial_composition_abstract_eq}), together with the properties 
	(\ref{operad_axiom_standard}), differ from those of a usual 
	non-symmetric operad in a monoidal category $(\mathcal C,\odot)$. 
	Indeed, in our setting the fiber product (pullback) in 
	(\ref{partial_composition_abstract_eq}) depends on $i$, so the domain 
	of the partial composition varies with $i$. This does not agree with the 
	standard definition in the literature 
	(see \cite[Definition~1.16]{markl2002operads}, 
	\cite[Section~5.9.4]{loday2012algebraic}).
	It is actually a multicategory or a colored operad.
	
	In the last, we provide some natural examples of $S$-labeled operads.

	\begin{ex}[\textbf{$S$-weighted rooted planar trees}]
		\label{ex_rooted_trees}
		A \textit{planar tree} is a tree $\Gamma$ with an embedding $\Gamma \xhookrightarrow{} \mathbb D^2\subset \mathbb C$ such that a vertex $v$ has only one edge if and only if $v$ lies in the unit circle $\partial \mathbb D^2$. Such a vertex is called an exterior vertex, and each other vertex is called an interior vertex. The set of exterior (resp. interior) vertices is denoted by $\Cext_0(\Gamma)$ (resp. $\Cint_0(\Gamma)$). Then, $C_0(\Gamma)=\Cext_0(\Gamma)\cup \Cint_0(\Gamma)$ is the set of all vertices.
		Also, an edge of $\Gamma$ is called exterior if it contains an exterior vertex and is called interior otherwise. The set of all exterior edges is denoted by $\Cext_1(\Gamma)$ and that of all interior edges is denoted by $\Cint_1(\Gamma)$.
		An exterior edge is called the \textit{outgoing edge} if it contains the root and is called an \textit{incoming edge} or a \textit{leaf} if not.
		A \textit{planar rooted tree} is a planar tree $\Gamma$ with a specified exterior vertex $v_0$ therein. We call $v_0$ the \textit{root}; it produces a natural partial order on the set of vertices $C_0(\Gamma)$ by setting
		$
		v<v'
		$
		if $v\neq v'$ and there is a path in $\Gamma$ from $v$ to $v_0$ which passes through $v'$. Particularly, the root $v_0$ is the largest vertex with respect to this partial order.
		Besides, we order the leaves counterclockwise starting from the root.
		
		Fix a commutative monoid $(S,+,\theta)$. 
		For $n\ge 0$ and $\beta\in S$, let $\mathcal T_S(n,\beta)$ be the set of isomorphism classes of finite \emph{planar rooted trees} $\Gamma$ with $n$ outgoing edges, together with an $S$-weight 
		\[
		w:V_{\mathrm{int}}(\Gamma)\longrightarrow S
		\]
		assigned to each interior vertex, such that the \emph{total weight} is
		\[
		|\Gamma|:=\sum_{v\in V_{\mathrm{int}}(\Gamma)} w(v)\ =\ \beta.
		\]
		
		We claim that $\mathcal T_S=\{\mathcal T_S(n,\beta)\}_{n\ge 0, \beta\in\Gamma}$ is an $S$-labeled operad.
		Indeed, the composition maps are defined as follows.
		Given a tree $\Gamma_0\in \mathcal T_S(k,\beta_0)$ and trees $\Gamma_i\in \mathcal T_S(n_i,\beta_i)$ $(1\le i\le k)$, define
		\[
		\gamma (\Gamma_0;\,\Gamma_1,\dots,\Gamma_k)=\gamma_{k;\,n_1,\dots,n_k}^{\,\beta_0;\,\beta_1,\dots,\beta_k}(\Gamma_0;\ \Gamma_1,\dots,\Gamma_k)
		\]
		to be the planar rooted tree obtained by grafting the root of $\Gamma_i$ to the $i$-th leaf of $\Gamma_0$ (for all $i$) and reading leaves in the induced counterclockwise order. 
		The vertex weights are inherited from the pieces, so the resulting tree has the total weight
		\[
		\bigl|\gamma(\Gamma_0;\,\Gamma_1,\dots, \Gamma_k)\bigr|=|\Gamma_0|+\sum_{i=1}^k|\Gamma_i|=\beta_0+\beta_1+\cdots+\beta_k,
		\]
		Therefore, we have defined:
		\[
		\gamma_{k;\,n_1,\dots,n_k}^{\,\beta_0;\,\beta_1,\dots,\beta_k}:\ 
		\mathcal T_S(k,\beta_0)\times \prod_{i=1}^k \mathcal T_S(n_i,\beta_i)\longrightarrow 
		\mathcal T_S\!\Bigl(\textstyle\sum_i n_i,\ \beta_0+\sum_i\beta_i\Bigr).
		\]
		Equivalently, we can define the $i$-th partial composition
		\[
		\circ_i:\   \mathcal T_S(k,\beta_0)\times \mathcal T_S(n,\beta)\longrightarrow \mathcal T_S(k+n-1,\beta_0+\beta)
		\]
		by grafting the root of $\Gamma\in\mathcal T_S(n,\beta)$ onto the $i$-th leaf of $\Gamma_0\in \mathcal T_S(k,\beta_0)$. The vertex weights are inherited similarly, and the total weight is $\beta_0+\beta$.
	\end{ex}

	\begin{ex}[\textbf{$S$-labeled endomorphism operad}]
		\label{ex_endomorphism_operad}
		Fix a commutative monoid $(S,+,\theta)$.
		Fix a vector space $A$.
		We define the $S$-labeled endomorphism operad $\End_A^S=\{\End_A^S(n,\beta)\}$ as follows.
		Given $n\ge 0$ and $\beta\in S$, we set
		\[
		\End^S_A(n,\beta) = \Hom (A^{\otimes n} , A)
		\]
		where the right hand side does not depend on $\beta$, and $\beta$ is just an extra label. In other words,
		\[
		\End_A^S(n)=\coprod_{\beta\in S} \End_A^S(n,\beta)= \coprod_{\beta\in S} \{\beta\}\times \Hom(A^{\otimes n}, A)
		\]
		When $n=0$, our convention is that we identify $\Hom(A^{\otimes 0}, A)$ with $A$.
		The extra labels are useful in the studies of curved $A_\infty$ algebras in Lagrangian Floer theory; see e.g. \cite{FuCyclic,Yuan_unobs,Yuan_I_FamilyFloer}. We will also go back to this point later.

		The unit $\eta$ for $\End^S_A(1,\theta)$ is given by the identity map $\id_A$.
		Given $k\ge0$, $n_1,\dots,n_k\ge0$, and degrees $\beta_0,\beta_1,\dots,\beta_k\in S$, we define the composition maps
		to be the multilinear composition:
		\[
		\gamma_{k;\,n_1,\dots,n_k}^{\,\beta_0;\,\beta_1,\dots,\beta_k}(f;g_1,\dots,g_k)
		=
		f\ \circ\ \bigl(g_1\ \otimes \ \cdots\ \otimes \ g_k\bigr) 
		\]
		The partial compositions are defined as follows.
		For $f\in\End^S_A(k,\beta_0)$ and $g\in\End^S_A(n,\beta)$, we define
		$
		f\circ_i g\ \in\ \End_A^S(k+n-1,\ \beta_0+\beta)
		$
		to be the usual insertion of $g$ into the $i$-th input of $f$, recording the sum $\beta_0+\beta$ on the target.
	\end{ex}

	\subsection{Moduli spaces of pseudo-holomorphic disks as multicategories}
	\label{s_moduli_disk}
	
	Let $(X,\omega)$ be a closed symplectic manifold, and $J$ an $\omega$-tame almost complex structure.
	Let $\iota: L \to X$ be a Lagrangian submanifold, that is, $\iota^*\omega=0$ and $\dim L= \frac{1}{2}\dim X$.

	When $\iota$ is embedded, we may identify $L$ with its image in $X$ and thus view it as a subspace $L\subset X$.
	The zero element of $H_2(X,L)$ is denoted by $\theta=\theta_L$. We consider the commutative monoid
	\begin{equation}
		\label{S_L_eq}
		S_L := \{ \theta \} \cup \{ \beta\in H_2(X,L) \mid  \omega(\beta) >0\}
	\end{equation}
	
	Fix an integer $k \ge 0$ and a relative homotopy class $\beta$ in $S_L$ such that $(k,\beta)\neq (0,\theta), (1,\theta)$.
	A \textit{$J$-holomorphic stable disk of type $(k, \beta)$} consists of the data
	\[
	\mathbf u=\bigl( \Sigma, z_0,z_1,\dots, z_k,\ u\bigr)
	\]
	where:
	\begin{itemize}
		\itemsep 2pt
		\item $\Sigma$ is a connected, oriented, nodal, genus-$0$ bordered Riemann surface.
		Nodes may be \emph{boundary nodes} (on $\partial\Sigma$) or \emph{interior nodes} (in $\Sigma\setminus\partial\Sigma$).
		\item $z_0,z_1,\dots, z_k \in \partial\Sigma$ are pairwise distinct $k+1$ (rather than $k$) \emph{boundary marked points}, ordered along the induced orientation of $\partial\Sigma$; no marked point is a node.
		\item $u:\Sigma\to X$ is continuous, $J$-holomorphic on each component, and satisfies the Lagrangian boundary condition
		$u(\partial\Sigma)\subset L$, with relative class $[u]=\beta\in S_L$.
		\item we require $\mathbf u$ is \textit{stable} in the sense that the automorphism group of $\mathbf u$ is finite.
	\end{itemize}
	Here two such objects $\mathbf u=(\Sigma, (z_i) ,u)$ and $\mathbf u'=(\Sigma', (z_i'),u')$ are \emph{isomorphic} if there exists a biholomorphism $\phi:\Sigma\to\Sigma'$ sending boundary to boundary, marked points to the corresponding marked points, and nodes to nodes, such that $u=u'\circ\phi$. We call $\phi$ an isomorphism from $\mathbf u$ to $\mathbf u'$. If $\mathbf u=\mathbf u'$, then $\phi$ is called an automorphism.
	
	
	\begin{defn}
		\label{moduli_embed_defn}
		For $(k,\beta)\neq (0,\theta), (1,\theta)$, the \emph{moduli space of $J$-holomorphic stable disks of type $(k,\beta)$}, denoted by ${\mathcal M}(k,\beta)$,
		is defined to be the set of isomorphism classes $[\mathbf u]$ of such objects $\mathbf u$. It admits evaluation maps
		\[
		\ev_i :{\mathcal M} (k,\beta) \to L \quad ,\qquad i=0,1,\dots, k
		\]
		given by $\ev_i ([\mathbf u])=u(z_i)$.
		For the case\footnote{ \scriptsize  This concerns the so-called stability for pseudo-holomorphic maps in symplectic geometry. We need to assume $k+1\ge 3$ marked points to ensure the stability of a constant map.} $(k,\beta)=(1,\theta)$, we artificially define
		$\mathcal M(1,\theta)=L$ equipped with an obvious evaluation map $\ev_0=\id_L: L\to L$.
		For $(k,\beta)=(0,\theta)$, we define $\mathcal M(0,\theta)=\varnothing$.
		
		Set $\mathcal M(k)=\coprod_{\beta\in S_L}\mathcal M(k,\beta)$, and call an element in $\mathcal M(k)$ \textit{a $J$-holomorphic stable disk of type $k$}.
	\end{defn}

	\smallskip
	
	Let $T$ be the free monoid monad (\ref{list_monad_eq}).
	We then have the following:
	
	\begin{prop}
		\label{multicategory_moduli_prop}
		The moduli space system 
		\begin{equation}
			\label{moduli_M_k_beta_eq}
			\mathcal M = \coprod_{k\ge 0} \mathcal M(k)  = \coprod_{\substack{k\ge 0, \beta\in S_L}} \mathcal M(k,\beta) 
		\end{equation}
		and their evaluation maps
		\[
		\xymatrix{
			& \mathcal M  \ar[dr]^{\ev_0}\ar[dl]_{(\ev_i)}& \\
			TL & & L
		}
		\]
		form an $S_L$-labeled topological multicategory on $L$ 
	\end{prop}

	\begin{proof}
		The labeling map $|\cdot|$ as in Definition \ref{S_labeled_multicategory_defn} is defined by requiring $|\mathbf u|=\beta$ for each $\mathbf u\in\mathcal M(k,\beta)$.
		The unit map $\iota: L\to \mathcal M(1,\theta) \xhookrightarrow{} \mathcal M $ is simply given the inclusion. 
		As discussed around (\ref{gamma_n1nk_eq}) in the proof of Proposition \ref{prop_coloured_operad}, establishing the multiplication map $\gamma$ is equivalent to specifying the partial composition maps with the required axioms:
		\begin{equation}
			\label{circ_i_moduli_eq}
			\circ_i=\circ_i^{\beta_1,\beta_2}: \mathcal M (k_1,\beta_1) \times_{L,(\ev_i,\ev_0)} \mathcal M (k_2,\beta_2) \to \mathcal M (k_1+k_2-1, \beta_1+\beta_2)
		\end{equation}
		Indeed, if either $(k_1,\beta_1)$ or $(k_2,\beta_2)$ is $(1,\theta)$, then it degenerates to the identity map.
		In other cases, given $[\mathbf u_1]=[\Sigma_1, z_0,z_1,\dots, z_{k_1}, u_1] \in\mathcal M (k_1,\beta_1)$ and $[\mathbf u_1]=[\Sigma_2, w_0, w_1,\dots, w_{k_2}, u_2] \in\mathcal M (k_2,\beta_2)$, we define 
		\[
		[\mathbf u_1]\circ_i [\mathbf u_2]=[\Sigma, y_0, y_1,\dots, y_{k_1+k_2-1}, u]
		\]
		as follows.
		The nodal surface $\Sigma$ is obtained by gluing $z_i\in \Sigma_1$ with $w_0\in \Sigma_2$, and the marked points are $(y_0,y_1,\dots, y_{k_1+k_2-1})=(z_0,\dots, z_{i-1}, w_1, \dots, w_{k_2}, z_{i+1},\dots, z_{k_1})$.
		The map $u:\Sigma \to X$ is defined by declaring $u|_{\Sigma_1}=u_1$ and $u|_{\Sigma_2}=u_2$; this is well-defined since $u_1(z_i)=\ev_i(\mathbf u_1)=\ev_0(\mathbf u_2)=u_2(w_0)$.
		The class of $u$ is simply the sum of the classes of $u_1$ and $u_2$, 
		namely $\beta_1+\beta_2$. 
		By construction, one may readily verify that the associativity axioms (\ref{operad_axiom_standard}) hold. Lastly, this is a topological multicategory, because it is known that each $\mathcal M (k,\beta)$ is a compact and Hausdorff topological space \cite[Theorem 7.1.43]{FOOOBookTwo}, and the evaluation maps $\ev_0,\ev_i$ are continuous \cite[Proposition 7.1.1]{FOOOBookTwo}.
	\end{proof}

	\section{$\mathbf {fc}$-multicategories}
	\label{s_fc_multicategory}
	
	An $\mathbf{fc}$-multicategory is a general framework introduced by Leinster \cite{leinster1999fc}, which can be viewed as a kind of \textit{"2-dimensional multicategory,"} where one keeps track not only of objects and multi-input operations, but also of two different kinds of morphisms and the 2-cells that relate them. 
	The letter "$\mathbf{fc}$" stands for "free category".
	This can simultaneously encompass ordinary categories, multicategories, bicategories, monoidal categories, double categories, and so on. In this way, $\mathbf {fc}$-multicategories provide a convenient unifying setting in which many apparently different categorical constructions can be studied side by side.
	And, one of our goals in this paper is to explore its potential topological and geometric realizations.

	\subsection{Directed graphs and the functor $\mathbf{fc}$}
	\label{s_fc_directedgraphs}
	Let $\mathcal D$ be the category of \textit{directed graphs} (sometimes people also call them \textit{quivers}) defined by the functor category $\mathcal D=[\mathbb H^{\mathrm{op}}, \mathbf{Set}]$ where $\mathbb H$ is the category with two objects and two distinguished morphisms $\{0 \rightrightarrows 1\}$.
	Specifically, an object in $\mathcal D$, called a \textit{directed graph}, is a tuple $(V,E, s, t)$ consisting of a set $V$ of vertices and a set $E$ of directed edges, equipped with two functions $s,t:E\to V$.
	Given two vertices $v,v'\in V$, we write
	$
	E(v,v')=\{ e\in E \mid s(e)=v , \ t(e)=v'\}
	$
	for the set of edges from $v$ to $v'$.
	
	\textit{Slightly abusing notation, we often write $E=(V,E,s,t)$ for a directed graph $(V,E,s,t)$; thus, $E$ may denote either the graph itself or its edge set, depending on the context.}

	A directed sub-graph of $E$ is a quadruple
	$
	E'=(V',E',s|_{E'},t|_{E'}),
	$
	where $V'\subseteq V$, $E'\subseteq E$, and $s(E'),t(E')\subseteq V'$.
	A \textit{morphism} $\pmb \phi=(\bar\phi,\phi):E\to E'$ between two directed graphs $E=(V,E,s,t)$ and $E'=(V',E',s',t')$ consists of two continuous maps
	$
	\bar\phi:V\to V'$, $\phi :E\to E'$
	such that the source and target maps are preserved, i.e.
	$		\bar\phi \circ s = s'\circ \phi$ and $\bar\phi\circ t = t'\circ \phi$.
	Composition and identities are defined componentwise in the evident way.
	\[
	\xymatrix{ 
		E \ar@<2pt>[d]^t \ar@<-2pt>[d]_s \ar[rr]^{\phi}  & & E'  \ar@<2pt>[d]^{t'} \ar@<-2pt>[d]_{s'} \\
		V \ar[rr]^{\bar\phi}  & & V'
	}
	\]

	Let $\mathbf{Cat}$ be the category of small categories.
	Every category can be viewed as a directed graph, but the converse is not true in general.
	There is a functor
	$
	F: \mathcal D \longrightarrow \mathbf{Cat}
	$
	called the \emph{path category} functor, which is left adjoint to the forgetful functor $U: \mathbf{Cat} \to \mathcal D$.
	Specifically, for a directed graph $E=(V,E,s,t)$, the objects of the category $F(E)$ are the vertices in $V$, and a morphism $x\to y$ in $F(E)$ is a (possibly empty) finite composable path
	\[
	\xymatrix{
		\cdot & \cdot\ar[l]_{e_n} & \cdot \ar[l] & \cdot \ar[l] & \cdot \ar[l]_{e_1} 
	}
	\]
	denoted by $(e_1 , \dots, e_n)$. For $n=0$ we get the empty path $x\to x$ which is the identity. Composition is concatenation of paths.
	For a morphism of directed graphs $\pmb \phi:E\to E'$,
	define
	$
	F(\pmb \phi):F(E)\longrightarrow F(E')
	$
	by $x\mapsto \bar\phi(x)$ on objects and by sending a path
	$(e_1,\dots, e_n)$ to the path $(\phi(e_1), \dots, \phi(e_n))$ on morphisms.

	Define 
	\begin{equation}
		\label{fc_eq}
		\mathbf{fc}:=U\circ F : \mathcal D \to \mathcal D
	\end{equation}
	Specifically, given a directed graph $E=(V,E,s,t)$, the new directed graph $\mathbf{fc}(E)$, denoted by 
	\[
	E^*=(V,E^\ast,s^\ast,t^\ast),
	\]
	is defined as follows: The set $V$ of vertices is unchanged. The set $E^\ast$ of edges is the set of all finite composable paths in the original $E$, that is,
	\begin{equation}
		\label{E*_eq}
		E^\ast\ = \mathbf{fc}(E) = \ \bigsqcup_{n\ge 0}\ \{(e_1,\dots,e_n)\in E^{\times n} \mid t(e_i)=s(e_{i+1})\}
	\end{equation}		
	For a path $\vec e=(e_1,\dots,e_n)$ put $s^\ast(\vec e):=s(e_1)$ and $t^\ast(\vec e):=t(e_n)$. For a morphism of directed graphs
	$\pmb \phi=(\bar \phi,\phi): E\to E'$, we define a new morphism of directed graphs $\mathbf{fc}(\pmb\phi):\mathbf{fc}(E)\to \mathbf{fc}(E')$ by
	\[
	\overline{\mathbf{fc}(\pmb\phi)}:=\bar\phi,\qquad
	\mathbf{fc}(\pmb\phi)(e_1,\dots,e_n):=\big(\phi(e_1),\dots,\phi_1(e_n)\big).
	\]
	For a composable family of edge-strings 
	\[
	\vec e_i=(e_{i1},\dots,e_{i n_i})\in E^{*}\quad(1\le i\le k),
	\]
	with $t^*(\vec e_i)=t(e_{in_i})=s(e_{i+1,1})=s^*(\vec e_{i+1})$, 
	we write their \textit{concatenation} in $E^{*}$ as
	\[
	\vec e_1 * \cdots * \vec e_k
	\ =\ (e_{11},\dots,e_{1n_1},\,\dots,\,e_{k1},\dots,e_{kn_k}).
	\]
	One can also view the concatenation as a map
	\begin{equation}
		\label{concatenation_eq}
		E^*\times_V E^* \times_V \cdots \times_V E^* \to E^*  
	\end{equation}

	\begin{defn}
		\label{profile_loop_defn}
		A \textit{profile-loop} of a directed graph $E=(V,E, s,t)$ is defined as a pair
		\[
		(\vec e ; e_0) = ((e_1,\dots, e_n) ; e_0) \ \in  E^*\times E
		\]
		with $\vec e=(e_1,\dots, e_n)\in E^*$, $e_0\in E$, and matching endpoints in the sense that 
		\[
		s^*(\vec e) =s(e_0) \qquad  \text{and} \qquad t^*(\vec e)=t(e_0)
		\]
		In other words, it looks like
		\[
		\xymatrix{
			\cdot & \cdot\ar@/_5pt/[l]_{e_n} & \cdot \ar@/_5pt/[l] & \cdot \ar@/_5pt/[l] & \cdot \ar@/_5pt/[l]_{e_1} \ar@/^10pt/ [llll]^{e_0}
		}
		\]
		in the graph.
		Denote the set of all profile-loops in $(V,E)$ by $\mathcal P_E$.
		Besides, we call $(e; e)$ the {identity profile-loop} at $e$.
		The special case $n=0$ is allowed above, namely, $(\varnothing; e_0)$ with $s(e_0)=t(e_0)$ is also a profile-loop and is depicted as a tree with one vertex and one edge:
		\[
		\xymatrix@1{ \bullet \ar@(dr,dl)[]^{e_0} }
		\]
	\end{defn}

	\subsection{$\mathbf{fc}$-multicategory in concrete terms}
	\label{s_fc_multicategory_unpack}

	\begin{figure}
		\centering
		\begin{tikzpicture}[x=0.75pt,y=0.75pt,yscale=-0.7,xscale=0.7]
			
			\draw  [dash pattern={on 0.75pt off 3.75pt}]  (205.22,44.77) .. controls (286.79,44.1) and (310,62.91) .. (339.53,85.08) ;
			\draw    (393.11,195.4) -- (48.44,194.48) ;
			\draw [shift={(46.44,194.47)}, rotate = 0.15] [color={rgb, 255:red, 0; green, 0; blue, 0 }  ][line width=0.75]    (10.93,-3.29) .. controls (6.95,-1.4) and (3.31,-0.3) .. (0,0) .. controls (3.31,0.3) and (6.95,1.4) .. (10.93,3.29)   ;
			\draw    (402,185) -- (372,132.73) ;
			\draw [shift={(371,131)}, rotate = 60.14] [color={rgb, 255:red, 0; green, 0; blue, 0 }  ][line width=0.75]    (10.93,-3.29) .. controls (6.95,-1.4) and (3.31,-0.3) .. (0,0) .. controls (3.31,0.3) and (6.95,1.4) .. (10.93,3.29)   ;
			\draw    (66,121) -- (39.4,180.76) ;
			\draw [shift={(38.59,182.59)}, rotate = 293.99] [color={rgb, 255:red, 0; green, 0; blue, 0 }  ][line width=0.75]    (10.93,-3.29) .. controls (6.95,-1.4) and (3.31,-0.3) .. (0,0) .. controls (3.31,0.3) and (6.95,1.4) .. (10.93,3.29)   ;
			\draw    (140,56) -- (85.61,95.82) ;
			\draw [shift={(84,97)}, rotate = 323.79] [color={rgb, 255:red, 0; green, 0; blue, 0 }  ][line width=0.75]    (10.93,-3.29) .. controls (6.95,-1.4) and (3.31,-0.3) .. (0,0) .. controls (3.31,0.3) and (6.95,1.4) .. (10.93,3.29)   ;
			
			\draw (22.45,188.32) node [anchor=north west][inner sep=0.75pt]    {$v_{n}$};
			\draw (400.79,188.32) node [anchor=north west][inner sep=0.75pt]    {$v_{0}$};
			\draw (65.61,102.32) node [anchor=north west][inner sep=0.75pt]    {$v_{n-1}$};
			\draw (140.18,42.52) node [anchor=north west][inner sep=0.75pt]    {$v_{n-2}$};
			\draw (357.1,112.06) node [anchor=north west][inner sep=0.75pt]    {$v_{1}$};
			\draw (32.59,136.29) node [anchor=north west][inner sep=0.75pt]    {$e_{n}$};
			\draw (86,50) node [anchor=north west][inner sep=0.75pt]    {$e_{n-1}$};
			\draw (389.13,149.31) node [anchor=north west][inner sep=0.75pt]    {$e_{1}$};
			\draw (208.01,200.33) node [anchor=north west][inner sep=0.75pt]    {$e_{0}$};
			\draw (207.51,121.5) node [anchor=north west][inner sep=0.75pt]    {$\mathbf{u}$};
			
		\end{tikzpicture}
		\caption{a 2-cell in a vertically discrete $\mathbf {fc}$-multicategory}
		\label{fig:2_cell_vert_discrete}
	\end{figure}
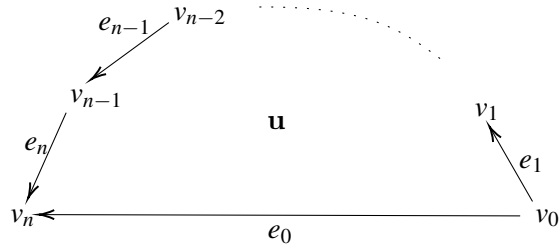

	It is known that $\mathbf{fc}$ is a cartesian monad (see e.g. \cite[Section 5.1]{leinster2004higher}), so the general definitions for $T$-categories with a cartesian monad $T$ (Definition \ref{multicategory_defn}) applies. Namely, a \textbf{\textit{$\mathbf{fc}$-multicategory}} is defined as the data of a tuple
	\begin{equation}
		\label{fc_mul_tuple}
		\mathscr M = ( V, E , \mathcal V, \mathcal M, \pmb d, \pmb c, \pmb \iota, \pmb \gamma)
	\end{equation}
	consisting of two directed graphs $E=(V,E,s,t)$, $\mathcal M=(\mathcal V, \mathcal M, \hat s, \hat t)$ together with a diagram of directed graphs $E^* \xleftarrow{\pmb d} \mathcal M \xrightarrow{\pmb c} E$ where $E^*=\mathbf{fc}(E)$ and $\pmb d=(\bar d,d)$, $\pmb c=(\bar c,c)$ are maps of directed graphs.
	These data satisfy the conditions as described in Definition \ref{multicategory_defn}. There are also associated unit and multiplication maps $\pmb \iota=(\bar\iota,\iota): (V,E)\to (\mathcal V, \mathcal M)$ and $\pmb \gamma=(\bar\gamma,\gamma): (\mathcal V,\mathcal M) \circ (\mathcal V,\mathcal M) \to (\mathcal V,\mathcal M)$ satisfying the associativity and unit axioms.
	We also say that the $\mathbf{fc}$-multicategory \textit{$\mathscr M$ is over $E$} and $E=(V,E)$ is the \textit{\textbf{underlying directed graph}} of $\mathscr M$.
	
	To illustrate, we have the following diagram:
	\[
	\label{fc_multicategory_diagram_2_eq}
	\xymatrix{
		& (\mathcal V, \mathcal M) \ar[dr]^{\pmb c} \ar[dl]_{\pmb d}& \\
		(V, E^*) & & (V, E)
	}
	\]

	Following the works of Leinster \cite{leinster1999fc,leinster2004higher}, we can unpack the data for a $\mathbf{fc}$-multicategory as follows: 
	
	\paragraph{$\bullet$}
	First, elements of $V$ are called \textit{objects} or \textit{0-cells}; elements of $E$ are called \textit{horizontal 1-cells}; elements of $\mathcal V$ are called \textit{vertical 1-cells}; elements of $\mathcal M$ are called \textit{2-cells}.
	Specifically, an element of $\mathcal M$ can be represented as\footnote{ \scriptsize  To maintain compatibility with the standard orientation of pseudo-holomorphic curves in symplectic geometry, we adopt in (\ref{2_cell_eq}) arrow directions opposite to Leinster’s convention \cite{leinster2004higher}. We apologize for any potential confusion.}
	\begin{equation}
		\label{2_cell_eq}
		\xymatrix{
			v_n\ar@{<-}[rr]^{e_n} \ar[d]^{f'} & & v_{n-1} \ar@{<-}[rr]^{e_{n-1}} & & & \cdots & {}\ar@{<-}[rr]^{e_1} & & v_0 \ar[d]^{f} \\
			v' \ar@{<-}[rrrrrrrr]_{e_0} & &  & & \ar@{}[u]|{\Downarrow \ \mathbf u}&  &  & & v
		}
	\end{equation}
	where $v_0,\dots, v_n, v, v'\in V$ are 0-cells, $e_1,\dots, e_n, e_0\in E$ are horizontal 1-cells, $f,f'\in\mathcal V$ are vertical 1-cells, and $\mathbf u\in\mathcal M$ is a 2-cell.
	Note that the edge-string $\vec e=(e_1,\dots, e_n)\in E^*$.

	\paragraph{$\bullet$} For the source and target maps $s,t:E\to V$ and $s^*,t^*:E^*\to V$, we have $s(e_0)=v$, $t(e_0)=v'$; $s^*(\vec e)=v_0$, $t^*(\vec e)=v_n$; $\hat s (\mathbf u)= f$, $\hat t(\mathbf u)= f'$ in the rectangle (\ref{2_cell_eq}).

	\paragraph{$\bullet$}  The maps $\bar d, \bar c$ of the vertices give a diagram $V \xleftarrow{\bar d} \mathcal V\xrightarrow{\bar c} V$.
	They satisfy $\bar d(f)=v_0$ and $\bar c(f)=v$ in (\ref{2_cell_eq}).
	Besides, composition and identity functions $\bar\gamma,\bar\iota$ (the vertex parts of $\pmb\gamma$ and $\pmb\iota$) make the 0-cells and vertical 1-cells into a category $(V,\mathcal V)$.

	\begin{defn}
		\label{vert_disc_fc_defn}
		A $\mathbf{fc}$-multicategory $\mathscr M=(V,E,\mathcal V,\mathcal M)$ is called \textbf{\textit{vertically discrete}} if all vertical 1-cells are identities \cite[Example 5.1.4]{leinster2004higher}. That is to say, the category formed by $V,\mathcal V, \bar d,\bar c$ is a discrete category, and thus  $V \equiv \mathcal V$. The absence of vertical edges allow us to draw the diagram \eqref{2_cell_eq} as in Figure \ref{fig:2_cell_vert_discrete}.
	\end{defn}

	\paragraph{$\bullet$} The maps $d,c$ of the edges give a diagram $E^* \xleftarrow{d} \mathcal M\xrightarrow{c} E$.
	In the rectangle (\ref{2_cell_eq}), we have $d(\mathbf u)=(e_1,e_2,\dots, e_n)\in E^*$ and $c(\mathbf u)=e_0\in E$.
	If we write
	$
	\mathcal M(\vec e ; e_0) = \{\mathbf u\in\mathcal M \mid d(\mathbf u)=\vec e \ , \  c(\mathbf u)= e_0 \}
	$
	for the \textbf{\textit{fiber}} of $\mathcal M$ over $(\vec e, e_0)$,
	then $\mathcal M$ admits the decomposition
	$
	\mathcal M = \coprod_{e_0\in E} \ \coprod_{\vec e \in E^*} \  \mathcal M(\vec e; e_0)
	$.
	For a vertically discrete $\mathbf{fc}$-multicategory, it can be further refined to
	\begin{equation}
		\label{M_decompose_profile_eq}
		\mathcal M= \coprod_{(\vec e; e_0)\in\mathcal P_E}  \mathcal M(\vec e; e_0)
	\end{equation}
	where $\mathcal P_E$ is the set of profile-loops in the graph $(V,E)$; see Definition \ref{profile_loop_defn}.

	\begin{figure}[h]
		\centering
		\begin{tikzpicture}[x=0.75pt,y=0.75pt,yscale=-0.7,xscale=0.7]
			
			\draw    (508.23,306.53) -- (136.4,305.57) ;
			\draw [shift={(134.4,305.56)}, rotate = 0.15] [color={rgb, 255:red, 0; green, 0; blue, 0 }  ][line width=0.75]    (10.93,-3.29) .. controls (6.95,-1.4) and (3.31,-0.3) .. (0,0) .. controls (3.31,0.3) and (6.95,1.4) .. (10.93,3.29)   ;
			\draw    (438,188) -- (263.95,146.46) ;
			\draw [shift={(262,146)}, rotate = 13.42] [color={rgb, 255:red, 0; green, 0; blue, 0 }  ][line width=0.75]    (10.93,-3.29) .. controls (6.95,-1.4) and (3.31,-0.3) .. (0,0) .. controls (3.31,0.3) and (6.95,1.4) .. (10.93,3.29)   ;
			\draw    (283,81) -- (250.14,128.36) ;
			\draw [shift={(249,130)}, rotate = 304.76] [color={rgb, 255:red, 0; green, 0; blue, 0 }  ][line width=0.75]    (10.93,-3.29) .. controls (6.95,-1.4) and (3.31,-0.3) .. (0,0) .. controls (3.31,0.3) and (6.95,1.4) .. (10.93,3.29)   ;
			\draw    (353,46) -- (293.85,70.24) ;
			\draw [shift={(292,71)}, rotate = 337.71] [color={rgb, 255:red, 0; green, 0; blue, 0 }  ][line width=0.75]    (10.93,-3.29) .. controls (6.95,-1.4) and (3.31,-0.3) .. (0,0) .. controls (3.31,0.3) and (6.95,1.4) .. (10.93,3.29)   ;
			\draw    (432,72) -- (361.88,46.68) ;
			\draw [shift={(360,46)}, rotate = 19.86] [color={rgb, 255:red, 0; green, 0; blue, 0 }  ][line width=0.75]    (10.93,-3.29) .. controls (6.95,-1.4) and (3.31,-0.3) .. (0,0) .. controls (3.31,0.3) and (6.95,1.4) .. (10.93,3.29)   ;
			\draw    (457,177) -- (438.38,81.96) ;
			\draw [shift={(438,80)}, rotate = 78.92] [color={rgb, 255:red, 0; green, 0; blue, 0 }  ][line width=0.75]    (10.93,-3.29) .. controls (6.95,-1.4) and (3.31,-0.3) .. (0,0) .. controls (3.31,0.3) and (6.95,1.4) .. (10.93,3.29)   ;
			\draw    (231,150) -- (178.8,175.13) ;
			\draw [shift={(177,176)}, rotate = 334.29] [color={rgb, 255:red, 0; green, 0; blue, 0 }  ][line width=0.75]    (10.93,-3.29) .. controls (6.95,-1.4) and (3.31,-0.3) .. (0,0) .. controls (3.31,0.3) and (6.95,1.4) .. (10.93,3.29)   ;
			\draw    (170,181) -- (120.83,289.18) ;
			\draw [shift={(120,291)}, rotate = 294.44] [color={rgb, 255:red, 0; green, 0; blue, 0 }  ][line width=0.75]    (10.93,-3.29) .. controls (6.95,-1.4) and (3.31,-0.3) .. (0,0) .. controls (3.31,0.3) and (6.95,1.4) .. (10.93,3.29)   ;
			\draw    (519,290) -- (468.06,208.69) ;
			\draw [shift={(467,207)}, rotate = 57.93] [color={rgb, 255:red, 0; green, 0; blue, 0 }  ][line width=0.75]    (10.93,-3.29) .. controls (6.95,-1.4) and (3.31,-0.3) .. (0,0) .. controls (3.31,0.3) and (6.95,1.4) .. (10.93,3.29)   ;
			
			\draw (109.27,300) node [anchor=north west][inner sep=0.75pt]    {$v_{n}$};
			\draw (517.22,300) node [anchor=north west][inner sep=0.75pt]    {$v_{0}$};
			\draw (239.29,140) node [anchor=north west][inner sep=0.75pt]    {$v_{i}$};
			\draw (445.11,188) node [anchor=north west][inner sep=0.75pt]    {$v_{i-1}$};
			\draw (309.26,311.98) node [anchor=north west][inner sep=0.75pt]    {$e_{0}$};
			\draw (308.6,230.07) node [anchor=north west][inner sep=0.75pt]    {$\mathbf{u}$};
			\draw (349.6,106.07) node [anchor=north west][inner sep=0.75pt]    {$\mathbf{u} '$};
			\draw (338.26,169.98) node [anchor=north west][inner sep=0.75pt]    {$e_{i}$};
			\draw (50,165.4) node [anchor=north west][inner sep=0.75pt]    {$\mathbf{u} \circ_i \mathbf u' =$};
		\end{tikzpicture}
		\caption{The partial composition of 2-cells $\mathbf u$ and $\mathbf u'$}
		\label{fig:2_cell_composition}
	\end{figure}
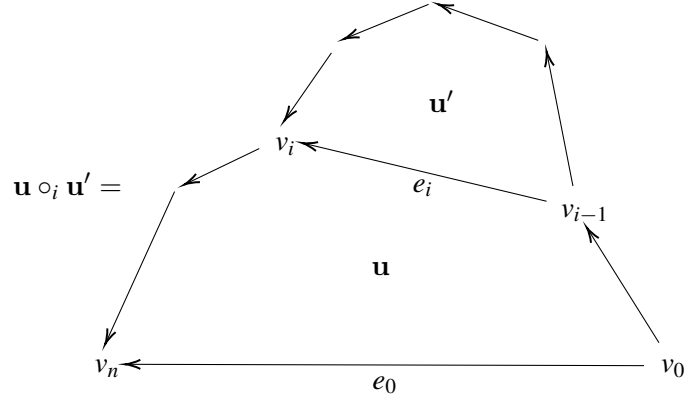

	\paragraph{$\bullet$} 
	Unwinding the definition of the composite graph $\mathcal M\circ \mathcal M$ in Definition~\ref{multicategory_defn}
	(for $T=\mathbf{fc}$), an edge of $\mathcal M\circ \mathcal M$ can be described as a tuple
	$(\mathbf u;\ \mathbf u_1,\dots,\mathbf u_n)$
	where $\mathbf u\in \mathcal M(\vec e;\,e_0)$ is a $2$-cell with
	$\vec e=(e_1,\dots,e_n)\in E^*$ and $e_0\in E$; for each $i=1,\dots,n$, one has a $2$-cell
	$\mathbf u_i\in \mathcal M(\vec e_i;\,e_i)$
	for some word $\vec e_i\in E^*$ so that $c(\mathbf u_i)=e_i$ matches the $i$-th input of $\mathbf u$.
	To describe the vertex set of $\mathcal M\circ \mathcal M$, recall that the vertices of $\mathcal M$ are vertical $1$-cells in $\mathcal V$. Then, a vertex of $\mathcal M\circ \mathcal M$
	is a pair
	$(f_1,f_2)\in \mathcal V\times_V \mathcal V$
	of composable vertical $1$-cells, i.e. $\bar c(g_1)=\bar d(g_2)$, representing the left and right vertical boundaries of a pasted rectangle.
	Given this, we can describe the multiplication $\pmb\gamma=(\bar\gamma,\gamma): \mathcal M\circ \mathcal M\to \mathcal M$ as follows: its vertex part $\bar\gamma$ is just he composition in the category $(V,\mathcal V)$ of vertical 1-cells, and its edge part is given by
	\[
	\quad \gamma:\ \mathcal M(\vec e;\,e_0)\ \times\ \bigtimes_{i=1}^n \mathcal M(\vec e_i;\,e_i)
	\ \longrightarrow\ \mathcal M(\vec e_1 \ast \cdots \ast \vec e_n;\,e_0),
	\quad  \ \ (\mathbf u_1,\dots,\mathbf u_n, \mathbf u)\mapsto \mathbf u\circ (\mathbf u_1,\dots, \mathbf u_n)
	\]
	where $\vec e=(e_1,\dots, e_n)\in E^*$ and $\vec e_1 \ast \cdots \ast \vec e_n\in E^*$ is the concatenation (\ref{concatenation_eq}).
	See also \cite[page 2]{leinster1999fc} and \cite[(5:2)]{leinster2004higher} for more detailed diagrams.
	The edge part $\iota$ of the unit $\pmb\iota=(\bar\iota ,\iota)$ assigns to every horizontal 1-cell $e$ an identity 2-cell 
	\[
	\id_e=\iota(e)
	\]
	with $\bar d(\id_e)=\bar c(\id_e)=e$.
	These maps all need to obey the corresponding associativity and identity laws. 
	Roughly, a diagram of pasted-together 2-cells with a rectangular boundary should have a well-defined composition outcome. 
	Moreover, the composition of 2-cells can be equivalently expressed in terms of certain partial compositions of the form 
	\[
	\circ_i: \mathcal M(e_1,\dots, e_n; e_0 ) \times \mathcal M(e_1',\dots, e_m'; e_i) \to \mathcal M(e_1,\dots, e_{i-1}, e_1', \dots, e_m', e_{i+1},\dots, e_n; e_0)
	\]
	where
	\[
	\mathbf u\circ_i\mathbf u' : =
	\mathbf u\circ (\id_{e_1}, \dots , \mathbf u',\dots, \id_{e_n})
	\]
	For a vertically discrete $\mathbf{fc}$-multicategory, this partial composition $\circ_i$ can be depicted as in Figure \ref{fig:2_cell_composition}.
	For the special case when $\mathbf u'$ is a 2-cell with empty input (cf. Definition \ref{profile_loop_defn}), the partial composition then looks like in Figure \ref{fig:2_cell_composition_curved}.
	\footnote{ \scriptsize  Intuitively, these pictures may likely remind us of the bubbling of obstructing pseudo-holomorphic disks in the studies of (curved) Lagrangian Floer theory and Fukaya category.}

	\begin{figure} 
		\centering
		\begin{tikzpicture}[x=0.75pt,y=0.75pt,yscale=-0.7,xscale=0.7]
			
			\draw    (350.11,249.73) -- (59.71,248.96) ;
			\draw [shift={(57.71,248.95)}, rotate = 0.15] [color={rgb, 255:red, 0; green, 0; blue, 0 }  ][line width=0.75]    (10.93,-3.29) .. controls (6.95,-1.4) and (3.31,-0.3) .. (0,0) .. controls (3.31,0.3) and (6.95,1.4) .. (10.93,3.29)   ;
			\draw    (295.18,155.28) -- (159.46,122.28) ;
			\draw [shift={(157.51,121.81)}, rotate = 13.66] [color={rgb, 255:red, 0; green, 0; blue, 0 }  ][line width=0.75]    (10.93,-3.29) .. controls (6.95,-1.4) and (3.31,-0.3) .. (0,0) .. controls (3.31,0.3) and (6.95,1.4) .. (10.93,3.29)   ;
			\draw    (133.27,125) -- (92.82,144.84) ;
			\draw [shift={(91.03,145.72)}, rotate = 333.87] [color={rgb, 255:red, 0; green, 0; blue, 0 }  ][line width=0.75]    (10.93,-3.29) .. controls (6.95,-1.4) and (3.31,-0.3) .. (0,0) .. controls (3.31,0.3) and (6.95,1.4) .. (10.93,3.29)   ;
			\draw    (85.55,149.7) -- (47.26,235.53) ;
			\draw [shift={(46.45,237.35)}, rotate = 294.05] [color={rgb, 255:red, 0; green, 0; blue, 0 }  ][line width=0.75]    (10.93,-3.29) .. controls (6.95,-1.4) and (3.31,-0.3) .. (0,0) .. controls (3.31,0.3) and (6.95,1.4) .. (10.93,3.29)   ;
			\draw    (358.53,236.56) -- (318.91,172.12) ;
			\draw [shift={(317.86,170.42)}, rotate = 58.41] [color={rgb, 255:red, 0; green, 0; blue, 0 }  ][line width=0.75]    (10.93,-3.29) .. controls (6.95,-1.4) and (3.31,-0.3) .. (0,0) .. controls (3.31,0.3) and (6.95,1.4) .. (10.93,3.29)   ;
			\draw    (244.34,49.3) .. controls (293.37,136.51) and (147.25,109.34) .. (225.92,47.06) ;
			\draw [shift={(227.13,46.11)}, rotate = 142.27] [color={rgb, 255:red, 0; green, 0; blue, 0 }  ][line width=0.75]    (10.93,-3.29) .. controls (6.95,-1.4) and (3.31,-0.3) .. (0,0) .. controls (3.31,0.3) and (6.95,1.4) .. (10.93,3.29)   ;
			\draw   (404,159.5) -- (435.2,159.5) -- (435.2,154) -- (456,165) -- (435.2,176) -- (435.2,170.5) -- (404,170.5) -- cycle ;
			\draw    (794.11,248.73) -- (503.71,247.96) ;
			\draw [shift={(501.71,247.95)}, rotate = 0.15] [color={rgb, 255:red, 0; green, 0; blue, 0 }  ][line width=0.75]    (10.93,-3.29) .. controls (6.95,-1.4) and (3.31,-0.3) .. (0,0) .. controls (3.31,0.3) and (6.95,1.4) .. (10.93,3.29)   ;
			\draw    (658,129) -- (537.01,144.46) ;
			\draw [shift={(535.03,144.72)}, rotate = 352.72] [color={rgb, 255:red, 0; green, 0; blue, 0 }  ][line width=0.75]    (10.93,-3.29) .. controls (6.95,-1.4) and (3.31,-0.3) .. (0,0) .. controls (3.31,0.3) and (6.95,1.4) .. (10.93,3.29)   ;
			\draw    (529.55,148.7) -- (491.26,234.53) ;
			\draw [shift={(490.45,236.35)}, rotate = 294.05] [color={rgb, 255:red, 0; green, 0; blue, 0 }  ][line width=0.75]    (10.93,-3.29) .. controls (6.95,-1.4) and (3.31,-0.3) .. (0,0) .. controls (3.31,0.3) and (6.95,1.4) .. (10.93,3.29)   ;
			\draw    (802.53,235.56) -- (686.5,133.32) ;
			\draw [shift={(685,132)}, rotate = 41.38] [color={rgb, 255:red, 0; green, 0; blue, 0 }  ][line width=0.75]    (10.93,-3.29) .. controls (6.95,-1.4) and (3.31,-0.3) .. (0,0) .. controls (3.31,0.3) and (6.95,1.4) .. (10.93,3.29)   ;
			
			\draw (35.98,237.41) node [anchor=north west][inner sep=0.75pt]    {};
			\draw (355.18,238.63) node [anchor=north west][inner sep=0.75pt]    {};
			\draw (138.44,111.46) node [anchor=north west][inner sep=0.75pt]    {$v$};
			\draw (299.43,148.25) node [anchor=north west][inner sep=0.75pt]    {$v$};
			\draw (192.55,187.26) node [anchor=north west][inner sep=0.75pt]    {$\mathbf{u}$};
			\draw (215.86,139.37) node [anchor=north west][inner sep=0.75pt]    {$e$};
			\draw (229.81,39.09) node [anchor=north west][inner sep=0.75pt]    {$v$};
			\draw (343.86,188.37) node [anchor=north west][inner sep=0.75pt]    {$e'$};
			\draw (101.86,112.37) node [anchor=north west][inner sep=0.75pt]    {$e''$};
			\draw (479.98,236.41) node [anchor=north west][inner sep=0.75pt]    {};
			\draw (799.18,237.63) node [anchor=north west][inner sep=0.75pt]    {};
			\draw (663,117) node [anchor=north west][inner sep=0.75pt]    {$v$};
			\draw (600,186.26) node [anchor=north west][inner sep=0.75pt]    {$\mathbf{u}\circ_i \mathbf u'$};
			\draw (758.86,169.37) node [anchor=north west][inner sep=0.75pt]    {$e'$};
			\draw (584.86,112.37) node [anchor=north west][inner sep=0.75pt]    {$e''$};
			\draw (221.55,68.26) node [anchor=north west][inner sep=0.75pt]    {$\mathbf{u} '$};
		\end{tikzpicture}
		\caption{The partial composition with $\mathbf u'$ of empty input resolves the horizontal 1-cell input and removes that slot.}
		\label{fig:2_cell_composition_curved}
	\end{figure}
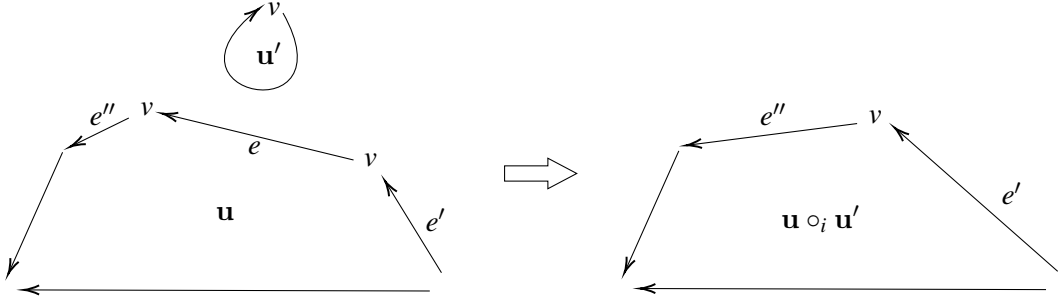

	\begin{ex}
		\label{suspension_ex}
		Recall that a (non-symmetric) operad is a multicategory with only one object. Further, by \cite[Example 5.1.6 \& 5.1.7]{leinster2004higher}, a multicategory $M$ can be viewed as a (vertically discrete) $\mathbf{fc}$-multicategory $\Sigma M$ as follows. We call $\Sigma M$ the \textit{suspension} of $M$. It only has a single 0-cell and a single vertical 1-cell; its horizontal 1-cells are objects of $M$.
		In this way, given objects $c_1,\dots, c_n,c_0$ of $M$, a multicategory composition $\theta$ from $(c_1,\dots, c_n)$ to $c$ can be viewed as a 2-cell in $\Sigma M$.
	\end{ex}


	\begin{ex}
		\label{profile_loop_ex}
		For a directed graph $E=(V,E,s,t)$, let $\mathcal P_E$ be the set of profile-loops as in Definition \ref{profile_loop_defn}.
		Then, the tuple $(V,E ,\mathcal P_E)$ forms a vertically discrete $\mathbf{fc}$-multicategory.
		Indeed, the 0-cells are the vertices $v\in V$; the horizontal 1-cells are the edges $e\in E$.
		The 2-cells are the set $\mathcal P_E$ of all profile-loops. (So, the set of 2-cells filling a specified rectangle as (\ref{2_cell_eq}) is a singleton.)
		The unit map $\pmb\iota=(\bar \iota,\iota): (V,E)\to (V, \mathcal P_E)$ is given by $\bar \iota=\id_V$ and $\iota(e)=(e;e) \in \mathcal P_E$.
		The multiplication map $\bar\gamma$ for the vertical 1-cells is evident.
		The multiplication map $\gamma$ for the 2-cells is given by the (partial) composition of two profile-loops (see also Figure \ref{fig:2_cell_composition}):
		\[
		\Big( (e_1,\dots, e_{n}; e_0)  \ , \ (e_1',\dots, e_{m}'; e_i) \Big) \mapsto 
		(e_1,\dots, e_{i-1}, e_1',\dots, e_{m}', e_{i+1},\dots, e_{}; e_0).
		\]
	\end{ex}

	\subsection{Maps of $\mathbf{fc}$-multicategories and factor-closedness}
	\label{s_map_fc_mul}
	
	Let
	$
	\mathscr M=(V,E,\mathcal V,\mathcal M,\pmb d,\pmb c,\pmb\iota,\pmb\gamma)$ and $
	\mathscr M'=(V',E',\mathcal V',\mathcal M',\pmb d',\pmb c',\pmb\iota',\pmb\gamma')$
	be $\mathbf{fc}$-multicategories, where
	\[
	E=(V,E,s,t),\quad \mathbf{fc}(E)=(V,E^*,s^*,t^*),\quad \mathcal M=(\mathcal V,\mathcal M,\hat s,\hat t),
	\]
	are directed graphs, and similarly for $\mathscr M'$.

	\begin{defn} 
		\label{map_fc_defn}
		A \textbf{\emph{map of $\mathbf{fc}$-multicategories}}
		\[
		\pmb \alpha = (\alpha_E, \alpha ) =(\alpha_V, \alpha_E,  \alpha_{\mathcal V},  \alpha):\ \mathscr M \longrightarrow \mathscr M'
		\]
		consists of maps of directed graphs (here we have slightly abused the notations)
		\[
		\alpha_E=	(\alpha_V, \alpha_E):\ (V,E)\to (V',E'),
		\qquad
		\alpha =	(\alpha_{\mathcal V}, \alpha ):\ (\mathcal V,\mathcal M)\to (\mathcal V',\mathcal M'),
		\]
		such that the following three compatibilities hold.
		
		\begin{enumerate}
			\itemsep 2pt
			\item  $
			\alpha \circ \pmb\iota =  \pmb\iota'\circ \alpha_E
			$.
			\item  One requires
			$
			\pmb c'\circ \alpha = \alpha_E\circ \pmb c$ and $
			\pmb d'\circ \alpha = \mathbf{fc}(\alpha_E)\circ \pmb d$
			where $\mathbf{fc}(\alpha_E):\mathbf{fc}(E)\to \mathbf{fc}(E')$ is the induced map of directed graphs
			sending $(e_1,\dots,e_n)\in E^*$ to $(\alpha_E(e_1),\dots,\alpha_E(e_n))\in (E')^*$.
			\item  Let
			$\alpha\circ \alpha
			:\ \mathcal M\circ \mathcal M \longrightarrow \mathcal M'\circ \mathcal M'$
			denote the map induced on the composite graph. Then one requires
			$
			\alpha \circ \pmb\gamma=
			\pmb\gamma'\circ ( \alpha \circ \alpha)
			$, that is, the following diagram commutes:
			\[
			\xymatrix{
				\mathcal M\circ \mathcal M \ar[rr]^{\alpha\circ \alpha}
				\ar[d]_{\pmb\gamma} & &
				\mathcal M'\circ \mathcal M' \ar[d]^{\pmb\gamma'} \\
				\mathcal M \ar[rr]_{\alpha} & & \mathcal M',
			}
			\]
		\end{enumerate}
	\end{defn}
	
	Recall that an element of $\mathcal M\circ \mathcal M$ may be represented by pasting data $
	(\mathbf u;\mathbf u_1,\dots,\mathbf u_n)$
	where $\mathbf u\in \mathcal M(\vec e;e_0)$ with $\vec e=(e_1,\dots,e_n)\in E^*$, and
	$\mathbf u_i\in \mathcal M(\vec e_i;e_i)$ for $1\le i\le n$.
	The above third condition
	means precisely that applying $\alpha$ to the composite $2$-cell equals composing after applying $\alpha$
	to each constituent $2$-cell; in other words,
	$
	\alpha (\gamma(\mathbf u_1,\dots,\mathbf u_n,\mathbf u) )
	=
	\gamma' (\alpha(\mathbf u_1),\dots, \alpha(\mathbf u_n), \alpha(\mathbf u) )$.
	Equivalently, for the partial composition $\circ_i$, one gets
	\[
	\alpha(\mathbf u\circ_i \mathbf u')
	=
	\alpha (\mathbf u)\circ_i \alpha (\mathbf u').
	\]

	In the context of Example \ref{suspension_ex}, one can check that a map of multicategories $M\to M'$ (and in particular a map of operads) is equivalently a map of the associated
	$\mathbf{fc}$-multicategories $\Sigma M\to \Sigma M'$.

	\begin{defn}
		\label{fc_submul_defn}
		A \textbf{\emph{$\mathbf{fc}$-submulticategory}} of $\mathscr M$ is the data of subsets
		\[
		V_0\subset V,\qquad E_0\subset E,\qquad \mathcal V_0\subset \mathcal V,\qquad \mathcal M_0\subset \mathcal M,
		\]
		such that the restricted structure maps make
		$
		\mathscr M_0:=(V_0,E_0,\mathcal V_0,\mathcal M_0)
		$
		into a $\mathbf{fc}$-multicategory, and the inclusion maps assemble to a map of $\mathbf{fc}$-multicategories
		$\mathscr M_0  \to  \mathscr M$.
		We say that a $\mathbf{fc}$-subcategory $\mathscr M_0$ of $\mathscr M$ is \textbf{\emph{factor-closed}} in $\mathscr M$ if for $2$-cells
		$\mathbf u\in \mathcal M(e_1,\dots,e_n;\,e_0)$ and
		$\mathbf u'\in \mathcal M(e_1',\dots,e_m';\,e_i)$
		such that the partial composition $\mathbf u\circ_i \mathbf u'$ is defined, one has
		\[
		\mathbf u\circ_i \mathbf u'\in \mathcal M_0
		\quad\Longrightarrow\quad
		\mathbf u\in \mathcal M_0\ \text{ and }\ \mathbf u'\in \mathcal M_0.
		\]
	\end{defn}

\begin{defn}
\label{full_fc_sub_defn}
A $\mathbf{fc}$-submulticategory
	$\mathscr M_0=(V_0,E_0,\mathcal V_0,\mathcal M_0)$ of $\mathscr M$ is called \textit{\textbf{full}} if the following hold:
\begin{enumerate}
		\item For any $v,w\in V_0$, the set of vertical 1-cells in $\mathcal V_0$ from $v$ to $w$ coincide with the set of vertical 1-cells in $\mathcal V$ from $v$ to $w$.
		In other words, $(V_0,\mathcal V_0)$ is the full subcategory of the vertical category $(V,\mathcal V)$ on the object set $V_0$.
		\item For every $(\vec e;e')\in E_0^*\times E_0$, $\mathcal M_0(\vec e;e') =\{\mathbf u\in \mathcal M(\vec e;e')\mid \hat s(\mathbf u),\hat t(\mathbf u)\in \mathcal V_0\}$. In particular, if $\mathscr M$ is vertically discrete, then it says that for every profile-loop $(\vec e, e')$, we have $\mathcal M_0(\vec e; e')=\mathcal M(\vec e; e')$.
	\end{enumerate}
Note that $E_0=(V_0,E_0)$ is a \textit{directed sub-graph} of $E=(V,E)$, meaning that $V_0\subset V$, $E_0\subset E$, and $s(E_0),t(E_0)\subset V_0$. Then, we say that $\mathscr M_0$ is \textit{\textbf{the full $\mathbf{fc}$-submulticategory of $\mathscr M$ on $E_0$}}.
\end{defn}

\begin{defn}\label{endpoint_closed_defn}
	A directed subgraph $E_0=(V_0,E_0)$ of $E=(V,E)$ is called \emph{\textbf{endpoint-closed}} in $E$ if the following holds:
	for any profile-loop $(\vec e;e')\in E^*\times E$ in $E$ (Definition~\ref{profile_loop_defn}), if
	$\vec e\in E_0^*$, then $e'\in E_0$.
\end{defn}

For later use, we introduce a sufficient criterion for the factor-closedness of a $\mathbf{fc}$-submulticategory.
	Let $\mathscr M=(V,E,\mathcal M,d,c,\iota,\gamma)$ be a vertically discrete $\mathbf{fc}$-multicategory.
Let $E_0=(V_0,E_0)$ be a directed subgraph of $E=(V,E)$, and let $\mathscr M_0\subset \mathscr M$ be the
\emph{full} $\mathbf{fc}$-submulticategory on $E_0$.

\begin{prop}\label{prop:endpoint_closed_factor_closed}
If $E_0$ is \textit{endpoint-closed} in $E$, then $\mathscr M_0$ is factor-closed in $\mathscr M$.
\end{prop}

\begin{proof}
Let $\mathbf u\in\mathcal M(e_1,\dots,e_n;\,e_0)$ and
	$\mathbf u'\in\mathcal M(e_1',\dots,e_m';\,e_i)$ be such that $\mathbf u\circ_i\mathbf u'$ is defined and lies in $\mathcal M_0$.
Our goal is to show $\mathbf u$ and $\mathbf u'$ lie in $\mathcal M_0$ as well.
Applying the codomain and domain maps, we get $e_0=c(\mathbf u\circ_i\mathbf u')\in E_0$ and $d(\mathbf u\circ_i\mathbf u')=(e_1,\dots,e_{i-1},e_1',\dots,e_m',e_{i+1},\dots,e_n)\in E_0^*$.
	In particular, $(e_1',\dots,e_m')\in E_0^*$, so letting
	\[
	v:=s^*(e_1',\dots,e_m'),\qquad w:=t^*(e_1',\dots,e_m'),
	\]
	we have $v,w\in V_0$ and the set $E_0^*(v,w)$ of edges from $v$ to $w$ is nonempty.
Since $\mathbf u'\in\mathcal M(e_1',\dots,e_m';\,e_i)$ and $\mathbf u\circ_i \mathbf u'$ can be defined, the edge $e_i$ has endpoints $s(e_i)=v$ and $t(e_i)=w$.
By the endpoint-closedness condition for $E_0$, we have $e_i\in E_0$.
Therefore $d(\mathbf u)=(e_1,\dots,e_n)\in E_0^*$ and $c(\mathbf u)=e_0\in E_0$, hence $\mathbf u\in\mathcal M_0$.
	Also $d(\mathbf u')=(e_1',\dots,e_m')\in E_0^*$ and $c(\mathbf u')=e_i\in E_0$, hence $\mathbf u'\in\mathcal M_0$.
	This proves that $\mathscr M_0$ is factor-closed.
\end{proof}

	\subsection{Labeling by $\mathbf{fc}$-multicategories}
	\label{s_labeling_fc}
	
	Similar to the ideas in Section \ref{s_additional_grading}, we may introduce the labeling in the context of $\mathbf{fc}$-multicategories as well.

	\begin{defn}
		\label{label_by_fc_multicategory_defn}
		Let $\mathbb S$ be a fixed $\mathbf{fc}$-multicategory. A $\mathbf{fc}$-multicategory
		$
		\mathscr M=(V,E,\mathcal V,\mathcal M)
		$
		is called \textbf{\textit{$\mathbb S$-labeled}} if it is equipped with a map of $\mathbf{fc}$-multicategories
		$
		|\cdot|:\ \mathscr M \longrightarrow \mathbb S
		$
		whose components on $0$-cells, vertical $1$-cells, and horizontal $1$-cells are the identity maps.
	\end{defn}

In other words,
	$|\cdot|$ acts trivially on $V$, $\mathcal V$, and $E$, and only records the label of each $2$-cell of $\mathcal M$ in $\mathbb S$.
	Intuitively, by considering the fiber of the map $|\cdot|$, we can have finer decomposition for the 2-cells of $\mathcal M$.
	Indeed, the set $\mathcal M (\vec e; e')$ of 2-cells over $(\vec e, e') \in E^*\times E$ can be further decomposed into 
	\[ \mathcal M (\vec e; e' ; \beta)= \{\mathbf u \in \mathcal M(\vec e; e') \mid |\mathbf u|=\beta\}
	\]
	where $\beta$ runs over the corresponding set $\mathbb S(\vec e; e')$ of 2-cells in $\mathbb S$.
Remark that only a few $\mathcal M(\vec e; e'; \beta)$ are non-empty. We do not allow a map from a non-empty set to an empty set, so $\mathbb S(\vec e; e') = \varnothing$ implies $\mathcal M(\vec e; e') =\varnothing$.
However, it is possible that $\mathbb S(\vec e; e')\neq \varnothing$ when $\mathcal M(\vec e; e')=\varnothing$.

	\begin{ex}
		\label{ex_S_E_labeling}
		The above notion is a generalization of Definition \ref{S_labeled_multicategory_defn}.
		Let $(S,+,\theta)$ be a monoid, with identity element denoted by $\theta$.
		Let $E=(V,E,s,t)$ be a fixed directed graph.
		We can build a vertically discrete $\mathbf{fc}$-multicategory
		\[
		S_E \ := \  \left( V \ , \  E \ , \  \bigsqcup_{(\vec e; e') \ \text{profile-loop}} S  \right )
		\]
		whose set of $0$-cells is $V$, whose set of horizontal $1$-cells is $E$, and whose $2$-cells are uniformly labeled by $S$,
		in the sense that for each profile-loop $(\vec e; e') \in E^*\times E$ (Definition \ref{profile_loop_defn}), we set
		$
		S_E (\vec e; e'):=S,
		$
		and otherwise $S_E (\vec e; e'):=\varnothing$.
		For each horizontal $1$-cell $e\in E$, the corresponding identity $2$-cell $\id_e$ in $S_E(e;e)=S$ is defined to be the unit element of the monoid $S$, that is, $\id_e :=\theta \in S$. 
		The partial composition map 
		$
		\circ_i: S_E(e_1,\dots, e_n; e')\times S_E(f_1,\dots, f_m; e_i)
		\to 
		S_E(e_1,\dots, e_{i-1}, f_1,\dots, f_m , e_{i+1},\dots, e_n; e')
		$
		is defined by
		$\beta \circ_i \beta' := \beta+\beta'$.
		The associativity and unit axioms reduce to the associativity of $+$ and the fact that $\theta$ is the identity of $S$.
		When $E$ is the trivial directed graph, $S_E$ can be naturally identified with the (non-symmetric) operad $\{S(n)\}_{n\ge 0}$ given by
		$S(n):=S$ for each $n\ge 0$ with operadic unit given by $\theta\in S(1)=S$.
		When $E$ is a directed graph with a single vertex, $S_E$ can be identified with a multicategory whose set of objects is the edge set in an analogous manner.
		By Example~\ref{suspension_ex}, operads and multicategories arise as special cases of $\mathbf{fc}$-multicategories. In these two situations, the notion of an $S_E$-labeled $\mathbf{fc}$-multicategory reduces to the notion of an $S$-labeled operad/multicategory in Definition~\ref{S_labeled_multicategory_defn}.
	\end{ex}

\begin{ex}
\label{ex_S_E_labeling_reduced}
We also consider a variant of Example \ref{ex_S_E_labeling}.
Recall that $(\varnothing; e')$ is also regarded as a profile-loop (Definition \ref{profile_loop_defn}). Let's introduce
\[
S_E^{red} \ := \  \left( V \ , \  E \ , \  \bigsqcup_{(\vec e; e')  \ , \ \vec e\neq \varnothing} S  \sqcup \bigsqcup_{(\varnothing; e')} S\setminus \{\theta\}  \right )
\]
obtained by removing the identity elements $\theta$ in $S_E(\varnothing; e')$ for all $e'\in E$.
Since the composition maps $\circ_i$ will never produce outputs in $S_E(\varnothing; e')$, one can verify that $S_E^{red}$ is also a $\mathbf{fc}$-multicategory.
\end{ex}

\begin{ex}
		\label{ex_trivial_labeling}
		Every $\mathbf{fc}$-multicategory $\mathscr M=(V,E,\mathcal V,\mathcal M)$ is labeled by the $\mathbf{fc}$-multicategory obtained by collapsing each nonempty $\mathcal M(\vec e; e')$ to a point.
		Indeed, we define $\mathbb S$ so that, for each pair $(\vec e; e')\in E^*\times E$, the set $\mathbb S(\vec e; e')$ is a singleton whenever $\mathcal M(\vec e; e')$ is nonempty, and is empty whenever $\mathcal M(\vec e; e')$ is empty.
In particular, every \emph{vertically discrete} $\mathbf{fc}$-multicategory $\mathscr M$ is naturally labeled by $\pmb 0_E$ defined in Example \ref{ex_S_E_labeling}, where $S$ is the trivial monoid and $E=(V,E)$ is the underlying directed graph of 0-cells and horizontal 1-cells of $\mathscr M$.
Moreover, if $\mathcal M(\varnothing ; e')$ is always empty, then $\mathscr M$ is labeled by $\pmb 0_E^{red}$ (Example \ref{ex_S_E_labeling_reduced}).
	\end{ex}

	\section{Moduli spaces of pseudo-holomorphic polygons as $\mathbf{fc}$-multicategories}
	\label{s_moduli_fc_multicategories}
	
	Let's resume the context in Section \ref{s_moduli_disk}.
	It is more subtle if $\iota:L\to X$ is a Lagrangian \textit{immersion}; cf.~\cite{FuUnobstructed}. 
	In this case, a pseudo-holomorphic disk with boundary on $\iota(L)$ is more appropriately called \emph{a pseudo-holomorphic polygon}, since self-intersection points may become ``corners'' along the boundary. 
	Nevertheless, we occasionally use the two terms interchangeably if the context is clear.
	In this setting, we assume that $\iota$ has \textit{clean self-intersection}: namely, we require that the diagonal 
	\[
	\Delta_L=\{(p,p)\mid p\in L\}
	\]
	in the fiber product
	\[
	L \times_X L \;=\; \{(p,q)\in L\times L \mid \iota(p)=\iota(q)\}.
	\]
	is a union of connected components of $L\times_X L$; each other connected component of $L\times_X L$ is a smooth submanifold of $L\times L$; at every $(p,q)\in L\times_X L$, the tangent space satisfies
	\[
	T_{(p,q)}(L\times_X L)
	\;=\; \{(v,w)\in T_pL \times T_qL \mid d\iota_p(v)=d\iota_q(w)\}.
	\]

	\subsection{Pseudo-holomorphic polygons and motivations}
	The definition of the moduli spaces of pseudo-holomorphic stable disks needs to be modified as follows (cf. \cite[Definition 3.17]{FuUnobstructed}).
	Comparing Definition \ref{moduli_embed_defn}, a \textit{$J$-holomorphic stable polygon of type $k$} is defined as the data
	\begin{equation}
		\label{pseudo_holo_polygon_tuple}
		\mathbf u= (\Sigma, z_0,z_1,\dots, z_k,  u ,\gamma)
	\end{equation}
	where $\Sigma, z_j, u$ are defined in the same way except that the Lagrangian boundary condition is now $u(\partial\Sigma)\subset \iota(L)$, and where
	\begin{equation}
		\label{gamma_boundary_lift_eq}
		\gamma: \partial \Sigma \setminus \{z_0,z_1,\dots, z_k, \ \text{all boundary nodal points} \} \to L
	\end{equation}
	is an \textit{extra} continuous map such that $\iota\circ \gamma = u$. This condition in fact implies that $\gamma$ is \emph{relatively continuous} at each nodal point, in the following sense:  
	let $\sigma:(-\epsilon,\epsilon)\to \partial \Sigma$ be a continuous path respecting the induced boundary orientation on $\partial\Sigma$ with $\sigma(0)=0$.  
	If $\sigma$ meets two distinct irreducible components of $\Sigma$, then we require that $\gamma\circ \sigma:(-\epsilon,0)\cup (0,\epsilon) \to L$ continuously extends to $0$.
	In other words, the datum $\gamma$ is equivalent to specifying $k+1$ paths in $\partial\Sigma$ together with $k+1$ paths in $L$,  
	\begin{equation}
		\label{gamma_j_lift_eq}
		\sigma_j:(-\infty,+\infty)\to \partial\Sigma,\qquad 
		\gamma_j:(-\infty,+\infty)\to L,\qquad 0\le j\le k,
	\end{equation}
	such that each $\sigma_j$ respects the boundary orientation, connects the consecutive marked points with $\sigma_j(-\infty)=z_j$ and $\sigma_j(+\infty)=z_{j+1}$, and satisfies the compatibility condition
	$
	\iota \circ \gamma_j \;=\; u\circ \sigma_j
	$.
	
	The isomorphism between $\mathbf u$ and $\mathbf u'$ is still a biholomorphism $\phi:\Sigma\to \Sigma'$ except we further require $\gamma =\gamma'\circ \phi$.
	Similar to Definition \ref{moduli_embed_defn}, we define $\mathcal M (k)$ to be the set of isomorphism classes $[\mathbf u]$'s.
	However, its evaluation maps
	\[
	\ev_0,\ev_i \  (1\le i \le k): \mathcal M (k)\to L\times_X L 
	\]
	are then mapped into $L\times_X L$ (instead of $L$) and are modified to be
	\begin{equation}
		\label{evaluation_gamma_case_Eq}
		\begin{aligned}
			\ev_0 (\mathbf u) &= \big( \gamma(z_0+), \gamma(z_0-) \big) = \big(\gamma_0(-\infty), \gamma_k(+\infty) \big) \\
			\ev_i (\mathbf u) &= \big(\gamma(z_i-), \gamma(z_i+) \big)	 =  \big( \gamma_{i-1}(+\infty), \gamma_i(-\infty) \big) \qquad 1\le i\le k
		\end{aligned}
	\end{equation}
	where
	$\gamma(z_j\pm)=\lim_{{z\to z_j\pm , \ z\in\partial\Sigma}} \gamma(z)$ in $L$ and $z \to z_j{+}$ (resp.\ $z \to z_j{-}$) means that $z$ approaches $z_j$ along a path which reverses (resp.\ preserves) the induced boundary orientation on $\partial \Sigma$.
	See e.g. \cite[\S 3]{FuUnobstructed} for more details.
	Remark that compared to other $\ev_i$'s, the first and second components of $\ev_0$ are exchanged as $\ev_0$ stands for the ``output''.

	The monoid $S_L$ in (\ref{S_L_eq}) is still defined in the same way, but within $H_2(X, \iota(L))$. One may similarly define $\mathcal M(k,\beta)$ with $\beta\in S_L$.
	Following Fukaya's framework in \cite{FuUnobstructed}, and repeating the arguments of Proposition~\ref{multicategory_moduli_prop} almost verbatim with $L$ replaced by $L\times_X L$ and with the evaluation maps changed to the form in \eqref{evaluation_gamma_case_Eq}, we can similarly obtain:

	\begin{prop}
		\label{multicategory_moduli_immersed_prop}
		The moduli spaces $\mathcal M = \coprod_{\substack{k\ge 0}} \mathcal M(k)=\coprod_{\substack{k\ge 0, \beta\in S_L}} \mathcal M(k,\beta) $ and the evaluation maps
		\begin{equation}
			\label{diagram_eq}
			\xymatrix{
				& \mathcal M  \ar[dr]^{c=\ev_0}\ar[dl]_{d=(\ev_i)}& \\
				T(L\times_X L) & & L \times_X L
			}
		\end{equation}
		form an $S_L$-labeled multicategory on $L$. Here $T$ is the free monoid monad (\ref{list_monad_eq}). 
	\end{prop}

	However, this is not the whole picture: when $\iota:L\to X$ is an immersion (rather than an embedding), finer structure may emerge. 
	Observe that $L\times_X L$ naturally forms a directed graph $(L,\,L\times_X L;\,s,t)$ with vertex set $L$, edge set $L\times_X L$, and source/target maps \(s=\pr_1\), \(t=\pr_2\). This may remind us of the discussions in Section~\ref{s_fc_directedgraphs} on directed graphs and $\mathbf{fc}$-multicategories. It seems reasonable to extend this graph $(L,\,L\times_X L)$ to a vertically discrete $\mathbf{fc}$-multicategory by appropriately incorporating the moduli spaces $\mathcal M$. 
	A closer look will show that this is feasible, provided we take the 0-cells to be $\pi_0(L)$ rather than $L$: a pseudo-holomorphic polygon $\mathbf u$ in $\mathcal M$ carries boundary paths $\gamma_j\subset L$ between adjacent evaluation maps (cf. \eqref{gamma_j_lift_eq}), so for adjacent horizontal 1-cells we can only guarantee that their endpoints lie in the same path-connected component of $L$, not that they coincide.

\subsection{$\mathbf{fc}$-multicategory structure on moduli spaces }
	The domain $L$ of the immersion $i$ is not necessarily connected (cf.~\cite[Remark~3.3]{FuUnobstructed}).
	Each boundary path $\gamma_j$ is contained in a single connected component of $L$. 
	Accordingly, let's decompose $L$ into its connected components:
	\begin{equation}
		\label{connected_components_L_eq}
		L=\coprod_{v\in V_\iota} L_v \ , \qquad \ \ V_\iota:=\pi_0(L)
	\end{equation}
Define $E_\iota= L\times_X L$.
Denote by $\rho_\iota$ the canonical map $L\to V_\iota$ sending points in $L_v$ to $v$.
Define
\[
s_\iota, t_\iota : E_\iota \xlongrightarrow{\pr_1,\pr_2} L \xlongrightarrow{\rho_\iota} V_\iota 
\]
where $\pr_1,\pr_2$ are the natural projection maps to the first and second factors.
In this way, we obtain a directed graph
	\begin{equation}
		\label{E_L_graph_eq}
		E_\iota:=(V_\iota, E_\iota,s_\iota,t_\iota)
	\end{equation}
For every $v,v'\in V_\iota$, the edges from $v$ to $v'$ are $L_v\times_X L_{v'}$.
In the special case where$ \iota$ is a Lagrangian embedding and $L$ is connected, the vertex set consists of a single element, and the edge set $E_\iota = L\times_X L$ identifies canonically with $L$.

Denote by $\mathcal P_\iota$ the set of all profile-loops in $E_\iota=(V_\iota, E_\iota)$ as in Definition \ref{profile_loop_defn}.
For $e_0, e_1,\dots, e_n\in E_\iota$, we define $\mathcal M(e_1,\dots, e_n; e_0 ) $ to be the subset of $\mathcal M(n)$ consisting of pseudo-holomorphic polygons $\mathbf u$ with $\ev_0(\mathbf u)= e_0$ and $\ev_i(\mathbf u) = e_i$ for $i=1,\dots, n$.

\begin{lem}
		\label{M_pseudo_loop_lem}
If $\mathcal M(e_1,\dots, e_n; e_0)$ is non-empty, then $(e_1,\dots, e_n; e_0)\in \mathcal P_\iota$
	\end{lem}
	
\begin{proof}
Assume it is non-empty, and we pick an element $[\mathbf u]=[\Sigma, z_0,\dots, z_n, u, \gamma]$ within it.
Recall that $\gamma$ can induce continuous paths $\gamma_j$ in $L$ for $0\le j\le n$ as in (\ref{gamma_j_lift_eq}).
Each $\gamma_j$ must be in a connected component of $L$, so by (\ref{connected_components_L_eq}) $\gamma_j\subset L_{  v_j}$ for some $  v_j\in V_\iota$.
By (\ref{evaluation_gamma_case_Eq}), we conclude that $\ev_0(\mathbf u) \in L_{  v_0} \times_X L_{  v_n}$ and $\ev_i(\mathbf u) \in L_{  v_{i-1}}\times_X L_{  v_i}$ for $i\neq 0$.
Namely, $e_0$ is an edge from $v_0$ to $v_n$, and for $i\neq 0$, the edge $e_i$ goes from $v_{i-1}$ to $v_i$.
Thus, $\vec e=(e_1,\dots, e_n) \in E_\iota^*$ by \eqref{E*_eq}, and $(\vec e; e_0)$ is a profile loop in $E_\iota$.
\end{proof}

It follows that we can write
	\begin{equation}
		\label{moduli_decompose_eq}
		\mathcal M=  \coprod_{(e_1,\dots, e_n; e_0)\in\mathcal P_\iota} \ \mathcal M(e_1,\dots, e_n; e_0)
	\end{equation}

\begin{thm}
		\label{fc_moduli_unlabeled_thm}
		Given the Lagrangian immersion \(\iota:L\to X\), we can canonically associate a vertically discrete \(\mathbf{fc}\)-multicategory
		\[
		\mathscr M_\iota =\big( \  \pi_0(L) \ , \ L\times_X L   \ , \ \mathcal M \ \big)
		\]
		where the set of \(0\)-cells is $V_\iota=\pi_0(L)$; the set of horizontal \(1\)-cells is $E_\iota=L \times_X L$;
		and \(2\)-cells are elements in the moduli spaces of pseudo-holomorphic polygons \(\mathcal M\).
	\end{thm}
	
	\begin{proof}
This is essentially achieved in the same way as the proof of Proposition \ref{multicategory_moduli_prop} but with some novel viewpoint for the notion of $\mathbf{fc}$-multicategories.
The underlying graph of 0-cells and horizontal 1-cells is exactly the above $E_\iota=(V_\iota, E_\iota)$.
As for the 2-cells, the source map sends $\mathbf u\in\mathcal M$ to $(\ev_1(\mathbf u),\dots,\ev_n(\mathbf u))\in E_\iota^*$, and the target map sends $\mathbf u$ to the horizontal 1-cell $\ev_0(\mathbf u)\in E_\iota$.
	\end{proof}

	\subsection{Second relative singular homology revisited}
	\label{s_variant_H_2(X,L)}

	Since operads and multicategories are special cases of $\mathbf{fc}$-multicategories, the additional labeling for a \textit{monoid} as in Section~\ref{s_additional_grading} should extend to the $\mathbf{fc}$-multicategory setting of Theorem~\ref{fc_moduli_unlabeled_thm}. For instance, let $(S,+,\theta)$ be a commutative monoid; we call a $\mathbf{fc}$-multicategory $C=(V,E,\mathcal V,\mathcal M)$ \emph{$S$-labeled} if there is a labeling map
	$|\cdot|\colon \mathcal M \to S$
	such that $|\iota(e)|=\theta$ for each $e\in E$ and
	$|\mathbf u\circ_i \mathbf u’| \;=\; |\mathbf u|+|\mathbf u’|$ for composable $2$-cells $\mathbf u,\mathbf u’$.
	In this context, if we imitate \eqref{S_L_eq} to define $S_L'=\{\theta\} \cup \{\beta\in H_2(X,\iota(L)) \mid \omega(\beta)>0\}
	$ with the zero element $\theta$ in $H_2(X,\iota(L))$, then one can readily show that the $\mathbf{fc}$-multicategories in Theorem~\ref{fc_moduli_unlabeled_thm} are actually $S_L'$-labeled.
	
	Nevertheless, this formulation is not entirely satisfactory. 
	Retaining only the relative homologous class $[\mathbf u]$ in $S_L\subset H_2(X,\iota(L))$ discards the finer boundary-path information $\{\gamma_j\}$.
	As noted in Section~\ref{s_additional_grading}, an additional labeling is useful for symplectic applications; in particular, a labeling structure finer than the monoid structure $S_L\subset H_2(X,\iota(L))$ may be useful.
	
%
%
	
	Accordingly, we would like to follow Definition \ref{label_by_fc_multicategory_defn} and study labeled $\mathbf{fc}$-multicategories.

	If $\iota:L\to X$ is an embedding, then the relative singular chain complex $C_\bullet (X,\iota(L))$ is quasi-isomorphic to the mapping cone $\mathrm{Cone}_\bullet(\iota)$ of $\iota_*:C_\bullet(L)\to C_\bullet(X)$, where
	$
	\mathrm{Cone}_k(\iota) = C_k (X) \oplus C_{k-1}(L) 
	$ with the differential given by $(u, \gamma)\mapsto (\partial u - \iota_*(\gamma) \ , \ \partial \gamma )$.
	However, if $\iota : L\to X$ is not an embedding, then the quasi-isomorphism does not hold in general; so, this discussion may suggest that merely considering $H_2(X,\iota(L))$ could lose information.
	Indeed, let's consider
	\[
	\mathrm{Cone}_2(\iota) = C_2 (X) \oplus C_{1}(L) =  C_2(X)
	\oplus \bigoplus_{v\in V_L} C_1(L_v) \]
	by (\ref{connected_components_L_eq}).
	Then, given a pseudo-holomorphic polygon $\mathbf u$, instead of recording it in the degree-two relative singular chain group $C_2(X,\iota(L))$, it seems preferable to encode it in the degree-two mapping-cone group $\mathrm{Cone}_2(\iota)$, whose $C_1(L_v)$ components leave room to encode the aforementioned boundary paths $\gamma_j$ associated to $\mathbf u$.
	Developing further these ideas, we propose the following.
	

	Suppose $\iota: L\to X$ is a Lagrangian immersion as before.
	We construct a vertically discrete $\mathbf{fc}$-multicategory, denoted by
	\begin{equation}
		\label{S_iota_eq}
		\mathbb S_\iota = \mathbb S(X, \iota) = (\pi_0(L), L\times_X L, \mathbb S_\iota)
	\end{equation}
	as follows:

	\paragraph{\underline{0-cells and horizontal 1-cells}:}
	The set of 0-cells is defined to be $V_\iota=\pi_0(L)$.
	The set of horizontal 1-cells is defined to be the set of points in $E_\iota=L\times_X L $.
	In other words, the underlying directed graph is $E_\iota (V_\iota, E_\iota, s_\iota, t_\iota)$ in (\ref{E_L_graph_eq}).

	\paragraph{\underline{General case for 2-cells}:}
	Let $(e_1,\dots, e_n; e_0)$ be a profile-loop in $E_\iota$.
	We assume that $e_i$ is an edge from $v_{i-1}$ to $v_i$ for $i=1,\dots, n$ and $e_0$ is an edge from $v_0$ to $v_n$ (cf. Figure \ref{fig:2_cell_composition}). 
	Denote the first and second projection maps $L \times_X L \to L$ by $s=\pr_1$ and $t=\pr_2$ respectively, and observe that 
	$
	\jmath:= s\circ  \iota = t\circ \iota: L\times_X L \to L
	$.

Fix horizontal 1-cells $q_i  \in L \times_X L$ for $i=0,1\dots, n$.
Denote by $Z_\iota(q_1,\dots, q_n; q_0)$ the subset of
	\[
	C_2(X)   \oplus \bigoplus_{i=0}^n C_1(L_{v_i})
	\]
	consisting of 
	\[\left( u  \ , \  (\gamma_i)_{i=0}^n 
	\right) \]
	subject to
	\begin{equation}
		\begin{aligned}
			&	\partial u   =  \sum_{i=0}^n \iota_* \gamma_i    \\
			& \begin{cases}
				s_* q_i = \partial^+ \gamma_{i-1}   \\ 
				t_* q_i = \partial^- \gamma_{i} 
			\end{cases}  \qquad  i=1,2, \dots, n \\
			&
			\begin{cases}
				s_*	q_0  =   \partial^- \gamma_0 \\
				t_*	q_0 =  \partial^+ \gamma_n
			\end{cases}
		\end{aligned}
	\end{equation}
	and the semi-positive condition\footnote{ \scriptsize  This reflects the fact that a pseudo-holomorphic curve $u$ always has positive symplectic energy, $\int_u \omega >0$, unless $u$ is constant.}:
	\begin{equation}
		\label{energy_S_iota}
		\text{either} \ \ u=0 \quad \text{or} \quad  \textstyle \int_u \omega >0 
	\end{equation}
	Here we recall that the boundary operator is the alternating sum of face maps
	$
	\partial \;=\;\sum_{i=0}^{n}(-1)^i\, \partial_i
	$
	where for $i=0,\dots,n$ we let $\delta_i:\Delta^{n-1}\hookrightarrow\Delta^n$ denote the $i$-th face inclusion
	(omit the $i$-th vertex) and define the $i$-th face map
	$
	\partial_i$ by sending a singular $n$-chain $\sigma$ to the induced $(n-1)$-chain $\sigma\circ \delta_i$ and extending it linearly.
	In degree one this reads and is denoted by
	$\partial=\partial_0-\partial_1=:\partial^+ -\partial^-$

	Two such elements $( u   , (\gamma_i))$ and $(\tilde u , (\tilde\gamma_i) )$ are called \textit{$\iota$-equivalent}, written 
	\[( u   , (\gamma_i)) \sim_\iota  (\tilde u , (\tilde\gamma_i) ),
	\]
	if there exists
	\begin{align*}
		\left( U  , (\Gamma_i)  \right)
		\ \ \in \ \
		C_3(X) \oplus
		\bigoplus_{i=0}^n C_2(L_{v_i})
	\end{align*}
	such that
	\begin{align*}
		& u-\tilde u = - \partial U + \sum_{i=0}^n \iota_* \Gamma_i    \\
		&
		\gamma_i -\tilde \gamma_i = -\partial \Gamma_i  \qquad  &&i=0,1,2,\dots, n
	\end{align*}
	Now, we define the corresponding set of 2-cells to be the set of $\iota$-equivalence classes, that is,
	\[
	\mathbb S_\iota(q_1,\dots, q_n; q_0) = Z_\iota(q_1,\dots, q_n; q_0 ) / \sim_\iota
	\]
	The $\iota$-equivalence class of an element $(u, (\gamma_i))$ is denoted by $[u, (\gamma_i)]$.

	\paragraph{\underline{Partial composition}:}
	Given 
	\[
	\mathbf u=[u, (\gamma_0,\dots, \gamma_n)] \in Z_\iota(q_1,\dots, q_i,\dots , q_n; q_0)\]
	and 
	\[\mathbf u'=[u', (\gamma'_0,\dots, \gamma_m')] \in Z_\iota(q_1',\dots, q_m'; q_i)
	\]
	for a fixed $1\le i\le n$, we define their partial composition $\mathbf u \circ_i \mathbf u'$ to be the $\iota$-equivalence class of the element
	\[
	\big(u+u',  (\gamma_0,\dots, \gamma_{i-2}, \gamma_{i-1}+\gamma_0', \gamma_1',\dots, \gamma'_{m-1}, \gamma'_m+\gamma_i, \gamma_{i+1},\dots, \gamma_n) \big)
	\]
	in
	\[
	Z_\iota ( q_1,\dots, q_{i-1}, q'_1,\dots, q'_m, q_{i+1},\dots, q_n; q_0)
	\]
	
	This completes the construction of $\mathbb S_\iota$ in (\ref{S_iota_eq}).
	
	%
	%
	%

	\paragraph{\underline{Recovery of the second relative homology}:}
	If $\iota: L\to X$ is an \textit{embedding} and $L\cong \iota(L)$ is connected, then $\mathbb S_\iota$ recovers $S_L\subset H_2(X,L)$ in \eqref{S_L_eq}.
	Indeed, the associated graph $E_\iota=(V_\iota,E_\iota)$ simply consists of one vertex $v$ and one loop $e$.
	Since $\iota$ is an embedding, we have $L_v\cong L$, $R(e)=\Delta_L\cong L$, and the maps $s,t: R(e)\to L_v$ are identified with the identity maps $L\to L$.
	Observe that by (\ref{energy_S_iota}), there is a natural well-defined map
	\[
	\mathbb S_\iota (q_1,\dots, q_n; q_0) \to S_L\subset H_2(X,L)  \ , \quad [u, (\gamma_i)_{i=0}^n] \mapsto [u]
	\]

	For $n=0$, an element of $Z_\iota(\varnothing;q_0)$ is a pair $(u,\gamma_0)\in C_2(X)\oplus C_1(L)$ satisfying
	$
	\partial u= \gamma_0$, $\partial^-\gamma_0=q_0$, and $\partial^+\gamma_0=q_0$. Hence, $\partial\gamma_0=\partial^+\gamma_0-\partial^-\gamma_0=0$.
	In particular, the class $[u]\in C_2(X)/C_2(L)$ is a relative $2$-cycle.
	If $(u,\gamma_0)\sim_\iota(\tilde u,\tilde\gamma_0)$, then by definition there exists
	$(U,\Gamma_0)\in C_3(X)\oplus C_2(L)$ such that
	$
	u-\tilde u=-\partial U+ \Gamma_0$ and $ \gamma_0-\tilde\gamma_0=-\partial\Gamma_0$.
	Passing to the quotient $C_2(X)/C_2(L)$ gives
	$
	[u]-[\tilde u]=-[\partial U]=-\partial[U]$,
	so $[u]$ and $[\tilde u]$ give the same class in $H_2(X,L)$.
	Conversely, let $\alpha\in H_2(X,L)$ and pick a representative $u\in C_2(X)$ with $\partial u\in C_1(L)$.
	Since $L$ is connected, the class $[\gamma]\in H_1(L)$ admits a representative by a loop based at $q_0$,
	i.e.\ a $1$-chain $\gamma_0$ with $\partial^\pm\gamma_0=q_0$ and $[\gamma_0]=[\gamma]\in H_1(L)$.
	Then $\gamma-\gamma_0=\partial\Gamma_0$ for some $\Gamma_0\in C_2(L)$, and setting
	$
	u':=u-\iota_*\Gamma_0
	$
	gives $\partial u'=\iota_*\gamma_0$.
	Hence, $(u',\gamma_0)\in Z_\iota(\varnothing;q_0)$ represents $\alpha$.
	To sum up, the map
	\begin{equation}
		\label{S_iota_q_0_S_L_eq}
		\mathbb S_\iota(\varnothing;q_0) \xrightarrow{\cong}  S_L
	\end{equation}
	given by $ [u,\gamma_0] \mapsto [u]$ is an isomorphism.
	However, for $n\ge 1$, the datum of $(\gamma_0,\dots,\gamma_n)$ records a decomposition of the boundary $\partial u$ into
	segments with specified corners, and hence contains more information than the resulting relative class in $H_2(X,L)$.
	Keeping the total boundary $\sum_{i=0}^n \gamma_i$ fixed, one can modify the tuple by adding $1$-cycles in $L$; modulo
	$\iota$-equivalence this ambiguity is governed by $H_1(L)$.
	So, $\mathbb S_\iota(q_1,\dots,q_n;q_0)$ refines $H_2(X,L)$ rather than coinciding with it.

	\paragraph{\underline{Hands-on situations}}

Recall that $L=\coprod_{v\in V_\iota} L_v$ with $V_\iota=\pi_0(L)$. Assume that, for each $v\in V_\iota$, the restriction
	$\iota_v:=\iota|_{L_v}:L_v\to X$
	is an embedding. In other words, $\iota$ is an immersion obtained by taking the disjoint union of embedded submanifolds $L_v\cong \iota(L_v)\subset X$ and allowing their images to intersect each other in $X$. Namely, $\iota$ has no self-intersections within a single component $L_v$, and all self-intersections of $\iota$ arise from intersections between distinct components.
	In particular, the corresponding graph $E_\iota=(V_\iota, E_\iota)$ has no loops except the distinguished loops (\ref{E_L_graph_eq}).
	This is a typical setup in Lagrangian Floer theory and in the study of Fukaya categories.
	
	The previous discussions around (\ref{S_iota_q_0_S_L_eq}) and the partial composition in $\mathbb S_\iota$ immediately imply the following (cf. Figure \ref{fig:2_cell_composition_curved}): 
	
	\begin{prop}
		Let $\mathbb S_\iota(q_1,\dots, q_n; q_0)$ be the set of 2-cells with fixed source and target horizontal 1-cells $(q_1,\dots, q_n)$ and $q_0$. 
		Assume that, for some $1\le i\le n$, the point $q_i$ lies in $R(e_v)$, where $e_v$ is the distinguished loop at a vertex $v\in V_\iota$, and that $\iota|_{L_v}$ is an embedding. Then there is a natural action
		\[
		S_{L_v} \times \mathbb S_\iota (q_1,\dots, q_n; q_0) \to \mathbb S_\iota (q_1,\dots, q_{i-1}, q_{i+1},\dots, q_n; q_0) 
		\]
		where $S_{L_v}\subset H_2(X,L_v)$ is defined as in (\ref{S_L_eq}).
	\end{prop}

	Finally, one can improve Theorem~\ref{fc_moduli_unlabeled_thm} as follows:

	\begin{prop}
		\label{fc_moduli_labeled_prop}
		The $\mathbf{fc}$-multicategory $\mathscr M_\iota$
		is $\mathbb S_\iota$-labeled in the sense of Definition \ref{label_by_fc_multicategory_defn}.
	\end{prop}
	
	\begin{proof}
		Note that both $\mathscr M_\iota$ and $\mathbb S_\iota$ are vertically discrete and have the same 0-cells and 1-cells.
		So, it suffices to define the map for the 2-cells.
		Recall that a 2-cell of $\mathscr M$ over $(e_1,\dots, e_n; e')\in E^*\times E$ is (the isomorphism class of) a pseudo-holomorphic polygon $\mathbf u= (\Sigma, z_0,z_1,\dots, z_k,  u ,\gamma)$ as (\ref{pseudo_holo_polygon_tuple}).
		As in (\ref{gamma_j_lift_eq}), the datum $\gamma$ consists of $k+1$ paths $\gamma_j$ in $L$, giving rise to singular 1-chains in $C_1(L_{v_i})$, still denoted by $\gamma_j$'s.
		Besides, forgetting the pseudo-holomorphic condition, the map $u: (\Sigma, \partial \Sigma)\to (X, \iota(L))$ is first a continuous map and thus defines a singular 2-chain in $C_2(X)$, still denoted by $u$.
		Then, $(u,(\gamma_j))$ induces a 2-cell in $\mathbb S_\iota$. In this way, one can find a map $\mathscr M_\iota \to\mathbb S_\iota$ of $\mathbf{fc}$-multicategories.
	\end{proof}

\subsection{Gromov compactness and factor-closedness of $\mathbf{fc}$-multicategory structure}
\label{s_Gromov_factor_closed}
Let $\iota: L\to X$ be a Lagrangian immersion with clean self-intersection as before. We have constructed a natural $\mathbf{fc}$-multicategory $\mathscr M_\iota=(\pi_0(L), L\times_X L, \mathcal M)$ of moduli spaces in Theorem~\ref{fc_moduli_unlabeled_thm}, and showed in Proposition~\ref{fc_moduli_labeled_prop} that it admits an $\mathbb S_\iota$-labeling in the sense of Definition \ref{label_by_fc_multicategory_defn}, where $\mathbb S_\iota=(\pi_0(L), L \times_X L, \mathbb S_\iota)$ is the $\mathbf{fc}$-multicategory defined as in \eqref{S_iota_eq}.
The underlying directed graph is $E_\iota=(\pi_0(L), L \times_X L)$ in (\ref{E_L_graph_eq}).

In symplectic geometry, one can often single out an appropriate subcollection of the relevant moduli spaces in $ \mathscr M_\iota$ which is \textit{closed under Gromov limits}, meaning that if a sequence of pseudo-holomorphic curves $\mathbf u_n$ lies in this subcollection, then any Gromov limit $\mathbf u_\infty$ (after passing to a subsequence) is represented by an element of the same subcollection. When such a subcollection is available, the usual virtual-counting and algebraic constructions can be carried out internally, producing refined or relative invariants.
Our point is that the factor-closedness condition is an operadic reflection of this closure under Gromov limits.

In practice, Gromov compactness often forces one to work with finitely many Lagrangians at a time, and one faces technical limitations in treating infinitely many Lagrangian submanifolds simultaneously. To define a ``full Fukaya category'' one should take an exhausting increasing sequence of finite collections of Lagrangians and form an appropriate homotopy inductive limit. See Fukaya's remark in \cite[\S 18.1]{FuUnobstructed}.
From this perspective, we believe that it is useful to employ the language of $\mathbf{fc}$-multicategories as it packages a range of seemingly different $A_\infty$-type structures into a single, uniform operadic description.

Finally, let's give a few concrete examples encountered in symplectic geometry:

\medskip 

Let $\{L_0, L_1,\dots, L_N\}$ be a finite collection of connected embedded Lagrangian submanifolds, where $N\in\mathbb Z_+$ or $N=\infty$, and assume that the associated Lagrangian immersion
\[
\textstyle \iota_N: L\equiv  \bigsqcup_{i=0}^N L_i\to X
\]
has transverse self-intersection. Set $V=\{0,\dots, N\}$ and $E=L \times_X L$.

First, for each $m<n$, we can similarly obtain a Lagrangian immersion
\[
\textstyle \iota_m: L'\equiv \bigsqcup_{i=0}^m L_i\to X.
\]

\begin{prop}
\label{factor_closed_m_N_prop}
$\mathscr M_{\iota_m}$ is a factor-closed $\mathbf{fc}$-submulticategory of $\mathscr M_{\iota_N}$.
\end{prop}

\begin{proof}
Let $V_0=\{0,\dots, m\}$.
The directed graph $E_0:= L'\times_X L'$
is a directed subgraph of
$E =  L\times_X L$,
and observe that $\mathscr M_{\iota_m}$ is the full $\mathbf{fc}$-submulticategory of $\mathscr M_{\iota_n}$ on $E_0$ in the sense of Definition~\ref{full_fc_sub_defn}. 
Let $\vec e=(e_1,\dots, e_n)\in (E_0)^*$, and let $e_i$ be an edge from $v_{i-1}$ to $v_{i}$ for $i=1,\dots, n$.
This means $e_i$ is a point in $L_{v_{i-1}}\times_X L_{v_i}=L_{v_{i-1}}\cap L_{v_i}$, and thus $v_{i-1}, v_i\in V'$. In particular, $v_0,v_n\in V_0$.
Then, every edge $e' \in E$ from $v_0$ to $v_n$ is contained in $L_{v_0}\cap L_{v_n} \subset L'\cap L'$, namely, $e'\in E_0$.
So, $E_0$ is endpoint-closed in $E$ in the sense of Definition \ref{endpoint_closed_defn}.
Therefore Proposition~\ref{prop:endpoint_closed_factor_closed} applies and yields the desired factor-closedness.
\end{proof}

Second, consider a partition $V=V'\sqcup V''$ with $V=\{0,\dots,N\}$. This determines two Lagrangian immersions
\[
\iota' : L'=\bigsqcup_{v\in V'} L_v \to X,
\qquad
\iota'' : L''=\bigsqcup_{v\in V''} L_v \to X .
\]
Recall that $L=\bigsqcup_{v\in V} L_v=L'\sqcup L''$ and $E=L\times_X L$. We regard an element
$e=(p_1,p_2)\in E$ as a directed edge from $v_1$ to $v_2$, where $p_1\in L_{v_1}$ and $p_2\in L_{v_2}$.
There is also the reversed edge $\bar e:=(p_2,p_1)$ from $v_2$ to $v_1$ in $E$.
Now, we further consider the directed subgraph
\[
E_0:=(L'\times_X L') \cup (L'\times_X L'') \cup (L''\times_X L'')
\]
of $E$.
In other words, $E_0$ consists of all edges $e=(p_1,p_2)$ except those with $p_1\in L''$ and $p_2\in L'$, i.e.\ except the edges whose source lies in $V''$ and whose target lies in $V'$.
For instance, if $e\in L'\times_X L''\subset E_0$, then $\bar e\in L''\times_X L'$ and hence $\bar e\notin E_0$.

In this context, we can construct a vertically discrete $\mathbf{fc}$-submulticategory
$\mathscr M_{\iota',\iota''}$ of $\mathscr M_{\iota_N}$
to be the full $\mathbf{fc}$-submulticategory of $\mathscr M_\iota$ on the subgraph $E_0$ in the sense of Definition~\ref{full_fc_sub_defn}.

\begin{prop}
\label{factor_closed_partition_2_prop}
$\mathscr M_{\iota',\iota''}$ is factor-closed in $\mathscr M_{\iota_N}$.
\end{prop}

\begin{proof}
By Proposition \ref{prop:endpoint_closed_factor_closed},
it suffices to show that $E_0$ is endpoint-closed in $E$.
Suppose $\vec e=(e_1,\dots, e_n)\in (E_0)^*$ and $e_i$ is an edge from $v_{i-1}$ to $v_i$ for $i=1,\dots, n$. Let $e'\in E$ be an edge from $v_0$ to $v_n$. We aim to show $e'\in E_0$.
If both $v_0$ and $v_n$ lie in $V'$, then $e'\in L_{v_0}\times_X L_{v_n}\subset L'\times_X L'$ lies in $E_0$.
If both $v_0$ and $v_n$ lie in $V''$, the argument is similar.
If $v_0\in V'$ and $v_n\in V''$, then $e'\in L'\times_X L''$ also lies in $E_0$.
Finally, we observe that the case $v_0\in V''$ and $v_n\in V'$ is impossible since there is no edge from $V''$ to $V'$.
\end{proof}

We remark that given a partition $V=\bigsqcup_{j=1}^r V^{(j)}$ with induced immersions $\iota^{(j)}:L^{(j)}\to X$, set
\begin{equation}
\label{partition_V_to_r_sets}
E_0:=\bigcup_{j\le k} \, L^{(j)}\times_X L^{(k)}  
\end{equation}
Then, $E_0$ is endpoint-closed, hence determines a factor-closed $\mathbf{fc}$-submulticategory
$\mathscr M_{\iota^{(1)},\dots,\iota^{(r)}}$ of $\mathscr M_\iota$ by almost the same argument as above. More generally, any endpoint-closed directed subgraph $E_0\subset E$ induces a factor-closed $\mathbf{fc}$-submulticategory of $\mathscr M_\iota$.
At present, however, the cases we have encountered in symplectic applications are primarily those arising from the subcollection of moduli spaces in Proposition \ref{factor_closed_m_N_prop} and \ref{factor_closed_partition_2_prop}.

	\section{Algebras over $\mathbf{fc}$-multicategories and their dg variants}
	\label{s_A_inf_as_algebra_for_dg_fc}
	

	Operads and multicategories are special cases of $\mathbf{fc}$-multicategories by Example \ref{suspension_ex}. Hence, one should be able to recover, from the notion of an algebra over $\mathbf{fc}$-multicategories, the usual notion of an algebra over an operad. 
	Just as groups admit representations, multicategories admit algebras. 
	One can develop the idea of algebras over $T$-operads and $T$-multicategories for any cartesian monad $T$, and there are several equivalent definitions as discussed in \cite[\S 4.3, \S 6.3, \S 6.4]{leinster2004higher}.
	In this paper, we choose the one in \cite[\S 6.4]{leinster2004higher} which is in a form analogous to the classical endomorphism description as in \eqref{operad_A_inf_intro_eq}.
	
	For our purpose, we only focus on the case $T=\mathbf {fc}: \mathcal D\to\mathcal D$ here.
	Recall that we denote by $\mathcal D$ the category of directed graphs (\S \ref{s_fc_directedgraphs}).

	Fix a directed graph $E=(V,E,s,t)$ in $\mathcal D$. 
	Let $p_X:X\to E$ and $p_Y:Y\to E$ be objects of the slice category $\mathcal D/E$.
	More specifically, the relevant directed graphs are written as
	\[
	X=(\bar X, X, \mathfrak s_X,\mathfrak t_X), \qquad Y=(\bar Y, Y, \mathfrak s_Y, \mathfrak t_Y)
	\]
	and the maps of directed graphs are written as \[
	p_X=(\bar p_X,p_X):(\bar X,X)\to (V,E),\qquad
	p_Y=(\bar p_Y,p_Y):(\bar Y,Y)\to (V,E),
	\]

	\subsection{Exponential objects and endomorphisms}
\label{s_exponential_objects_endomorphisms}
	For objects $A\xrightarrow{a}E$ and $B\xrightarrow{b}E$
	in the slice category $\mathcal D/E$, an \emph{exponential object} $B^{A}\xrightarrow{\pi}E$
	is an object of $\mathcal D/E$ together with an {evaluation morphism}
	$\ev:\ B^{A}\times_{E} A \longrightarrow B$ in $\mathcal D/E$ with the following universal property: for every object $Z\xrightarrow{z}E$ in $\mathcal D/E$, the map
	$
	\Hom_{\mathcal D/E}(Z,\ B^{A})
	\to
	\Hom_{\mathcal D/E}(Z\times_{E}A,\ B)$ given by
	$u \mapsto \ev\circ (u\times_{E}\id_A)$
	is a bijection.
	Note that for each $E\in\mathcal D$, the slice category $\mathcal D/E$ admits exponential objects.
	
	Define objects
	\[
	G_1(X):=G_{1,E}(X)\ :=\ \bigl(X^*\times E \xrightarrow{\,p^*_X \times \id_E\,} E^*\times E\bigr)\]
	\[
	G_2(Y):=G_{2,E}(Y)\ :=\ \bigl(E^* \times Y \xrightarrow{\,\id_{E^*}\times p_Y\,} E^*\times E\bigr).
	\]
	in the slice category $\mathcal D/ (E^*\times E)$, and define
	\[
	\mathbf{Hom}(X,Y)\ :=\ G_2(Y)^{\,G_1(X)}\ \in\ \mathcal D/(E^*\times E),
	\]
	as the exponential object.
	Recall that we denote by $\mathbf {fc}(E)$ by $E^*$ as in (\ref{E*_eq}). 	

	For $v\in V$ and $e\in E$, define $\bar X(v)=\bar p_X^{-1}(v)$ and $X(e)=p_X^{-1}(e)$; define $\bar Y(v)$ and $Y(e)$ similarly.
	For $\vec e=(e_1,\dots, e_n)\in E^*$,
	we define the set of \emph{lifts of $\vec e$ in $X$} by
	\begin{equation}
		\label{X*_eq}
		\begin{aligned}
			X^*(\vec e) 
			&= (p_X^*)^{-1} (\vec e)  = X(e_1)\times_{\bar X} \cdots \times_{\bar X} X(e_n) \\
			&= \left\{ (x_1,\dots, x_n) \in X(e_1)\times \cdots \times X(e_n) \mid \mathfrak t_X(x_i)=\mathfrak s_X(x_{i+1}) \ , i=1,\dots, n-1
			\right\}
		\end{aligned}
	\end{equation}
	In other words, $X^*(\vec e)$ is the set of composable paths in $X$ that project to $\vec e$ under $p_X^*$. We also define $Y^*(\vec e)$ similarly.
	Note that $E^*\times E$ is a directed graph whose set of vertices is $V\times V$.
	
	Denote the vertex set of $\mathbf{Hom}(X,Y)$ by
	$\mathbf {V}(X, Y)$
	and by abuse of notation, also write $\mathbf{Hom}(X,Y)$ for its edge set.
	\textit{A vertex} over $(u , u')\in V\times V$ is a triple $(u, u', \phi )$ or simply $\phi$, where $u, u' \in V$ and $\phi : \bar X(u) \to \bar Y(u')$ is a function.
	In other words,
	\[\mathbf {V}(X, Y) := \{  \phi : \bar X(u)\to\bar Y(u')  \ \mid \ u,u'\in V  \}  \ . \]
	Given $(\vec e; e')\in E^*\times E$, let $v_0, v_1,\dots, v_n$ be the vertices of $\vec e$, and let $v'_0,v'_1$ be the source and target vertices of $e'$.
	\textit{An edge of $\mathbf{Hom}(X,Y)$} over $(\vec e, e')$ from $(v_0, v'_0 , \phi )$ to another $(v_n, v'_1, \psi)$
	is a tuple $(\vec e, e'; \phi, \psi; \xi)$ or simply $\xi$ where
	\begin{equation}
		\label{xi_eq}
		\xi :\  X^*(\vec e) \to  Y(e')
	\end{equation}
	is a function such that for a lifted path $\sigma\in X^*(\vec e)$, the edge $\xi(\sigma)\in Y(e')$
	has endpoints compatible with $\phi$ and $\psi$.
	The source and target maps $\mathbf s_{X,Y},  \mathbf t_{X,Y} : \mathbf{Hom}(X,Y)\to \mathbf V(X,Y)$ send this edge $\xi$ to the vertices $(v_0, v'_0,\phi)$ and $(v_n, v'_1, \psi)$ respectively.
	For the slice category $\mathcal D/ (E^*\times E)$, we have two structure maps
	\[
	\pmb d_{E,X,Y} \times \pmb c_{E,X,Y} : \ (\mathbf V(X,Y), \mathbf{Hom}(X,Y) ) \to  (V\times V, E^* \times E)
	\]
	of directed graphs, sending a vertex $(u,u',\phi)\in\mathbf V(X,Y)$ to $(u, u')\in V\times V$, and sending an edge, i.e. a tuple $\xi=(\vec e, e'; \phi, \psi ; \xi)$ to $(\vec e, e')\in E^*\times E$. 
	
	When $X=Y$, we write $\mathbf V(X)=\mathbf V(X,X)$ and $\mathbf {End}(X)=\mathbf{Hom}(X,X)$; the structure maps are denoted by $\pmb d_{E,X}$ and $\pmb c_{E,X}$.
	The following is due to \cite[Proposition 6.4.2]{leinster2004higher}, and here we specialize to the case of $\mathbf{fc}$-multicategories.

	\begin{prop}
		\label{End_E_fc_prop}
		Let $E=(V,E,s,t)$ be an object in $\mathcal D$.
		Let $X=(\bar X, X,\mathfrak s,\mathfrak t, p_X)$ be an object in the slice category $\mathcal D/E$.
		There is a natural $\mathbf{fc}$-multicategory (here we abuse the notation again)
		\[
		\mathbf{End}(X) = (V, E, \mathbf {V}(X), \mathbf{End}(X) )
		\]
		Specifically, 
		\begin{itemize}
			\itemsep 2pt
			\item a $0$-cell is a vertex $V$ of $E$.
			\item a horizontal 1-cell from $v_0$ to $v_1$ is an edge $e\in E$.
			\item a vertical $1$-cell from $u$ to $u'$ is an element in $\mathbf V(X)$ over
			$(u,u')\in V\times V$.
			\item a $2$-cell $\xi$ over an edge $(\vec e,e')\in E^*\times E$
			is an element in $\mathbf{End}(X)$ over $(\vec e,e')$.
		\end{itemize}
	\end{prop}
	
	\begin{proof}
		Similar to \eqref{2_cell_eq}, the description of a 2-cell and the relevant 0-cells and horizontal/vertical 1-cells can be illustrated by a rectangle diagram as follows:
		\begin{equation}
			\label{2_cell_End}
			\xymatrix{
				\bar X(v_n)\ar@{<-}[rr]^{X(e_n)} \ar[d]^{{\psi }} & & \bar X(v_{n-1}) \ar@{<-}[rr]^{X(e_{n-1})} & & & \cdots & {}\ar@{<-}[rr]^{X(e_1)} & & \bar X(v_0) \ar[d]^{\phi} \\
				\bar X(v'_1) \ar@{<-}[rrrrrrrr]_{X(e')} & &  & & \ar@{}[u]|{\Downarrow \  \xi }&  &  & & \bar X(v'_0)
			}
		\end{equation}
		Here an arrow $\bar X(v_i)  \xleftarrow{X(e_i)}  \bar X(v_{i-1})$ really represents a span $\bar X(v_i) \xleftarrow{\mathfrak t} X(e_i) \xrightarrow{\mathfrak s} \bar X(v_{i-1})$
		and thus a string of arrows represents the set of lifts $X^*(\vec e)$ as in (\ref{X*_eq}).
		
		The identity 
		$
		\pmb\iota_{E,X}: E\to \mathbf{End}(X)
		$ is defined by sending a vertex $v$ of $E$ to $(v,v,\id_{\bar X(v)})$ and by sending an edge $e$ from $u$ to $u'$ to the evident 2-cell $\id_{X(e)}=((e), e \ ; \id_{\bar X(u)}, \id_{\bar X(u')} ; \id_{X(e)})$, illustrated as follows:
		\[
		\iota_{E,X} : \ \ \vcenter{\hbox{\xymatrix{u' & u \ar[l]_e }}}
		\qquad \leadsto \qquad
		\vcenter{\hbox{
				\xymatrix{ \bar X(u') \ar[d]_\id  \ar@{}[dr]|{\Downarrow \  \id }  & \bar X(u) \ar[d]_\id  \ar[l]_{X(e)}\\
					\bar X(u') & \bar X(u) \ar[l]^{X(e)} }
		}}
		\]
		Given $\vec e=(e_1,\dots, e_n) \in E^*$, $e'\in E$, $\vec f=(f_1,\dots, f_m)\in E^*$, and $1\le i\le n$, the composition
		\begin{align*}
			\circ_i : & \ \mathbf{End}(X) (e_1,\dots, e_n; e') \times  \mathbf{End}(X) (f_1,\dots, f_m; e_i)  \\
			& \to
			\mathbf{End}(X) (e_1,\dots, e_{i-1}, f_1,\dots, f_m, e_{i+1},\dots, e_n; e') \ , \quad (\xi_1,\xi_2)\mapsto \xi_1 \circ_i \xi_2
		\end{align*}
		for the 2-cells can be described as follows
		\[
		\xymatrix{
			\bar X(v_n) \ar[d]_{\id} \ar@{}[dr]|{\Downarrow \  \id } & {\quad} \ar[l] & \bar X(u_m) \ar[d]_{\psi_2} &  \cdots  \ar[l]_{X(f_m)} \ar@{}[d]|{\Downarrow \  \xi_2 } & \bar X(u_0) \ar[l]_{X(f_1)} \ar[d]_{\phi_2} & \quad &  \bar X(v_0) \ar[d]^{\id} \ar[l] \ar@{}[dl]|{\Downarrow \  \id } \\
			\bar X(v_n) \ar[d]_{\psi_1} & \cdots \ar[l] & \bar X(v_i) \ar[l] &  & \bar X(v_{i-1}) \ar[ll] & \cdots \ar[l] & \ar[l] \bar X(v_0) \ar[d]^{\phi_1}  \\
			\bar X(v_1')  & & & \ar@{}[u]|{\Downarrow \  \xi_1}  & & & \bar X(v_0') \ar[llllll]^{X(e')}
		}
		\]
	\end{proof}

	\subsection{Algebras over $\mathbf{fc}$-multicategories}
	
	Following \cite[\S 6.4]{leinster2004higher}, we introduce:
	
	\begin{defn}
		\label{C_alg_defn_End}
		Let $E=(V,E,s,t)\in\mathcal D$. Let $\mathscr M=(V,E,\mathcal V,\mathcal M, \pmb d, \pmb c,\pmb \iota, \pmb \gamma)$ be a $\mathbf{fc}$-multicategory over $E$.
		By a \textit{$\mathscr M$-algebra} or an \textit{algebra over $\mathscr M$}, we mean a pair $(X,\pmb \alpha)$ consisting of
		\vspace{2pt}
		\begin{itemize}
			\itemsep 4pt
			\item an object 
			$X=(\bar X,X, \mathfrak s,\mathfrak t, p_X)
			$
			in the slice category $\mathcal D/E$, that is, $(\bar X, X,\mathfrak s,\mathfrak t)$ is a directed graph and $p_X=(\bar p_X, p_X): (\bar X,X)\to (V,E)$ is a map of directed graphs.
			
			\item a map of $\mathbf{fc}$-multicategories 
			$
			\pmb \alpha : \mathscr M\to \mathbf{End}(X)
			$
			such that the induced map of directed graphs $(V,E) \to (V,E)$ is the identity.
		\end{itemize}
	\end{defn}

	The following extra condition is often useful in practice.
	
	\begin{defn}
		\label{C_alg_X_simple_defn}
		We call a directed graph $X=(\bar X, X, \mathfrak s,\mathfrak t,  p_X)$ over $E=(V,E,s,t)$ \textit{simple} if $\bar X=V$, $\mathfrak s=s\circ p_X$, $\mathfrak t= t\circ p_X$, and $\bar p_X=\id_V$.
		Also, we call a $\mathscr M$-algebra $(X,\pmb\alpha)$ simple if $X$ is \textit{simple}.
	\end{defn}

	From now on, we mostly restrict attention to the case $X=(\bar X, X, \mathfrak s,\mathfrak t, p_X)$ is {simple}. This hypothesis often simplifies the discussion and is sufficient for our purposes.
	Let's unpack the data for $(X,\pmb\alpha)$ when $X$ is simple.
	On one hand, we note that every $\bar X(v)=\{v\}$ is a one-point set, so a map $\bar X(u)\to \bar X(u')$ is always the unique trivial map.
	Thus, $\mathbf {V}(X)\cong V\times V$.
	By the simple condition, the set of lifts of $\vec e=(e_1,\dots, e_n)\in E^*$ in (\ref{X*_eq}) becomes just a direct product
	$
	X^*(\vec e) =X(e_1) \times \cdots \times X(e_n)
	$.
	In fact, since $(X,\pmb\alpha)$ is simple and $X(e)=p_X^{-1}(e)$, the condition $\mathfrak t(x_i)=\mathfrak s(x_{i+1})$ in (\ref{X*_eq}) reduces to $t\circ p_X(x_i)=s\circ p_X(x_{i+1})$ and then $t(e_i)=s(e_{i+1})$, which precisely corresponds to the condition $\vec e\in E^*$.
	Therefore, a 2-cell in $\mathbf {End}(X)$ is a map 
	$\xi: X(e_1)\times \cdots \times X(e_n)\to X(e')$ for $ \vec e=(e_1,\dots, e_n)\in E^*$ and $e'\in E$. In particular,
\begin{equation}
\label{End_X_set_eq}
	\mathbf{End}(X) = \bigcup_{(\vec e, e') \in E^*\times E} \Hom (X(e_1)\times \cdots \times X(e_n) \ , \ X(e') )
\end{equation}

	On the other hand,
	recall that 
	$\pmb \alpha = (\alpha_V, \alpha_E , \alpha_{\mathcal V} , \alpha )$
	consists of two maps of directed graphs $\alpha_E=(\alpha_V,\alpha_E): (V,E)\to (V,E)$ and $\alpha =(\alpha_{\mathcal V},\alpha ) : (\mathcal V, \mathcal M) \to (\mathbf {V}(X), \mathbf{End}(X))$ where $\alpha_E$ is required to be the identity and the three compatibility conditions in Definition \ref{map_fc_defn} hold.
	Specifically, 
	by Definition \ref{map_fc_defn} (1), we see that $\alpha_{\mathcal V}(\id_v)= (v,v)$ and $\alpha (\id_e)=\id_{X(e)}$,
	where $\id_v$ is the identity vertical 1-cell at $v\in V$ and $\id_e$ is the identity 2-cell at $e\in E$.
	By Definition \ref{map_fc_defn} (2), we have $\pmb c= \pmb c_{E,X}\circ \alpha$ and $\pmb d  = \pmb d_{E,X}\circ \alpha$.
	One can then check that $\alpha_{\mathcal V}: \mathcal V \to V\times V$ is just the map sending a vertical 1-cell $f: u\to u'$ in $\mathscr M$ to the pair $(u, u')\in V\times V$.
	In particular, if $\mathscr M$ is a vertically discrete $\mathbf{fc}$-multicategory, i.e. $\mathcal V\cong V$, then $\alpha_{\mathcal V}:V\to V\times V$ is the diagonal map $u\mapsto (u,u)$.
	One can also check that the $\alpha$ sends a 2-cell $\mathbf u \in \mathcal M( \vec e, e')$ for $\vec e=(e_1,\dots, e_n)\in E^*$ and $e'\in E$ to a map 
	$
	\alpha(\mathbf u) : X^*(\vec e) \to X(e')
	$.
	Lastly, by Definition \ref{map_fc_defn} (3),
	$\alpha$ is compatible with the 2-cell compositions, namely,
	$
	\alpha (\mathbf u\circ_i \mathbf u') = \alpha(\mathbf u) \circ_i \alpha(\mathbf u')
	$.

	\subsection{The dg variant of $\mathbf{fc}$-multicategories }

	We would like to realize various $A_\infty$ structures (such as $A_\infty$ algebras, $A_\infty$ bimodules, $A_\infty$ categories) as algebras over certain $\mathbf{fc}$-multicategories.
	Recall that an $A_\infty$ algebra is an algebra over the $A_\infty$ operad, which is a \textit{differential graded (dg)} operad, and in view of (\ref{operad_A_inf_intro_eq}), an $A_\infty$ algebra on a cochain complex $A$ is a dg operad morphism $\mathcal A_\infty \to \mathrm{End}_A$ from the $A_\infty$ operad to the endomorphism operad of $A$.
	
	Accordingly, we may need an appropriate notion of a \textit{dg} $\mathbf{fc}$-multicategories and their morphisms. This might involve invoking the enrichment theory for general $T$-multicategories, which is, in general, a quite nontrivial subject; see the works of Leinster \cite{leinster1999generalized}, \cite[\S6.8]{leinster2004higher}. Clearly, the full framework of this enrichment theory is very deep and may lead us too far afield. Thus, at the cost of reducing some generality, we instead adopt an ad hoc hands-on formulation that should be sufficient for our purposes in studying the $A_\infty$ structures; see also the comments in \cite[Example~5.1.11]{leinster2004higher}.

	Fix a commutative ground ring $\Bbbk$, and we will work in the category of cochain complexes of graded $\Bbbk$-vector spaces with differentials of degree $+1$. 
	The degree of an element $x$ in a cochain complex is usually denoted by $|x|$.

	The following definition is an attempt to generalize the notion of a dg operad.

	\begin{defn} 
		\label{dg_fc_defn}
		A \emph{dg $\mathbf{fc}$-multicategory} over $\Bbbk$ is defined to be the data
		\[
		\widehat {\mathscr M} =( \mathscr M;\delta)
		\]
		where
		\begin{enumerate}
			\itemsep 2pt
			\item $\mathscr M=(V,E, \mathcal V, \mathcal M, \pmb d,\pmb c, \pmb \iota,\pmb \gamma)$ is a $\mathbf {fc}$-multicategory.
			\item Fix $\vec e=(e_1,\dots, e_n)\in E^*$ and $e'\in E$. The fiber
			$
			\mathcal M(e_1,\dots, e_n; e')$ of 2-cells over $(\vec e, e')$ is a graded cochain complex over $\Bbbk$ whose differential is denoted by
			$
			\delta : =\delta_{\vec e, e'}
			$.
			\item  The identity 2-cell $\id_e=\iota(e)$ is a degree zero $\delta$-cycle in $\mathcal M(\vec e; e')$, i.e. $\delta (\id_e)=0$.
			\item  For the (partial) composition of 2-cells, we have the Leibniz-type rule
			\[
			\delta (\mathbf u \circ_i \mathbf u')= (\delta\mathbf u)\circ_i \mathbf u' \ +\ (-1)^{|\mathbf u|}\,\mathbf u\circ_i (\delta\mathbf u').
			\]
		\end{enumerate}
	\end{defn}
	
	\begin{defn}
		\label{map_fc_dg_defn}
		A \textit{map of dg $\mathbf{fc}$-multicategories} $\pmb \alpha : \widehat{\mathscr M} \to \widehat{\mathscr M}'$ is a map of $\mathbf{fc}$-multicategories in the sense of Definition \ref{map_fc_defn} such that the 2-cell components of $\pmb\alpha$ are given by degree-zero cochain maps.
	\end{defn}

	\begin{rmk}
		In the above definition, the collections of 0-cells and of horizontal/vertical 1-cells remain ordinary sets; only the 2-cells are promoted to cochain complexes.
	\end{rmk}
	
	\begin{rmk}
		Recall that the category of directed graphs is the functor category $\mathcal D=[\mathbb H^{\mathrm{op}}, \mathbf{Set}]$ where $\mathbb H$ is the category with two objects and two distinguished morphisms.
		In principle, we may introduce $\mathcal D(\mathcal C)=[\mathbb H^{\mathrm{op}}, \mathcal C]$ for a suitable category $\mathcal C$, and somehow develop the notion of ``$\mathbf{fc}$-multicategories internal to $\mathcal C$''.
		We do not pursue this internal approach here; instead, we adopt a more concrete and slightly ad hoc formulation tailored to our needs.
	\end{rmk}

	\begin{ex}
		\label{dg_operad_as_dg_fc_mult_ex}
		One can view a dg operad as a dg $\mathbf{fc}$-multicategory in Definition \ref{dg_fc_defn}.
		Let $\mathcal O=\{\mathcal O(n)\}_{n\ge 0}$ be a non-symmetric dg operad with unit $1\in \mathcal O(1)$ and partial compositions $\circ_i:\mathcal O(n) \times \mathcal O(m)\to \mathcal O(n+m-1)$ satisfying the associativity axiom as in \eqref{operad_axiom_standard}; here each $\mathcal O(n)$ is a cochain complex that carries a differential $\delta=\delta_n:\mathcal O(n)\to \mathcal O(n)$, compatible with compositions in the sense that
		\begin{equation}
			\label{operadic_leibniz_rule}
			\delta (x\circ_i y) = (\delta x)\circ_i y + (-1)^{|x|}\,x\circ_i(\delta y)
		\end{equation}
		where $|x|$ is the degree of $x$.
		By Example \ref{suspension_ex}, the operad $\mathcal O$ gives rise to a vertically discrete $\mathbf{fc}$-multicategory $ \Sigma\mathcal O$ with a single 0-cell and with vertical and horizontal 1-cells being the identities.
		The set of 2-cells with $n$ input horizontal 1-cells is precisely identified with $\mathcal O(n)$.
		Specifically, if we write $\Sigma\mathcal O=(V,E,\mathcal M,\delta)$ as above, then $V=E=\{\ast\}$ is a one-point set and
		\[
		\mathcal M(\underbrace{\ast, \dots \ast}_n ; \ast)=\mathcal O(n)
		\]
		The identity 2-cell is the unit in $\mathcal O(1)$. The operadic Leibniz rule (\ref{operadic_leibniz_rule}) corresponds to the condition (4) in Definition \ref{dg_fc_defn}.
	\end{ex}

\subsection{The dg endomorphisms}

	Let $E=(V,E,s,t)\in\mathcal D$. Let $\mathscr M=(V,E,\mathcal V,\mathcal M, \pmb d, \pmb c,\pmb \iota, \pmb \gamma)$ be a $\mathbf{fc}$-multicategory over $E$.
	Fix an object $X$ in $\mathcal D/E$ which is required to be \textit{simple} in the sense of Definition \ref{C_alg_X_simple_defn}, so we may write
	\[
	X=(V\ , \ X \ ,  \ s\circ p_X \ , \  s \circ p_X \ , \ p_X)
	\]
	where $p_X=(\id, p_X): (V, X)\to (V, E)$ is a map of directed graphs that is the identity on the vertices.

As an analogue of \eqref{X*_eq}, we introduce
\begin{equation}
\label{X*_dg_eq}
X^*(\vec e) = X(e_1)\otimes \cdots\otimes X(e_n)
\end{equation}
for $\vec e=(e_1,\dots, e_n)\in E^*$, where we abuse the notation.
Here the simple condition for $X$ is necessary, as explained in Remark \ref{X_simple_rmk} below.

Assume that $X(e)$ is a cochain complex for each $e\in E$ with the differential denoted by $\mathsf d_e=\mathsf d_e^X$.
The above $X^*(\vec e)$ is a cochain complex
equipped with the tensor product differential $\mathsf d_{\vec e}=\mathsf d_{\vec e}^X$ so that
\begin{equation}
	\label{d_tensor_X_eq}
	\mathsf d^X_{\vec e}(x_1\otimes\cdots\otimes x_n)
	:=\sum_{k=1}^n (-1)^{|x_1|+\cdots+|x_{k-1}|}\ 
	x_1\otimes\cdots\otimes \mathsf d^X_{e_k}(x_k)\otimes\cdots\otimes x_n.
\end{equation}

\begin{rmk}
\label{X_simple_rmk}
	We impose the auxiliary assumption that $X$ is \emph{simple} as in Definition \ref{C_alg_X_simple_defn}.
	This is somehow ad hoc. However, without this assumption, we may encounter fiber products of cochain complexes by \eqref{X*_eq}.
	Recall that if $(A,\mathsf d_A)$, $(B,\mathsf d_B)$, and $(C,\mathsf d_C)$ are cochain complexes and
	$f:A\to C$ and $g:B\to C$ are cochain maps, then their fiber product
	is the subcomplex
	$
	A\times_{C} B
	:= \ker   \ (f-g: A\oplus B \to C ),
	$
	with differential given by restriction of the product differential
	$
	\mathsf d_{A\times_C B}(a,b):=(\mathsf d_A a,\ \mathsf d_B b).
	$
	While this construction is natural in the category of cochain complexes, it seems not the one that arises in the $A_\infty$ structures.
\end{rmk}

	We can establish a dg variant of $\mathbf{End}(X)$ in Proposition \ref{End_E_fc_prop} as follows:

	\begin{prop}
		\label{End_E_X_widehat_prop}
There is a natural dg $\mathbf{fc}$-multicategory
		\[
		\widehat{\mathbf{End}}(X) = (V, E, V\times V, \widehat{\mathbf{End}}(X)  \ ; \  \widehat{\mathsf d})
		\]
		where
		\begin{itemize}
			\itemsep 2pt
			\item a $0$-cell is a vertex $V$ of $E$.
			\item a horizontal 1-cell from $v_0$ to $v_1$ is an edge $e\in E$.
			\item a vertical $1$-cell from $u$ to $u'$ is an element
			$(u,u')\in V\times V$.
			\item the space of $2$-cells over $(\vec e,e')\in E^*\times E$, denoted by $\widehat{\mathbf{End}}(X) (\vec e; e')$,
			is the graded vector space of maps
			$
			\xi : X(e_1) \otimes \cdots \otimes X(e_n) \to X(e')
			$.
Each $\widehat{\mathbf{End}}(X) (e_1,\dots, e_n; e')$ is a cochain complex whose differential $\widehat{\mathsf d}=\widehat{\mathsf d}_{\vec e; e'}$ is given by
		\end{itemize}
	\begin{equation}
		\label{widehat_d_eq}
		\big(\widehat{\mathsf d} \xi \big) (x_1\otimes \cdots \otimes x_n) := \mathsf d_{e'} \big( \xi (x_1\otimes \cdots \otimes x_n ) \big)-  \sum_{k=1}^n (-1)^{|\xi|+|x_1|+\cdots+|x_{k-1}|} \ \xi(x_1\otimes \cdots \otimes \mathsf d_{e_k} x_k \otimes \cdots \otimes x_n)
	\end{equation}
	\end{prop}
	
	\begin{proof}
		Following \eqref{2_cell_eq} and \eqref{2_cell_End}, a 2-cell $\xi$ is described by a rectangle as follows:
		\begin{equation*}
			\xymatrix{ v_n \ar@{.}[d] \ar@{<-}[rr]^{X(e_n)} & & v_{n-1} \ar@{<-}[rr]^{X(e_{n-1})} & & & \cdots & {}\ar@{<-}[rr]^{X(e_1)} & &  v_0   \ar@{.}[d]   \\
				v'_1 \ar@{<-}[rrrrrrrr]_{X(e')} & &  & & \ar@{}[u]|{\Downarrow \  \xi }&  &  & & v'_0
			}
		\end{equation*}
		where $e_i$ is an edge from $v_{i-1}$ to $v_i$ and $e'$ is an edge from $v_0'$ to $v_1'$.
		Note that the vertical 1-cell is simply a point in $V\times V$, so we present it as a dashed line here.
		
		The identity 2-cell over $e\in E$ is just the identity map $\id_{X(e)}$. The degree is $|\id_{X(e)}|=0$, and 
		\[
		\widehat{\mathsf d} (\id_{X(e)}) x = \mathsf d_e x - \id_{X(e)} (\mathsf d_{e} x) =0
		\]
		as required in Definition \ref{dg_fc_defn} (3).
		
		The composition for the 2-cells can be illustrated as follows:
		\[
		\xymatrix{
			v_n  \ar@{.}[d]  \ar@{}[dr]|{\Downarrow \  \id } & {\quad} \ar[l] & u_m \ar@{.}[d]  &  \cdots  \ar[l]_{X(f_m)} \ar@{}[d]|{\Downarrow \  \xi_2 } & u_0 \ar[l]_{X(f_1)} \ar@{.}[d]  & \quad &  v_0 \ar@{.}[d]  \ar[l] \ar@{}[dl]|{\Downarrow \  \id } \\
			v_n \ar@{.}[d]  & \cdots \ar[l] & v_i  \ar[l] &  & v_{i-1} \ar[ll] & \cdots \ar[l] & \ar[l] v_0 \ar@{.}[d]   \\
			v_1'  & & & \ar@{}[u]|{\Downarrow \  \xi_1}  & & & v_0' \ar[llllll]^{X(e')}
		}
		\]
		Specifically, given $\xi_1 \in\widehat{\mathbf{End}}(X)(e_1,\dots, e_n; e')$ and $\xi_2 \in \widehat{\mathbf{End}}(X)(f_1,\dots, f_n; e_i)$ for a fixed $1\le i\le n$, we define
		\begin{equation}
			\label{End_X_hat_composition}
			\begin{aligned}
				& \xi_1\circ_i \xi_2  \ (x_1\otimes \cdots \otimes x_{i-1} \otimes y_1\otimes \cdots \otimes y_m \otimes x_{i+1}\otimes \cdots \otimes x_n) \\ 
				& \qquad \qquad =
				(-1)^{|\xi_2| (|x_1|+\cdots+|x_{i-1}|)}
				\xi_1 (x_1\otimes \cdots \otimes x_{i-1} \otimes \xi_2(y_1\cdots y_m) \otimes x_{i+1}\otimes \cdots \otimes x_n)
			\end{aligned}
		\end{equation}
		as the induced cochain map
		\[
		X(e_1)\otimes \cdots \otimes X(e_{i-1})\otimes X(f_1)\otimes \cdots \otimes X(f_m) \otimes X(e_{i+1})\otimes \cdots \otimes X(e_n) \to X(e') .
		\]
		By routine computation, one can eventually check Definition \ref{dg_fc_defn} (4), that is, 
		\[
		\widehat{\mathsf d}(\xi_1\circ_i \xi_2) = \widehat{\mathsf d} \xi_1 \circ \xi_2   + (-1)^{|\xi_1|} \xi_1 \circ \widehat{\mathsf d}\xi_2
		\]
	\end{proof}

\subsection{Algebras over $\mathbf{fc}$-multicategories: dg variants}
	Now, we introduce a dg analogue of Definition \ref{C_alg_defn_End}:
	
	\begin{defn}
		\label{C_alg_defn_End_dg}
		Let $\widehat{\mathscr M}=(\mathscr M; \delta)$ be a dg $\mathbf{fc}$-multicategory.
		A \textit{$\widehat{\mathscr M}$-algebra} is defined as a pair $(X,\pmb \alpha)$ with
		\vspace{2pt}
		\begin{itemize}
			\itemsep 4pt
			\item a simple directed graph $X=(V, X, s\circ p_X, t\circ p_X, p_X)$ over the directed graph $E=(V,E,s,t)$ so that
			$X(e)=p_X^{-1}(e)$ is a cochain complex for $e\in E$ with the differential denoted by $\mathsf d_e=\mathsf d_e^X$.
			
			\item a map of dg $\mathbf{fc}$-multicategories $\pmb\alpha: \widehat{\mathscr M} \to \widehat{\mathbf{End}}(X)$ such that the induced map of directed graphs $(V,E)\to (V,E)$ is the identity. 
		\end{itemize}
	\end{defn}

	\vspace{5pt}
	
	Let's unpack the data of a $\widehat{\mathscr M}$-algebra $(X,\pmb\alpha)$ for a dg $\mathbf{fc}$-multicategory $\widehat{\mathscr M} = (\mathscr M; \delta)$ where $\mathscr M=(V,E, \mathcal V, \mathcal M)$.
Let's write $\pmb\alpha= (\id_V,\id_E, \bar\alpha, \alpha)$ by Definition \ref{map_fc_defn} \& \ref{map_fc_dg_defn}.
Recall that by Proposition \ref{End_E_X_widehat_prop}, the vertex set of $\widehat{\mathbf{End}}(X)$ is $V\times V$.
For $v, v'\in V$, the image of a vertical 1-cell $g\in\mathcal V$ from $v$ to $v'$ under $\bar \alpha$ is simply $(v,v')$; in other words, $\bar\alpha= (\bar d, \bar c): \mathcal V\to V\times V$ where $d,c$ are the vertex parts of $\pmb d, \pmb c$.
For $\vec e=(e_1,\dots, e_n)\in E^*$ and $e'\in E$, the image of each 2-cell $\mathbf u$ over $(\vec e, e')\in E^*\times E$ under $\alpha$ is a cochain map denoted by
	\[
	\alpha(\mathbf u) : X(e_1)\otimes \cdots \otimes X(e_n) \to X(e')
	\]
By Definition \ref{map_fc_dg_defn}, each component
	\[
	\alpha =\alpha(e_1,\dots, e_n; e'): \mathcal M(e_1,\dots, e_n; e') \to \widehat{\mathbf{End}}(X) (e_1,\dots, e_n; e')
	\]
	is a cochain map, that is,
	\begin{equation}
		\label{alpha_u_delta_eq}
		\widehat{\mathsf d} \alpha(\mathbf u) = \alpha (\delta \mathbf u)
	\end{equation}
Note also that 
\begin{equation}
	\label{alpha_u_composition}
	\alpha (\mathbf u\circ_i \mathbf u') = \alpha(\mathbf u) \circ_i \alpha(\mathbf u')
\end{equation}
where the partial composition on the right hand side is given in \eqref{End_X_hat_composition}.

	\section{$A_\infty$-type structures as algebras over dg $\mathbf{fc}$-multicategories}
\label{s_Ainf_type}

This section forms a major portion of the paper. Its purpose is to collect a large range of $A_\infty$-type structures appearing in the literature and recast them in a unified conceptual framework, namely as algebras over dg $\mathbf{fc}$-multicategories in the sense of Definition~\ref{C_alg_defn_End_dg}.
That is to say, we aim to extend the operadic description \eqref{operad_A_inf_intro_eq} to general $A_\infty$-type structures.
In particular, the constructions and descriptions below will resolve Theorem \ref{dg_fc_thm_intro}.

	\subsection{$A_\infty$ algebras}
	\label{s_A_inf_algebra}
	
	We begin with a warm-up case.
	Consider the \textit{$A_\infty$ operad }$\mathcal A_\infty=\{\mathcal A_\infty(n)\}_{n\ge 0}$. 
	Recall that $\mathcal A_\infty$ as an operad is freely generated by elements $\mathbf m_n$ with degree $|\mathbf m_n|=1$ for $n\ge 2$; see \cite{markl1998homotopy}.
	While the usual convention is that the degree of $\mathbf m_n$ should be $2-n$, one could use the so-called \textit{shifted degree} as in \cite{FOOOBookOne,Yuan_I_FamilyFloer,Yuan_2024open} to make $|\mathbf m_n|=1$.	
	The differential is determined on generators by
	\[
	\delta (\mathbf m_n)\;=\;-\!\!\sum_{\substack{r+s+t=n; \ r,s,t\ge 0}}
	\, \mathbf m_{r+1+t}\circ_{r+1} \mathbf m_s,
	\]
	for $n\ge 2$ and extended to all composites by the operadic Leibniz rule (\ref{operadic_leibniz_rule}).

	By Example \ref{dg_operad_as_dg_fc_mult_ex}, the dg operad $(\mathcal A_\infty,\delta)$ can be regarded as a dg $\mathbf {fc}$-multicategory $\Sigma\mathcal A_\infty$.
	Nonetheless, as mentioned in Section \ref{s_additional_grading}, incorporating an additional grading is often useful in symplectic applications. Accordingly, we follow Definition \ref{S_labeled_multicategory_defn} to formulate the following variant:
	
	Let $(S, + ,\theta)$ be a commutative monoid.
	We define the \textit{$S$-labeled $A_\infty$ operad} (cf. Definition \ref{S_labeled_multicategory_defn})
	\[
	\mathcal A_\infty^S = \{ \mathcal A^S_\infty(n,\beta)\}_{n\ge 0,\beta\in S}
	\]
	as follows: it is freely generated by symbols $\mathbf m_{n,\beta}$ with $|\mathbf m_{n,\beta}|=1$ for $(n,\beta)\neq (0,\theta), (1,\theta)$.
	The differential is decided on generators by
	\begin{equation}\label{eq:Ainf-d}
		\delta (\mathbf m_{n,\beta})\;=\;-\!\!\sum_{\substack{r+s+t=n ; \ r, s, t \ge 0 \\ \beta'+\beta''=\beta}}
		\, \mathbf m_{r+1+t,\beta'}\circ_{r+1} \mathbf m_{s,\beta''},
	\end{equation}
	and extended by the operadic Leibniz rule. This is a differential since one can check
	\[
	\delta^2(\mathbf m_{n,\beta}) = - \sum \delta (\mathbf m_{r+1+t,\beta'})\circ_{r+1} \mathbf m_{s,\beta''} -(-1)^{|\mathbf m|} \ \mathbf m_{r+1+t,\beta'} \circ_{r+1} \delta(\mathbf m_{s,\beta''}) =0
	\]
	
	Set $\mathcal A_\infty^S(n)=\bigoplus_{\beta\in S} \mathcal A_\infty^S(n,\beta)$, and we can view $\mathcal A_\infty^S$ as a dg $\mathbf{fc}$-multicategory from Example \ref{dg_operad_as_dg_fc_mult_ex}; we can also view the monoid $S$ as a $\mathbf{fc}$-multicategory from Example \ref{ex_S_E_labeling} \& \ref{ex_trivial_labeling} so that $\mathcal A_\infty^S$ is an $S$-labeled $\mathbf{fc}$-multicategory in the sense of Definition \ref{label_by_fc_multicategory_defn}.

\begin{defn}
\label{Ainf_alg_S_labeled_defn}
		By an \textit{$S$-labeled $A_\infty$ algebra}, we mean a $\mathcal A_\infty^S$-algebra $(X,\pmb\alpha)$ in Definition \ref{C_alg_defn_End_dg}.
\end{defn}

	Specifically, since the 0-cell set and 1-cell set of $\mathcal A_\infty^S$ are one-point sets and since $X$ is simple, we note that 
	$X$ can be just regarded as a single cochain complex whose differential is denoted by $\mathsf d=\dd^X$. Also, $\widehat{\mathbf{End}}(X)$ can be identified with the usual endomorphism dg operad consisting of cochain maps $X^{\otimes \bullet} \to X$.

	Define $\m_{1,\theta}=\mathsf d$ and 
\begin{equation}
\label{m_n_beta_Ainf_alg_eq}
	\m_{n,\beta} :=  \alpha (\mathbf m_{n,\beta})  \ : \ \ X^{\otimes n} \to X
\end{equation}
for $n\ge 0$ and $\beta\in S$ with $(n,\beta)\neq (1,\theta)$.
	When $n=0$, we conventionally set $X^{\otimes n}$ to be the ground field, and $\m_{0,\beta}$ is identified with an element in $X$.
	By definition, the degree of $\alpha$ is zero, and thus $|\m|=|\m_{n,\beta}|=1$.
	
	The condition (\ref{alpha_u_delta_eq}) produces:
	\begin{align*}
		\alpha (\delta \mathbf m_{n,\beta}) (x_1,\dots, x_n)
		&=
		\dd (\m_{n,\beta} (x_1,\dots, x_n)) - \sum_{i=1}^n (-1)^{|\m|+|x_1|+\cdots+|x_{i-1}|} \  \m_{n,\beta} (x_1,\dots,  \dd x_i, \dots, x_n) \\
		&=
		\m_{1,0} (\m_{n,\beta} (x_1,\dots, x_n)) + \sum_{i=1}^n (-1)^{|x_1|+\cdots+|x_{i-1}|} \  \m_{n,\beta} (x_1,\dots,  \m_{1,0}( x_i), \dots, x_n)
	\end{align*}
	Using \eqref{eq:Ainf-d}, \eqref{alpha_u_composition}, and \eqref{End_X_hat_composition}, we obtain
	\begin{align*}
		& \quad \alpha (\delta \mathbf m_n) (x_1,\dots, x_n) \\
		&=
		-\sum \alpha(\mathbf m_{r+1+t,\beta'} \circ_{r+1} \mathbf m_{s,\beta''} ) (x_1,\dots, x_n) 
		= -\sum (\m_{r+1+t,\beta'} \circ_{r+1} \m_{s,\beta''})  (x_1,\dots, x_n)  \\
		&=
		-\sum
		(-1)^{|\m|(|x_1|+\cdots+|x_r|)}
		\m_{ r+1+t,\beta'} (x_1,\dots, x_r,  \m_{s,\beta''} (x_{r+1},\dots, x_{r+s}) , \dots, x_n)
	\end{align*}
	where the right hand side does not involve $\m_{1,\theta}$ by construction.
	Putting them together, we exactly obtain the $A_\infty$ relations (with labels): for each $\beta\in S$ and $n\ge 0$,
	\begin{equation}
		\label{A_inf_alg_relation}
		\begin{aligned}
			\sum_{\beta'+\beta''=\beta} \sum_{r+s+t=n} (-1)^{|x_1|+\cdots+|x_r|}
			\m_{r+1+t,\beta'} (x_1,\dots, x_r, \m_{s,\beta''}(x_{r+1},\dots, x_{r+s}),\dots, x_n) = 0
		\end{aligned}
	\end{equation}

\begin{rmk}
\label{curved_Ainf_algebra_rmk}
When $S$ is the trivial monoid, the operad $\mathcal A_\infty^{S}$ recovers the usual $A_\infty$ operad in \eqref{operad_A_inf_intro_eq}.
In the standard literature \cite{FOOOBookOne}, the notion of a \textit{filtered $A_\infty$ algebra} can be viewed as an $S$-labeled $A_\infty$ algebra, where $S\subset \mathbb R_{\ge 0}$ is a \textit{nontrivial} discrete additive monoid.
\end{rmk}

\subsection{$A_\infty$ categories}
\label{s_Ainf_categories}

Let $V$ be a set. Let $E=(V,E,s,t)$ be the directed graph whose vertex set is $V$ and whose edge set is
$E=V\times V$,
so that for each ordered pair $(v,v’)\in V\times V$ there is a unique edge $e_{v,v'}$ from $v$ to $ v'$.
For each $v\in V$ we distinguish the loop $e_v:=e_{v,v}$. Let $\mathbb S$ be any fixed $\mathbf{fc}$-multicategory whose underlying directed graph is the given $E=(V,E)$.

We will construct a vertically discrete $\mathbb S$-labeled dg $\mathbf{fc}$-multicategory
\begin{equation}
\label{A_inf_V_cat_eq}
\mathcal A_{\infty,V}^{\mathbb S} := \mathcal A_{\infty, V\times V}^{\mathbb S}=(V,E=V\times V,\mathcal M;\delta)
\end{equation}
whose underlying directed graph of $0$-cells and horizontal $1$-cells is $(V,E)$.
We define $\mathcal M$ to be freely generated by symbols $\mathbf m_\beta \in \mathcal M(\vec e; e')$, for each $(\vec e; e')\in E^*\times E$ and each $\beta\in \mathbb S(\vec e; e')$, \textit{excluding the unit case $\beta=\theta_e\in \mathbb S(e; e)$}.
We require the degree of each $\mathbf m_\beta$ is one, and we define the differential on $\mathcal M$ by
\begin{equation}
\label{eq:Ainf-category_d}
\delta(\mathbf m_\beta)\;=\;-\sum \mathbf m_{\beta'} \circ_i \mathbf m_{\beta''}
\end{equation}
where the sum ranges over all possible decompositions $\beta=\beta'\circ_i \beta''$ for some $i$ and some $\beta'$,$\beta''$.

\begin{defn}
\label{Ainf_cat_S_labeled_defn}
An \textit{$\mathbb S$-labeled $A_\infty$ category} with object set $V$ is an $\mathcal A_{\infty, V}^{\mathbb S}$-algebra (Definition~\ref{C_alg_defn_End_dg}).
\end{defn}

In the special case $V$ is a one-point set, the equation \eqref{eq:Ainf-category_d} exactly retrieves \eqref{eq:Ainf-d}, and the above may also recover Definition \ref{Ainf_alg_S_labeled_defn}.

\begin{rmk}
In the literature, an $A_\infty$-type structure is called \textit{curved} (respectively \textit{uncurved}) depending on whether one allows (respectively forbids) operations with empty input. In our terminology, this corresponds to whether $\mathbb S(\varnothing; e')$ is allowed to be nonempty (respectively required to be empty).
\end{rmk}

Let $\mathcal A_{\infty, V}^{red}$ denote the dg $\mathbf{fc}$-multicategory in \eqref{A_inf_V_cat_eq} when $\mathbb S=\pmb 0^{red}_E$ in Example \ref{ex_trivial_labeling}, where $E=V\times V$ as above. In particular, by definition, $\mathbb S(\varnothing; e')$ and thus $\mathcal M(\varnothing; e')$ are empty.
If instead we take $\mathbb S=\pmb 0_E$ and write the resulting dg $\mathbf{fc}$-multicategory in \eqref{A_inf_V_cat_eq} as $\mathcal A_{\infty,V}$, then we propose to define the notion of a \textit{curved} $A_\infty$ category with object set $V$ as an $\mathcal A_{\infty,V}$-algebra. In contrast, the notion of $\mathcal A_{\infty,V}^{\mathrm{red}}$-algebras recover the (uncurved) notion of $A_\infty$ category, as shown in the next theorem.

\begin{prop}
\label{standard_A_inf_cat_prop}
An $\mathcal A_{\infty,V}^{red}$-algebra is equivalent to an (uncurved) $A_\infty$ category with object set $V$ (e.g. \cite[(1a)]{SeidelBook}).
\end{prop}	

\begin{proof}
Let $(X,\pmb\alpha)$ be an $\mathcal A^{red}_{\infty, V}$-algebra.
By Definition \ref{C_alg_defn_End_dg}, the data $X$ consists of cochain complexes $X(v,v')$ with the differential denoted by $\mathsf d_{v,v'}$ for every $(v,v')\in E=V\times V$.
By Example \ref{ex_trivial_labeling}, the set $\pmb 0^{\mathrm{red}}_E(\vec e; e')$ is either empty or a singleton, and the latter occurs if and only if $\vec e \neq \varnothing$.
Therefore, the generators $\mathbf m_\beta$ of $\mathcal A_{\infty, V}$ can equivalently be indexed as $\mathbf m_{\vec e; e'}$ for all $\vec e\in E^*$ and $e'\in E$ where $\vec e\in E^*$ and $e'\in E$ satisfy $\vec e\neq\varnothing$ and the source and target of $\vec e$ agree with those of $e'$.
Here by definition, we also need to require $\vec e\neq e'$ to exclude the unit case.
Since $E=V\times V$, it follows from \eqref{E*_eq} that
\[
E^* = \bigsqcup_{n} \left\{\Big( (v_0,v_1),(v_1,v_2),\dots, (v_{n-1},v_n) \Big) \in (V\times V)^{\times n} \mid v_0,\dots, v_n\in V  \right\} \cong \bigsqcup_n V^{\times (n+1)} \ .
\]
Accordingly, the generators can be further indexed as $\mathbf m_{\vec v}$ for $\vec v=(v_0,v_1,\dots, v_n) \in V^{\times (n+1)}$, where we need to require $n\neq 1$ to exclude the unit case.

We put $\mu_1=\mu_{1; v_0,v_1}=\mu_{v_0,v_1} =\dd_{v_0,v_1}$ for all $v_0,v_1\in V$.
By definition, the data $\pmb\alpha$ also consists of cochain maps
\[
\mu_n= \mu_{n; \vec v} = \mu_{\vec v} = \alpha (\mathbf m_{\vec v}) : X(v_0,v_1)\otimes \cdots \otimes X(v_{n-1},v_n) \to X(v_0,v_n) 
\qquad n \ge 2
\]
Finally, it remains to recover the standard $A_\infty$ relation. 
In reality, we fix $\vec v=(v_0,v_1,\dots, v_n)$ with $n \ge 2$ and and $x_k\in X(v_{k-1},v_k)$ for $k=1,\dots, n$.
By \eqref{alpha_u_delta_eq} and \eqref{widehat_d_eq}, we compute
\begin{equation}
\label{k=2,n+1_eq}
	\begin{aligned}
&\alpha(\delta \mathbf m_{\vec v}) (x_1,\dots, x_n) = (\widehat{\dd} \alpha) (\mathbf m_{\vec v}) \\
& = \mu_1 \left( \mu_{n} (x_1,\dots, x_n) \right)  + \sum_{k=1}^n (-1)^{|x_1|+\cdots+|x_{k-1}|} \ \mu_{n} (x_1,\dots, x_{k-1}, \mu_1 (x_k) ,x_{k+1},\dots, x_n) 
\end{aligned}
\end{equation}
In our case, the relation \eqref{eq:Ainf-category_d} becomes 
\[
\delta \mathbf m_{\vec v} = - \sum_{1\neq k\neq n} \sum_j \mathbf m_{v_0,v_1,\dots, v_{j-1},v_{j+k+1},\dots, v_n} \circ_{j+1} \mathbf m_{v_{j},\dots, v_{j+k}}
\]
By \eqref{alpha_u_composition} and \eqref{End_X_hat_composition}, we have
\begin{align*}
-\alpha(\delta\mathbf m_{\vec v}) (x_1,\dots, x_n) 
&= \sum_{1\neq k\neq n} \sum_j (\mu_{n-k} \circ_{j+1} \mu_{k+1} ) (x_1,\dots, x_n) \\
&= \sum_{1\neq k\neq n} \sum_j (-1)^{|x_1|+\cdots+|x_{j-1}|} \  \mu_{n-k}  (x_1,\dots, x_{j-1} , \mu_{k+1}(x_{j},\dots, x_{j+k}),\dots, x_n)
\end{align*}
Adding this with (\ref{k=2,n+1_eq}), we obtain the desired $A_\infty$ relation:
\[
\sum_{ k} \sum_j (-1)^{|x_1|+\cdots+|x_{j-1}|} \  \mu_{n-k}  (x_1,\dots, x_{j-1} , \mu_{k+1}(x_{j},\dots, x_{j+k}),\dots, x_n) = 0
\]
\end{proof}

\subsection{Generalized $A_\infty$ categories}
\label{s_generalized_Ainf}
In Section \ref{s_Ainf_categories}, we fix the edge set to be $E = V \times V$. However, this assumption is not necessary in general: one may allow a more general directed graph $E$ with the vertex set $V$. In such cases, $A_\infty$-type structures can still arise, for which we temporarily use the name "generalized $A_\infty$ categories". The language of $\mathbf{fc}$-multicategories and their algebras may provide a natural framework for describing such generality.
More precisely, we introduce:

\begin{defn}
\label{generalized_Ainf_defn}
Given an arbitrary directed graph $E=(V,E)$ and an $\mathbf{fc}$-multicategory $\mathbb S$ whose underlying directed graph is $E$, the constructions in \eqref{A_inf_V_cat_eq} and \eqref{eq:Ainf-category_d} apply in the same way. We denote the resulting $\mathbf{fc}$-multicategory by $\mathcal A_{\infty,E}^{\mathbb S}$.
When $\mathbb S=\pmb 0_E$ as in Example \ref{ex_trivial_labeling}, we also denote it by $\mathcal A_{\infty, E}$.
\end{defn}

Our basic claim is that many different $A_\infty$-type structures in the literature can be realized as algebras over $\mathcal A_{\infty,E}$ for suitable choices of the directed graph $E$.

\subsection{Left and right $A_\infty$ modules}
Let $E=(V,V\times V)$ be the directed graph from Section \ref{s_Ainf_categories}.
Define a new directed graph
\[
E_{l} := (V\sqcup \{\ast\} , (V\times V) \sqcup (\{\ast\} \times V) )
\]
Concretely, we add a new vertex denoted by $\ast$, and then for each $v\in V$, we add a new directed edge from $\ast$ to $v$.
However, note that there is no loop at the vertex $\ast$.
Similarly, define a directed graph
\[
E_{r} := (V\sqcup \{\ast'\} , (V\times V) \sqcup (V\times \{\ast'\}) )
\]
Here we add a new vertex $\ast'$ and for each $v\in V$, we add a new directed edge from $v$ to $\ast'$.

By Definition \ref{generalized_Ainf_defn}, we introduce the dg $\mathbf{fc}$-multicategories:
\[
\mathcal A_{\infty,V, l}:= \mathcal A_{\infty, E_{l}} \quad \text{and}\quad
\mathcal A_{\infty,V, r}:= \mathcal A_{\infty, E_{r}}
\] 
The similar arguments can imply the following:

\begin{prop}
An $\mathcal A_{\infty,V,l}$-algebra (resp. $\mathcal A_{\infty,V,r}$-algebra) is equivalent to the standard notion of a left (resp. right) $A_\infty$ module over an $A_\infty$ category with object set $V$ in the literature (see e.g. \cite[Definitions 2.9 \& 2.10]{Ganatra_thesis}).
\end{prop}

\subsection{$A_\infty$ bimodules over a pair of $A_\infty$ algebras}
\label{s_Ainf_bimodule_alg}
Consider the directed graph $E=(V,E,s,t)$ where $V=\{v_0,v_1\}$ is a two-element set and $E=\{e_0,e_1,e_{01}\}$ is a three-element set. Here $e_0$ and $e_1$ are loops based at $v_0$ and $v_1$, respectively, and $e_{01}$ is a directed edge from $v_0$ to $v_1$.
	Note that there is no directed edge from $v_1$ to $v_0$ (see Figure \ref{fig:A_inf_bimodule}).
	Let $\mathbb S=(V,E,\mathbb S)$ be a fixed $\mathbf{fc}$-multicategory with the same underlying directed graph $E=(V,E)$, where the collection of 2-cells is still denoted by $\mathbb S$. 
	
	We construct a vertically discrete $\mathbb S$-labeled dg $\mathbf{fc}$-multicategory (Definition \ref{label_by_fc_multicategory_defn}, \ref{dg_fc_defn})
	\begin{equation}
		\label{A_inf_2_bimodule_eq}
		\mathcal A_{\infty,\text{2-mod}}^{\mathbb S} := \mathcal A^{\mathbb S}_{\infty, \{e_0,e_1,e_{01}\}}=(V,E,\mathcal M,\delta)
	\end{equation}
as a special case of the one in Definition \ref{generalized_Ainf_defn}, by taking the directed graph to be $\{e_0,e_1,e_{01}\}$ described above.
Specifically,
the sets of 0-cells and horizontal 1-cells are respectively defined as $V=\{v_0,v_1\}$ and $E=\{e_0,e_1,e_{01}\}$.
A profile-loop (Definition \ref{profile_loop_defn}) in the graph $E$ can only have the following three possibilities:
	\begin{itemize}
		\itemsep 2pt
		\item  $ v_0\xleftarrow{e_0}\cdots  \xleftarrow{e_0}v_0$
		\item  $ v_1\xleftarrow{e_1}\cdots  \xleftarrow{e_1}v_1$
		\item 
		$ v_1\xleftarrow{e_1}\cdots  \xleftarrow{e_1}v_1
		\xlongleftarrow{\quad e_{01}\quad } v_0 \xleftarrow{e_0}v_0 \xleftarrow{} \cdots  \xleftarrow{e_0}v_0$
	\end{itemize}
	Therefore, by \eqref{M_decompose_profile_eq}, the non-empty parts of 2-cells have only three possible types as follows:
	\begin{align*}
		\mathcal M(\underbrace{e_0, \dots, e_0}_{n} ; e_0 ) , \qquad 
		\mathcal M(\underbrace{e_1, \dots, e_1}_{n} ; e_1 ) , \qquad  \mathcal M(\underbrace{e_0, \dots, e_0}_{n_0} \ , \ e_{01} \ , \ \underbrace{e_1, \dots, e_1}_{n_1} \ ; \ e_{01} )
	\end{align*}

	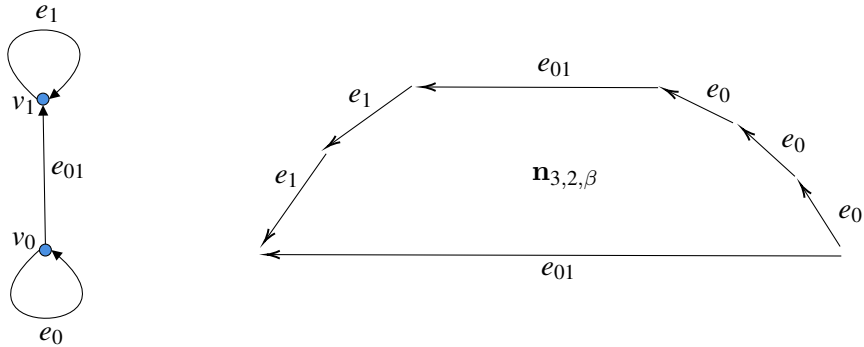
\begin{figure}[h]
	\centering

	\begin{tikzpicture}[x=0.75pt,y=0.75pt,yscale=-0.6,xscale=0.6]
		
		\draw    (771,255) -- (290.09,253.82) ;
		\draw [shift={(288.09,253.82)}, rotate = 0.14] [color={rgb, 255:red, 0; green, 0; blue, 0 }  ][line width=0.75]    (10.93,-3.29) .. controls (6.95,-1.4) and (3.31,-0.3) .. (0,0) .. controls (3.31,0.3) and (6.95,1.4) .. (10.93,3.29)   ;
		\draw    (770,247) -- (738.07,196.69) ;
		\draw [shift={(737,195)}, rotate = 57.6] [color={rgb, 255:red, 0; green, 0; blue, 0 }  ][line width=0.75]    (10.93,-3.29) .. controls (6.95,-1.4) and (3.31,-0.3) .. (0,0) .. controls (3.31,0.3) and (6.95,1.4) .. (10.93,3.29)   ;
		\draw    (732,188) -- (689.48,149.35) ;
		\draw [shift={(688,148)}, rotate = 42.27] [color={rgb, 255:red, 0; green, 0; blue, 0 }  ][line width=0.75]    (10.93,-3.29) .. controls (6.95,-1.4) and (3.31,-0.3) .. (0,0) .. controls (3.31,0.3) and (6.95,1.4) .. (10.93,3.29)   ;
		\draw    (680,143) -- (624.8,115.88) ;
		\draw [shift={(623,115)}, rotate = 26.16] [color={rgb, 255:red, 0; green, 0; blue, 0 }  ][line width=0.75]    (10.93,-3.29) .. controls (6.95,-1.4) and (3.31,-0.3) .. (0,0) .. controls (3.31,0.3) and (6.95,1.4) .. (10.93,3.29)   ;
		\draw    (617,113.41) -- (421.09,112.82) ;
		\draw [shift={(419.09,112.82)}, rotate = 0.17] [color={rgb, 255:red, 0; green, 0; blue, 0 }  ][line width=0.75]    (10.93,-3.29) .. controls (6.95,-1.4) and (3.31,-0.3) .. (0,0) .. controls (3.31,0.3) and (6.95,1.4) .. (10.93,3.29)   ;
		\draw    (410,112) -- (340.62,161.83) ;
		\draw [shift={(339,163)}, rotate = 324.31] [color={rgb, 255:red, 0; green, 0; blue, 0 }  ][line width=0.75]    (10.93,-3.29) .. controls (6.95,-1.4) and (3.31,-0.3) .. (0,0) .. controls (3.31,0.3) and (6.95,1.4) .. (10.93,3.29)   ;
		\draw    (338,169) -- (287.15,241.36) ;
		\draw [shift={(286,243)}, rotate = 305.1] [color={rgb, 255:red, 0; green, 0; blue, 0 }  ][line width=0.75]    (10.93,-3.29) .. controls (6.95,-1.4) and (3.31,-0.3) .. (0,0) .. controls (3.31,0.3) and (6.95,1.4) .. (10.93,3.29)   ;
		\draw  [fill={rgb, 255:red, 74; green, 144; blue, 226 }  ,fill opacity=1 ] (95,123) .. controls (95,120.24) and (97.24,118) .. (100,118) .. controls (102.76,118) and (105,120.24) .. (105,123) .. controls (105,125.76) and (102.76,128) .. (100,128) .. controls (97.24,128) and (95,125.76) .. (95,123) -- cycle ;
		\draw  [fill={rgb, 255:red, 74; green, 144; blue, 226 }  ,fill opacity=1 ] (97,250) .. controls (97,247.24) and (99.24,245) .. (102,245) .. controls (104.76,245) and (107,247.24) .. (107,250) .. controls (107,252.76) and (104.76,255) .. (102,255) .. controls (99.24,255) and (97,252.76) .. (97,250) -- cycle ;
		\draw    (100.05,131) -- (102,245) ;
		\draw [shift={(100,128)}, rotate = 89.02] [fill={rgb, 255:red, 0; green, 0; blue, 0 }  ][line width=0.08]  [draw opacity=0] (8.93,-4.29) -- (0,0) -- (8.93,4.29) -- cycle    ;
		\draw    (95,123) .. controls (5.45,44.39) and (196.08,46.97) .. (106.37,121.87) ;
		\draw [shift={(105,123)}, rotate = 320.74] [fill={rgb, 255:red, 0; green, 0; blue, 0 }  ][line width=0.08]  [draw opacity=0] (8.93,-4.29) -- (0,0) -- (8.93,4.29) -- cycle    ;
		\draw    (97,250) .. controls (8.44,328.61) and (198.09,324.05) .. (108.37,251.1) ;
		\draw [shift={(107,250)}, rotate = 38.51] [fill={rgb, 255:red, 0; green, 0; blue, 0 }  ][line width=0.08]  [draw opacity=0] (8.93,-4.29) -- (0,0) -- (8.93,4.29) -- cycle    ;
		
		\draw (514.87,257.41) node [anchor=north west][inner sep=0.75pt]    {$e_{01}$};
		\draw (511.87,86.41) node [anchor=north west][inner sep=0.75pt]    {$e_{01}$};
		\draw (716.87,150.41) node [anchor=north west][inner sep=0.75pt]    {$e_{0}$};
		\draw (764.87,209.41) node [anchor=north west][inner sep=0.75pt]    {$e_{0}$};
		\draw (653.87,108.41) node [anchor=north west][inner sep=0.75pt]    {$e_{0}$};
		\draw (352.87,112.41) node [anchor=north west][inner sep=0.75pt]    {$e_{1}$};
		\draw (290.87,180.41) node [anchor=north west][inner sep=0.75pt]    {$e_{1}$};
		\draw (506.87,178.41) node [anchor=north west][inner sep=0.75pt]    {$\mathbf{n}_{3,2,\beta}$};
		\draw (70,232.4) node [anchor=north west][inner sep=0.75pt]    {$v_{0}$};
		\draw (70,118.4) node [anchor=north west][inner sep=0.75pt]    {$v_{1}$};
		\draw (102.87,172.41) node [anchor=north west][inner sep=0.75pt]    {$e_{01}$};
		\draw (89.87,41.41) node [anchor=north west][inner sep=0.75pt]    {$e_{1}$};
		\draw (92.87,312.41) node [anchor=north west][inner sep=0.75pt]    {$e_{0}$};

	\end{tikzpicture}
	\caption{Left: The underlying directed graph of $\mathcal A_{\infty, \text{2-mod}}^{\mathbb S}$. \quad Right: The 2-cell corresponding to the generator $\mathbf n_{n_0,n_1,\beta }$.}
	\label{fig:A_inf_bimodule}
\end{figure}

Denote by $\theta_e\in\mathbb S(e; e)$ the identity 2-cell at $e\in E$.
Abusing the notations, the composition of 2-cells in $\mathbb S$ will be often denoted by $\beta+\beta' :=\beta\circ_i\beta'$.
The symbols $\mathbf m_\beta$ around \eqref{eq:Ainf-category_d} can be more explicitly described as follows:
The $2$-cells are freely generated by the non-identity $2$-cells of $\mathbb S$ in the sense that we take $\mathcal M$ to be freely generated by the following symbols:
	\vspace{3pt}
	\begin{itemize}
		\itemsep 2pt
		\item $\mathbf m_{n,\beta}^{(0)}$ in $\mathcal M((e_0)^{n}; e_0)$ for every $\beta \in \mathbb S( (e_0)^n; e_0)$ with $\beta\neq \theta_{e_0}$;
		\item  $\mathbf m_{n,\beta}^{(1)}$ in $\mathcal M((e_1)^{n}; e_1)$ for every $\beta \in \mathbb S( (e_1)^n; e_1)$ with $\beta\neq \theta_{e_1}$;
		\item  $\mathbf n_{n_0,n_1,\beta}$ in $\mathcal M((e_0)^{n_0}, e_{01}, (e_1)^{n_1}; e_{01})$ for every $\beta\in\mathbb S((e_0)^{n_0}, e_{01}, (e_1)^{n_1}; e_{01})$ with $\beta \neq \theta_{e_{01}}$.
	\end{itemize}

	\vspace{5pt}
Here the numbers $n, n_0,n_1$ are redundant since $\beta$ determines them, and the conditions may be more precisely written as $(n,\beta) \neq (1,\theta_{e_0})$, $(n,\beta)\neq (1,\theta_{e_1})$, $(n_0,n_1,\beta)\neq (0,0,\theta_{e_{01}})$.
	We also require that their degrees are all equal to one: $|\mathbf m_{n,\beta}^{(i)}|=|\mathbf n_{n_0,n_1,\beta}|=1$. Introduce the differentials $\delta=\{\delta_{\vec e; e'}\}$ on the components of $\mathcal M$ defined first on the above generators by 
	\[
	\delta (\mathbf m^{(i)}_{n,\beta})\;=\;-\!\!\sum_{\substack{r+s+t=n ; \ r, s, t \ge 0 \\ \beta'+\beta''=\beta}}
	\, \mathbf m^{(i)}_{r+1+t,\beta'}\circ_{r+1} \mathbf m^{(i)}_{s,\beta''},
	\]
	\begin{align*}
		\delta(\mathbf n_{n_0,n_1,\beta})
		= \
		&-\sum_{\substack{r_0+s_0+t_0=n_0 \\ \beta'+\beta''=\beta}} \  \mathbf n_{r_0+1+t_0, n_1,\beta'} \circ_{r_0+1} \mathbf m^{(0)}_{s_0,\beta''} \\
		&- \sum_{\substack{r_1+s_1+t_1=n_1 \\ \beta'+\beta''=\beta}} \ \mathbf n_{n_0, r_1+1+t_1,\beta'} \circ_{n_0+r_1+2} \mathbf m^{(1)}_{s_1,\beta''} \\
		&- \sum_{\substack{n_0'+n_0''=n_0 \\ n_1'+n_1''=n_1 \\ \beta'+\beta''=\beta}} \mathbf n_{n_0',n_1',\beta'} \circ_{n_0'+1} \mathbf n_{n_0'',n_1'',\beta''}
	\end{align*}
	and then extended by the Leibniz-type rule $\delta (\mathbf u \circ_i \mathbf u')= (\delta\mathbf u)\circ_i \mathbf u' \ +\ (-1)^{|\mathbf u|}\,\mathbf u\circ_i (\delta\mathbf u')$.
	This completes the construction of (\ref{A_inf_2_bimodule_eq}).

	Let $(X,\pmb\alpha)$ be an algebra over the dg $\mathbf{fc}$-multicategory $\mathcal A_{\infty, \text{2-mod}}^{\mathbb S}$ in the sense of Definition \ref{C_alg_defn_End_dg}, so we have a map of dg $\mathbf{fc}$-multicategories:
\[ \pmb\alpha : \ \mathcal A_{\infty, \text{2-mod}}^{\mathbb S} \to \widehat{\mathbf{End}}(X)
\]
	Let's unpack it as follows.
	Note first that $X$ consists of three cochain complexes $(X_i,\dd_i ) :=(X(e_i), \dd_{e_i})$ for $i=0,1$ and $(X_{01}, \dd_{01}) := (X(e_{01}), \dd_{e_{01}})$.
	Define the operators
	\begin{align*}
		&\m^{(i)}_{n,\beta} := \alpha (\mathbf m^{(i)}_{n, \beta}) : X_i^{\otimes  n} \to X_i   \\
		&\n_{n_0,n_1,\beta} := \alpha (\mathbf n_{n_0,n_1,\beta}) : X_0^{\otimes n_0} \otimes X_{01} \otimes X_1^{\otimes  n_1} \to X_{01}  
	\end{align*}
	where $(n,\beta)\neq (0,\theta_{e_i}), (1,\theta_{e_i})$ and $(n_0,n_1,\beta)\neq (0,0,\theta_{e_{01}})$ by construction.
	We also define
	\begin{equation}
		\label{n_00_eq}
		\m_{1,\theta_{e_i}}^{(i)} = \dd_i  \ , \qquad \n_{0,0,\theta_{e_{01}}}= \dd_{01}
	\end{equation}
	Note that $|\m^{(i)}|=|\n|=1$.
	By \eqref{End_X_hat_composition}, \eqref{alpha_u_composition}, and \eqref{alpha_u_delta_eq}, we can first show that the operators $\m_{n,\beta}^{(i)}$ satisfy the $A_\infty$ relations in the same way as (\ref{A_inf_alg_relation}).
	Besides, for $x_1,\dots, x_{n_0}\in X_0$, $z_1,\dots, z_{n_1} \in X_1$, and $y\in X_{01}$, we have
	\begin{align*}
		\alpha (\delta \mathbf n_{n_0,n_1,\beta}) (x_1,\dots, x_{n_0} , y, z_1,\dots, z_{n_1}) \\
		= \dd_{01} \circ \n_{n_0,n_1,\beta} (x_1,\dots, z_{n_1}) & + \sum (-1)^{\sum_{s=1}^{i-1}|x_s|} \ \n_{n_0,n_1,\beta} (x_1,\dots, \mathsf d_0 x_i,\dots, x_{n_0}, y,\dots, z_{n_1}) \\
		& + \sum (-1)^{\sum_{s=1}^{n_0}|x_s|}  \  \n_{n_0,n_1,\beta} (x_1,\dots, x_{n_0}, \dd_{01} y , z_1,\dots, z_{n_1}) \\
		& + \sum (-1)^{\sum_{s=1}^{n_0}|x_s|+\sum_{t=1}^{j-1}|z_t|} \  \n_{n_0,n_1,\beta}(x_1,\dots, x_{n_0}, y, z_1,\dots, \mathsf d_1 z_j, \dots, z_{n_1})
	\end{align*}
	and
	\begin{align*}
		& -  \alpha (\delta \mathbf n_{n_0,n_1,\beta}) (x_1,\dots, x_{n_0} , y, z_1,\dots, z_{n_1})  \\
		= \
		& \sum_{\substack{r_0+s_0+t_0=n_0 \\ \beta'+\beta''=\beta \\ (s_0,\beta'')\neq (1,\theta_{e_0})}} \ (-1)^{\sum_{s=1}^{r_0}|x_s|} \ \n_{r_0+1+t_0, n_1,\beta'} 
		(x_1,\dots, x_{r_0} ,  \m^{(0)}_{s_0,\beta''} (\dots , x_{r_0+s_0}), \dots, x_{n_0}, y, z_1,\dots, z_{n_1})
		\\
		+ \ & \sum_{\substack{r_1+s_1+t_1=n_1 \\ \beta'+\beta''=\beta \\ (s_1,\beta'')\neq (1,\theta_{e_1})}} \ (-1)^{\sum_{s=1}^{n_0}|x_s|+|y|+\sum_{t=1}^{r_1}|z_t|} \  \n_{n_0, r_1+1+t_1,\beta'} (x_1,\dots, x_{n_0}, y, z_1,\dots, z_{r_1} ,  \m^{(1)}_{s_1,\beta''} (\dots , z_{r_1+s_1}) ,\dots, z_{n_1})
		\\
		+ \ & \sum_{\substack{n_0'+n_0''=n_0 \\ n_1'+n_1''=n_1 \\ \beta'+\beta''=\beta \\ (n_0',n_1',\beta')\neq (0,0,\theta_{e_{01}}) \\ (n_0'',n_1'',\theta_{e_{01}}) \neq (0,0,\theta_{e_{01}})}} 
		\ (-1)^{\sum_{s=1}^{n_0'}|x_s|} \  \n_{n_0',n_1',\beta'} (x_1,\dots,   \n_{n_0'',n_1'',\beta''} (x_{n_0'+1} , \dots, x_{n_0}, y, z_1, \dots, z_{n_1''}) ,\dots, z_{n_1} )
	\end{align*}
	where the signs are derived from \eqref{End_X_hat_composition} and \eqref{alpha_u_composition}.
	The above two computations together with (\ref{n_00_eq}) imply the $A_\infty$ bimodule equation:
	\begin{align*}
		& \sum_{\substack{r_0+s_0+t_0=n_0 \\ \beta'+\beta''=\beta}} \ (-1)^{\sum_{s=1}^{r_0}|x_s|} \ \n_{r_0+1+t_0, n_1,\beta'} 
		(x_1,\dots, x_{r_0} ,  \m^{(0)}_{s_0,\beta''} (\dots , x_{r_0+s_0}), \dots, x_{n_0}, y, z_1,\dots, z_{n_1})
		\\
		& + \  \sum_{\substack{r_1+s_1+t_1=n_1 \\ \beta'+\beta''=\beta }} \ (-1)^{\sum_{s=1}^{n_0}|x_s|+|y|+\sum_{t=1}^{r_1}|z_t|} \  \n_{n_0, r_1+1+t_1,\beta'} (x_1,\dots, x_{n_0}, y, z_1,\dots, z_{r_1} ,  \m^{(1)}_{s_1,\beta''} (\dots , z_{r_1+s_1}) ,\dots, z_{n_1})
		\\
		& + \  \sum_{\substack{n_0'+n_0''=n_0 \\ n_1'+n_1''=n_1 \\ \beta'+\beta''=\beta }} 
		\ (-1)^{\sum_{s=1}^{n_0'}|x_s|} \  \n_{n_0',n_1',\beta'} (x_1,\dots,   \n_{n_0'',n_1'',\beta''} (x_{n_0'+1} , \dots, x_{n_0}, y, z_1, \dots, z_{n_1''}) ,\dots, z_{n_1} )  \ \  = \ \  0
	\end{align*}

In particular, let \(\mathcal A_{\infty,\mathrm{2\mbox{-}mod}}\) denote the dg $\mathbf{fc}$-multicategory in \eqref{A_inf_2_bimodule_eq} in the special case where $\mathbb S=\pmb 0_E$ is trivial as in Example~\ref{ex_trivial_labeling} and $E$ is the above graph with three edges. Then the operations are \(\m^{(i)}_n\) for $i=0,1$ and \(\n_{n_0,n_1}\) with no additional $\beta$-labels. The discussion above has produced the following result:
	
\begin{thm}
\label{standard_A_inf_bim_thm}
The notion of an $A_\infty$ bimodule (see e.g. \cite{keller2005infinity}) is equivalent to an algebra over the dg $\mathbf{fc}$-multicategory $\mathcal A_{\infty, \text{2-mod}}$.
\end{thm}

\begin{rmk}
	To get some intuition for the extra labeling from $\mathbb S$, one can look at Example \ref{ex_S_E_labeling}. We also give an extra practical case below. Let $L_0,L_1$ be connected, embedded, compact Lagrangian submanifolds intersecting transversely inside a symplectic manifold $(X,\omega)$.
	Assume that $[\omega]\in H^2(X;\mathbb Z)$. For $i=0,1$,
	\[
	S_i \;:=\; \bigl\{\, \omega(\beta)\ \big|\ \beta\in \pi_2(X,L_i)\,\bigr\}
	\;\subset\; \mathbb R 
	\]
	is a discrete subgroup of $\mathbb R$.
	Let $\overline{\mathbb R}=\mathbb R\cup\{-\infty,+\infty\}$, and consider continuous maps
	$v:[0,1]\times \overline{\mathbb R}\to X$ 
	such that
	$v(\{i\}\times \overline{\mathbb R})\subset L_i$ and $
	v\bigl([0,1]\times\{\pm\infty\}\bigr)\in L_0\cap L_1$.
	Let $\pi_2(L_0,L_1)$ denote the set of homotopy classes of such maps relative to these boundary
	and endpoint conditions. For $\alpha\in \pi_2(L_0,L_1)$ represented by $v$, we set
	$
	\omega(\alpha) = \int_{[0,1]\times \mathbb R} v^*\omega,
	$
	which is well-defined since $\omega|_{L_0}=\omega|_{L_1}=0$.
	Then,
	\[
	S_{01}\;:=\;\bigl\{\,\omega(\alpha)\ \big|\ \alpha\in \pi_2(L_0,L_1)\,\bigr\}\subset \mathbb R.
	\]
	is also a discrete subset of $\mathbb R$. Now we define a vertically discrete $\mathbf{fc}$-multicategory $\mathbb S=(V,E,\mathbb S)$ as follows. The underlying directed graph $(V,E)$ is as above, and the only nontrivial collections of 2-cells are
	$\mathbb S\bigl((e_0)^n; e_0\bigr)=S_0$, $
	\mathbb S\bigl((e_1)^n; e_1\bigr)=S_1$, and $\mathbb S\bigl((e_0)^{n_0},\, e_{01},\, (e_1)^{n_1};\, e_{01}\bigr)=S_{01}$.
	This choice of labeling $\mathbb S$ is basically the one used for the notion of a \textit{filtered $A_\infty$ bimodule} in the symplectic literature; see \cite[Definition 3.7.5]{FOOOBookOne}. However, in principle one can make other choices of $\mathbb S$, such as in \eqref{S_iota_eq}, in order to keep track of additional information in practice. For instance, we can replace the above $S_0,S_1,S_{01}$ by $\pi_2(X,L_0)$, $\pi_2(X,L_1)$ and $\pi_2(L_0,L_1)$ respectively so that we get another example of $\mathbb S$.
\end{rmk}

\subsection{$A_\infty$ bimodules over a pair of $A_\infty$ categories}
\label{s_Ainf_bimodule_cat}

One may find a natural extension of Section \ref{s_Ainf_bimodule_alg}.
Let $V$ be a set. Consider a partition $V=V'\sqcup V''$, and define 
\[
E=(V'\times V') \sqcup (V'\times V'') \sqcup (V''\times V'')
\]
Then, $E=(V,E)$ is a directed graph as illustrated below.

\begin{equation*}
\begin{tikzpicture}[x=0.75pt,y=0.75pt,yscale=-0.5,xscale=0.5]
	\draw  [fill={rgb, 255:red, 74; green, 144; blue, 226 }  ,fill opacity=1 ] (112,121) .. controls (112,118.24) and (114.24,116) .. (117,116) .. controls (119.76,116) and (122,118.24) .. (122,121) .. controls (122,123.76) and (119.76,126) .. (117,126) .. controls (114.24,126) and (112,123.76) .. (112,121) -- cycle ;
	\draw  [fill={rgb, 255:red, 74; green, 144; blue, 226 }  ,fill opacity=1 ] (169,270) .. controls (169,267.24) and (171.24,265) .. (174,265) .. controls (176.76,265) and (179,267.24) .. (179,270) .. controls (179,272.76) and (176.76,275) .. (174,275) .. controls (171.24,275) and (169,272.76) .. (169,270) -- cycle ;
	\draw    (112,121) .. controls (22,42) and (215,45) .. (122,121) ;
	\draw [shift={(124.17,63.34)}, rotate = 184.36] [color={rgb, 255:red, 0; green, 0; blue, 0 }  ][line width=0.75]    (10.93,-3.29) .. controls (6.95,-1.4) and (3.31,-0.3) .. (0,0) .. controls (3.31,0.3) and (6.95,1.4) .. (10.93,3.29)   ;
	\draw    (169,270) .. controls (80,349) and (272,344) .. (179,270) ;
	\draw [shift={(181.51,326.86)}, rotate = 175.09] [color={rgb, 255:red, 0; green, 0; blue, 0 }  ][line width=0.75]    (10.93,-3.29) .. controls (6.95,-1.4) and (3.31,-0.3) .. (0,0) .. controls (3.31,0.3) and (6.95,1.4) .. (10.93,3.29)   ;
	\draw  [fill={rgb, 255:red, 74; green, 144; blue, 226 }  ,fill opacity=1 ] (250,268) .. controls (250,265.24) and (252.24,263) .. (255,263) .. controls (257.76,263) and (260,265.24) .. (260,268) .. controls (260,270.76) and (257.76,273) .. (255,273) .. controls (252.24,273) and (250,270.76) .. (250,268) -- cycle ;
	\draw    (250,268) .. controls (161,347) and (353,342) .. (260,268) ;
	\draw [shift={(262.51,324.86)}, rotate = 175.09] [color={rgb, 255:red, 0; green, 0; blue, 0 }  ][line width=0.75]    (10.93,-3.29) .. controls (6.95,-1.4) and (3.31,-0.3) .. (0,0) .. controls (3.31,0.3) and (6.95,1.4) .. (10.93,3.29)   ;
	\draw  [fill={rgb, 255:red, 74; green, 144; blue, 226 }  ,fill opacity=1 ] (172,90) .. controls (172,87.24) and (174.24,85) .. (177,85) .. controls (179.76,85) and (182,87.24) .. (182,90) .. controls (182,92.76) and (179.76,95) .. (177,95) .. controls (174.24,95) and (172,92.76) .. (172,90) -- cycle ;
	\draw    (172,90) .. controls (82,11) and (275,14) .. (182,90) ;
	\draw [shift={(184.17,32.34)}, rotate = 184.36] [color={rgb, 255:red, 0; green, 0; blue, 0 }  ][line width=0.75]    (10.93,-3.29) .. controls (6.95,-1.4) and (3.31,-0.3) .. (0,0) .. controls (3.31,0.3) and (6.95,1.4) .. (10.93,3.29)   ;
	\draw  [fill={rgb, 255:red, 74; green, 144; blue, 226 }  ,fill opacity=1 ] (243,89) .. controls (243,86.24) and (245.24,84) .. (248,84) .. controls (250.76,84) and (253,86.24) .. (253,89) .. controls (253,91.76) and (250.76,94) .. (248,94) .. controls (245.24,94) and (243,91.76) .. (243,89) -- cycle ;
	\draw    (243,89) .. controls (153,10) and (346,13) .. (253,89) ;
	\draw [shift={(255.17,31.34)}, rotate = 184.36] [color={rgb, 255:red, 0; green, 0; blue, 0 }  ][line width=0.75]    (10.93,-3.29) .. controls (6.95,-1.4) and (3.31,-0.3) .. (0,0) .. controls (3.31,0.3) and (6.95,1.4) .. (10.93,3.29)   ;
	\draw  [fill={rgb, 255:red, 74; green, 144; blue, 226 }  ,fill opacity=1 ] (305,122) .. controls (305,119.24) and (307.24,117) .. (310,117) .. controls (312.76,117) and (315,119.24) .. (315,122) .. controls (315,124.76) and (312.76,127) .. (310,127) .. controls (307.24,127) and (305,124.76) .. (305,122) -- cycle ;
	\draw    (305,122) .. controls (215,43) and (408,46) .. (315,122) ;
	\draw [shift={(317.17,64.34)}, rotate = 184.36] [color={rgb, 255:red, 0; green, 0; blue, 0 }  ][line width=0.75]    (10.93,-3.29) .. controls (6.95,-1.4) and (3.31,-0.3) .. (0,0) .. controls (3.31,0.3) and (6.95,1.4) .. (10.93,3.29)   ;
	\draw  [dash pattern={on 0.84pt off 2.51pt}] (73,38.2) .. controls (73,20.42) and (87.42,6) .. (105.2,6) -- (325.8,6) .. controls (343.58,6) and (358,20.42) .. (358,38.2) -- (358,134.8) .. controls (358,152.58) and (343.58,167) .. (325.8,167) -- (105.2,167) .. controls (87.42,167) and (73,152.58) .. (73,134.8) -- cycle ;
	\draw  [dash pattern={on 0.84pt off 2.51pt}] (98,262.8) .. controls (98,250.21) and (108.21,240) .. (120.8,240) -- (315.2,240) .. controls (327.79,240) and (338,250.21) .. (338,262.8) -- (338,331.2) .. controls (338,343.79) and (327.79,354) .. (315.2,354) -- (120.8,354) .. controls (108.21,354) and (98,343.79) .. (98,331.2) -- cycle ;
	\draw [color={rgb, 255:red, 126; green, 211; blue, 33 }  ,draw opacity=1 ]   (308.88,129.78) -- (255,263) ;
	\draw [shift={(310,127)}, rotate = 112.02] [fill={rgb, 255:red, 126; green, 211; blue, 33 }  ,fill opacity=1 ][line width=0.08]  [draw opacity=0] (8.93,-4.29) -- (0,0) -- (8.93,4.29) -- cycle    ;
	\draw    (117,121) -- (310,122) ;
	\draw    (117,121) -- (177,90) ;
	\draw    (117,121) -- (248,89) ;
	\draw    (177,90) -- (248,89) ;
	\draw    (248,89) -- (310,122) ;
	\draw    (177,90) -- (310,122) ;
	\draw [color={rgb, 255:red, 126; green, 211; blue, 33 }  ,draw opacity=1 ]   (248.12,97) -- (255,263) ;
	\draw [shift={(248,94)}, rotate = 87.63] [fill={rgb, 255:red, 126; green, 211; blue, 33 }  ,fill opacity=1 ][line width=0.08]  [draw opacity=0] (8.93,-4.29) -- (0,0) -- (8.93,4.29) -- cycle    ;
	\draw [color={rgb, 255:red, 126; green, 211; blue, 33 }  ,draw opacity=1 ]   (178.26,97.72) -- (255,263) ;
	\draw [shift={(177,95)}, rotate = 65.1] [fill={rgb, 255:red, 126; green, 211; blue, 33 }  ,fill opacity=1 ][line width=0.08]  [draw opacity=0] (8.93,-4.29) -- (0,0) -- (8.93,4.29) -- cycle    ;
	\draw [color={rgb, 255:red, 126; green, 211; blue, 33 }  ,draw opacity=1 ]   (119.13,128.11) -- (255,263) ;
	\draw [shift={(117,126)}, rotate = 44.79] [fill={rgb, 255:red, 126; green, 211; blue, 33 }  ,fill opacity=1 ][line width=0.08]  [draw opacity=0] (8.93,-4.29) -- (0,0) -- (8.93,4.29) -- cycle    ;
	\draw [color={rgb, 255:red, 126; green, 211; blue, 33 }  ,draw opacity=1 ]   (113.23,129.74) -- (174,265) ;
	\draw [shift={(112,127)}, rotate = 65.81] [fill={rgb, 255:red, 126; green, 211; blue, 33 }  ,fill opacity=1 ][line width=0.08]  [draw opacity=0] (8.93,-4.29) -- (0,0) -- (8.93,4.29) -- cycle    ;
	\draw [color={rgb, 255:red, 126; green, 211; blue, 33 }  ,draw opacity=1 ]   (172.03,93) -- (174,265) ;
	\draw [shift={(172,90)}, rotate = 89.35] [fill={rgb, 255:red, 126; green, 211; blue, 33 }  ,fill opacity=1 ][line width=0.08]  [draw opacity=0] (8.93,-4.29) -- (0,0) -- (8.93,4.29) -- cycle    ;
	\draw [color={rgb, 255:red, 126; green, 211; blue, 33 }  ,draw opacity=1 ]   (241.91,91.79) -- (174,265) ;
	\draw [shift={(243,89)}, rotate = 111.41] [fill={rgb, 255:red, 126; green, 211; blue, 33 }  ,fill opacity=1 ][line width=0.08]  [draw opacity=0] (8.93,-4.29) -- (0,0) -- (8.93,4.29) -- cycle    ;
	\draw [color={rgb, 255:red, 126; green, 211; blue, 33 }  ,draw opacity=1 ]   (302.97,124.21) -- (174,265) ;
	\draw [shift={(305,122)}, rotate = 132.49] [fill={rgb, 255:red, 126; green, 211; blue, 33 }  ,fill opacity=1 ][line width=0.08]  [draw opacity=0] (8.93,-4.29) -- (0,0) -- (8.93,4.29) -- cycle    ;
	\draw    (174,270) -- (255,268) ;
	
	\draw (86,19.4) node [anchor=north west][inner sep=0.75pt]    {$V''$};
	\draw (105,259.4) node [anchor=north west][inner sep=0.75pt]    {$V'$};
\end{tikzpicture}
\end{equation*}

In the special case $V=\{v_0,v_1\}$ is a two-element set, this recovers the directed graph $\{e_0,e_1,e_{01}\}$ in Section \ref{s_Ainf_bimodule_alg} (cf. Figure \ref{fig:A_inf_bimodule}).
Let's take $\mathbb S=\pmb 0_E$ to be the trivial one for simplicity.
Due to \eqref{A_inf_V_cat_eq}, we can form dg $\mathbf{fc}$-multicategories $\mathcal A_{\infty, V'}$ and $\mathcal A_{\infty, V''}$. By Definition \ref{generalized_Ainf_defn}, we can form a dg $\mathbf{fc}$-multicategory $\mathcal A_{\infty, E}$.
Then, by an argument nearly identical to that of Section~\ref{s_Ainf_bimodule_alg}, giving an $\mathcal A_{\infty, E}$-algebra is equivalent to specifying an $A_\infty$ category with object set $V'$, an $A_\infty$ category with object set $V''$, and an $A_\infty$ bimodule over these two $A_\infty$ categories in the usual sense (see, for instance, \cite[Definition~2.12]{Ganatra_thesis}).

\subsection{$A_\infty$ multi-modules}
In view of the discussion around \eqref{partition_V_to_r_sets}, it is natural to introduce the following construction. Let V be a set and let $r\ge 1$ be an integer. Fix a partition $V=\bigsqcup_{i=1}^r V^{(i)}$ and define
\[E  := \bigcup_{j\le k} \, V^{(j)}\times V^{(k)}  
\]
Let $\mathcal A_{\infty,E}$ be the dg $\mathbf{fc}$-multicategory associated to E as in Definition~\ref{generalized_Ainf_defn}. We then propose to define an $A_\infty$ $r$-module (relative to the above partition) to be an algebra over $\mathcal A_{\infty,E}$. When $r=2$, this exactly recovers the above notion of $A_\infty$ bimodules.

	\bibliographystyle{abbrv}
	
	\bibliography{mybib_cohomology}	
	
\end{document}